\documentclass{amsart}
\usepackage[foot]{amsaddr}
\usepackage{graphicx} 
\usepackage{amsmath}
\usepackage{amsfonts}
\usepackage[
backend=biber,
style=alphabetic,
]{biblatex}
\usepackage{xcolor}
\usepackage{quiver}
\usepackage{enumerate}
\usepackage{mathrsfs}
\usepackage[hypertexnames=false]{hyperref}
\usepackage{amsthm}
\usepackage{thmtools}
\usepackage[nameinlink,noabbrev,capitalize]{cleveref}
\usepackage[left=1in, right=1 in, top=1in]{geometry}
\usepackage{mathtools}
\usepackage{tikz}

\theoremstyle{plain}
\newtheorem{thm}{Theorem}[subsection]
\newtheorem{cor}[thm]{Corollary}
\newtheorem{prop}[thm]{Proposition}
\newtheorem{lem}[thm]{Lemma}

\theoremstyle{definition}
\newtheorem{defn}[thm]{Definition}

\newtheorem{exmp}[thm]{Example}

\theoremstyle{remark}
\newtheorem{rem}[thm]{Remark}

\newtheorem{warn}[thm]{Warning}

\bibliography{refs}

\newcommand{\X}{\mathcal{X}}
\newcommand{\Gm}{\mathbb{G}_m}
\newcommand{\gm}{\Gm}
\newcommand{\EE}{\mathcal{E}}
\newcommand{\OO}{\mathcal{O}}

\newcommand{\ul}{\underline}
\newcommand{\wtag}{\widetilde{\ag}}
\newcommand{\wtpg}{\widetilde{\pg}}

\DeclareMathOperator{\op}{op}
\newcommand{\ts}{\textsuperscript}

\newcommand{\lam}{\lambda}
\newcommand{\A}{\mathbb{A}}
\newcommand{\Z}{\mathbb{Z}}
\newcommand{\C}{\mathbb{C}}
\newcommand{\mc}{\mathcal}
\newcommand{\mb}{\mathbb}

\DeclareMathOperator{\Stack}{Stack}
\DeclareMathOperator{\Tot}{Tot}
\DeclareMathOperator{\Hom}{Hom}
\DeclareMathOperator{\HHom}{{\mathcal{H}\kern -.5pt om}}
\DeclareMathOperator{\Mod}{Mod}
\DeclareMathOperator{\St}{St}
\DeclareMathOperator{\st}{st}
\DeclareMathOperator{\sti}{st_i}
\DeclareMathOperator{\Stc}{St^c}
\DeclareMathOperator{\QC}{QC}
\DeclareMathOperator{\Spec}{Spec}
\newcommand{\Coh}{\mathrm{Coh}}
\DeclareMathOperator{\IndCoh}{IndCoh}

\DeclareMathOperator{\Perft}{2Perf}
\DeclareMathOperator{\Pervt}{2Perv}
\DeclareMathOperator{\Coht}{2Coh}
\DeclareMathOperator{\Perf}{Perf}
\DeclareMathOperator{\pt}{pt}
\DeclareMathOperator{\id}{id}
\DeclareMathOperator{\Sph}{Sph}
\DeclareMathOperator{\Ind}{Ind}

\DeclareMathOperator{\colim}{colim}
\DeclareMathOperator{\cone}{cone}
\DeclareMathOperator{\fib}{fib}
\DeclareMathOperator{\cofib}{cofib}
\DeclareMathOperator{\Fun}{Fun}
\DeclareMathOperator{\Fuk}{Fuk}
\DeclareMathOperator{\lcm}{lcm}

\DeclareMathOperator{\oplaxlim}{oplaxlim}

\DeclareMathOperator{\oplaxcolim}{oplaxcolim}
\DeclareMathOperator{\Sing}{Sing}
\DeclareMathOperator{\supp}{supp}

\DeclareMathOperator{\Vect}{Vect}
\newcommand{\ag}{\A^1/\Gm}
\newcommand{\pg}{\pt/\Gm}
\newcommand{\age}{\overline{\A^1/\Gm}}
\newcommand{\pge}{\overline{\pt/\Gm}}

\newcommand{\Am}{\mc{A}}

\newcommand{\xto}[1]{\xrightarrow{#1}}
\newcommand{\xfrom}[1]{\xleftarrow{#1}}
\newcommand{\onto}{\twoheadrightarrow}
\newcommand{\inclto}{\hookrightarrow}
\newcommand{\adjto}{\dashv}

\newcommand{\tofrom}{\rightleftarrows}
\newcommand{\isomfrom}{\xleftarrow{\sim}}
\newcommand{\isomto}{\xrightarrow{\sim}}

\newcommand{\comment}[1]{}

\newcommand{\diskmark}{%
    \tikz[baseline={([yshift=-.5ex]current bounding box.center)}, scale=0.35]{
        \draw (0,0) circle (1);
        \fill (70:1) arc (70:110:1) -- (110:0.75) arc (110:70:0.75) -- cycle;
        \fill (250:1) arc (250:290:1) -- (290:0.75) arc (290:250:0.75) -- cycle;
        \fill (-20:1) arc (-20:20:1) -- (20:0.75) arc (20:-20:0.75) -- cycle;
        \fill (160:1) arc (160:200:1) -- (200:0.75) arc (200:160:0.75) -- cycle;
    }%
}

\newcommand{\crossdash}{%
    \tikz[baseline={([yshift=-.5ex]current bounding box.center)}, scale=0.35]{
        \draw[dash pattern=on 1.5pt off 1.5pt] (-0.5,1) -- (-0.5,0.5) -- (-1,0.5);
        \draw[dash pattern=on 1.5pt off 1.5pt] (-1,-0.5) -- (-0.5,-0.5) -- (-0.5,-1);
        \draw[dash pattern=on 1.5pt off 1.5pt] (0.5,-1) -- (0.5,-0.5) -- (1,-0.5);
        \draw[dash pattern=on 1.5pt off 1.5pt] (1,0.5) -- (0.5,0.5) -- (0.5,1);
        
        \draw (-1,0.5) -- (-1,-0.5);
        \draw (-0.5,-1) -- (0.5,-1);
        \draw (1,-0.5) -- (1,0.5);
        \draw (0.5,1) -- (-0.5,1);
    }%
}

\newcommand{\crossskel}{%
    \tikz[baseline={([yshift=-.5ex]current bounding box.center)}, scale=0.25]{
        \draw[thick] (0,-1) -- (0,1);
        \draw[thick] (-1,0) -- (1,0);
    }%
}

\newcommand{\cornerBL}{\tikz[baseline={([yshift=-.5ex]current bounding box.center)}, scale=0.25]{ \draw[thick] (0,1) -- (0,0) -- (1,0); }}
\newcommand{\cornerBR}{\tikz[baseline={([yshift=-.5ex]current bounding box.center)}, scale=0.25]{ \draw[thick] (-1,0) -- (0,0) -- (0,1); }}
\newcommand{\cornerTR}{\tikz[baseline={([yshift=-.5ex]current bounding box.center)}, scale=0.25]{ \draw[thick] (-1,0) -- (0,0) -- (0,-1); }}
\newcommand{\cornerTL}{\tikz[baseline={([yshift=-.5ex]current bounding box.center)}, scale=0.25]{ \draw[thick] (1,0) -- (0,0) -- (0,-1); }}

\newcommand{\diskmarkphi}{%
    \tikz[baseline={([yshift=-.5ex]current bounding box.center)}, scale=0.35]{
        \draw (0,0) circle (1);
        \fill (70:1) arc (70:110:1) -- (110:0.75) arc (110:70:0.75) -- cycle;
        \fill (250:1) arc (250:290:1) -- (290:0.75) arc (290:250:0.75) -- cycle;
        \fill (-20:1) arc (-20:20:1) -- (20:0.75) arc (20:-20:0.75) -- cycle;
        \fill (160:1) arc (160:200:1) -- (200:0.75) arc (200:160:0.75) -- cycle;
        \fill[blue] (0,0) circle (0.15); 
    }%
}

\newcommand{\crossdashphi}{%
    \tikz[baseline={([yshift=-.5ex]current bounding box.center)}, scale=0.35]{
        \draw[dash pattern=on 1.5pt off 1.5pt] (-0.5,1) -- (-0.5,0.5) -- (-1,0.5);
        \draw[dash pattern=on 1.5pt off 1.5pt] (-1,-0.5) -- (-0.5,-0.5) -- (-0.5,-1);
        \draw[dash pattern=on 1.5pt off 1.5pt] (0.5,-1) -- (0.5,-0.5) -- (1,-0.5);
        \draw[dash pattern=on 1.5pt off 1.5pt] (1,0.5) -- (0.5,0.5) -- (0.5,1);
        
        \draw (-1,0.5) -- (-1,-0.5);
        \draw (-0.5,-1) -- (0.5,-1);
        \draw (1,-0.5) -- (1,0.5);
        \draw (0.5,1) -- (-0.5,1);
        
        \fill[blue] (0,0) circle (0.15); 
    }%
}

\newcommand{\crossskelphi}{%
    \tikz[baseline={([yshift=-.5ex]current bounding box.center)}, scale=0.25]{
        \draw[thick] (0,-1) -- (0,1);
        \draw[thick] (-1,0) -- (1,0);
        \fill[blue] (0,0) circle (0.2); 
    }%
}

\newcommand{\dotlineR}{%
    \tikz[baseline={([yshift=-.5ex]current bounding box.center)}, scale=0.25]{
        \draw[thick] (0,0) -- (1,0);
        \fill[blue] (0,0) circle (0.2);
    }%
}
\newcommand{\dotlineU}{%
    \tikz[baseline={([yshift=-.5ex]current bounding box.center)}, scale=0.25]{
        \draw[thick] (0,0) -- (0,1);
        \fill[blue] (0,0) circle (0.2);
    }%
}

\definecolor{fibregrey}{RGB}{180, 180, 180}
\definecolor{fibrepurple}{RGB}{128, 0, 128}

\title{GIT for root stacks and 3d mirror symmetry}
\author{Swapnil Garg}
\email{swapnilg@berkeley.edu}
\author{Ruoxi Li}
\email{ruoxi\_li@berkeley.edu}
\author{Yuji Okitani}
\email{yuji\_okitani@berkeley.edu}
\date{August 2026}

\begin{document}

\begin{abstract}
Using window theory, Bodzenta--Donovan showed that the derived category of a root stack $\sqrt[n]{X/D}$ has a $2n$-periodic $2$-term semiorthogonal decomposition. We define a categorical generalization of the root stack construction and interpret this as a pullback on the B-side of the 3d mirror symmetry equivalence of Gammage--Hilburn--Mazel-Gee. We analyze this pullback using categorical representation theoretic results of Ben-Zvi--Francis--Nadler and Ben-Zvi--Nadler--Preygel applied to SODs. We also show that the pullback is equivalent to a pushforward of perverse schobers on the A-side, allowing us to deduce periodicity from a simple decomposition of an A-side Lagrangian skeleton. Additionally, we adapt the construction of Bodzenta--Donovan to a new GIT problem, which yields an embedding of $\Coh(\sqrt[m]{X/D})$ into $\Coh(\sqrt[n]{X/D})$ for $m<n$ coprime, and prove $2n$-periodicity of the resulting $2$-term SOD.
\end{abstract}
\maketitle

\tableofcontents

\section{Introduction}
This paper explores the geometry and algebra of the root stack construction, both extending results of \cite{BD} and generalizing the construction to fit within the 3d mirror symmetry equivalence of \cite{3dms}. The paper is roughly split into two parts following these two objectives. The first part comprises \cref{sect:geometryrootstack} through \cref{sect:multiwindow}, and the second part comprises \cref{sect:categorical} through \cref{sect:schob} (along with the Appendix), though \cref{sect:categorical} can be thought of as a transitional chapter.

The main result \cref{thm:main} is an equivalence (at the object level) of B-side pullback and A-side pushforward endofunctors defined on the 3d mirror symmetry equivalence of \cite{3dms}, namely the equivalence between perverse schobers on $\A^1$ and coherent sheaves of categories on $\ag$ (see \cref{subsect:3dmsintro}). The B-side pullback is a categorical version of the root stack functor, and we study it in the framework of ''categorical representation theory of semiorthogonal decompositions'', using results of \cite{BZFN} and \cite{BZNP}. We also give a new GIT realization of the root stack, generalizing the construction in \cite{BD}.

Root stacks are a particularly tractable class of Deligne--Mumford stacks. They arise naturally from a pair $(X, D)$ of a scheme with effective Cartier divisor, allowing one to study how geometric and categorical phenomena on schemes extend to the stacky setting. In this sense they serve as a useful testing ground for questions about derived categories of stacks. In particular, they can be referred to as a ``stacky blowup in codimension 1'', due to the similarity of their semiorthogonal decompositions to that of Orlov's blowup formula \cite{orlovblowup}.

Furthermore, \cite{BD} showed that the above SOD of the derived category of a root stack is periodic under mutation. It is known that $2n$-periodic SODs correspond to $n$-spherical functors as defined and shown by \cite{nspherical}, generalizing the equivalence of $4$-periodic SODs and usual spherical functors of \cite{window2}. Understanding the abstract nature of this periodicity was one of the starting points of the present work. In fact, we show that the root stack construction is abstractly equivalent to the relative Waldhausen S-construction of \cite{DKSS}.

Another motivation comes from 3d mirror symmetry. The root stack story gives a natural endofunctor on categories over $\ag$ (a B-side phenomenon), and the existence of a periodic SOD suggests a mirror-symmetric explanation. Part of our goal in understanding periodicity was to identify the mirror side behavior and explain periodicity geometrically on the A-side.

We now give a summary of the rest of the introduction, which will go into more detail about our results.

In \cref{subsect:vgitintro}, we give a geometric construction that generalizes the construction and results of \cite{BD}. In brief: the root stack $\sqrt[n]{X/D}$ is defined as the pullback of $X \to \ag$ via the $n$\textsuperscript{th} power map $\ag \to \ag$ (hereafter denoted $e_n$). There is also a construction of $\sqrt[n]{X/D}$ as the $\Gm$-stack quotient of a specific scheme, which \cite{BD} noticed can be realized as a certain semistable GIT quotient. We show that this construction comes from the pullback (along $X \to \ag$) of a specific map $\A^2/\Gm^2 \to \ag$, thus situating the latter construction within the original root stack definition. This perspective naturally leads to a generalization by taking other possible maps $\A^2/\Gm^2 \to \ag$. The GIT realization by \cite{BD} allowed the use of the window theory of \cite{window}, giving a natural proof of mutation periodicity in the derived category of a root stack. The same methods give a proof of periodicity in our generalized construction (i.e. periodicity of a different semi-orthogonal decomposition of the same category).

The nature of our construction suggests the possibility of a ``universal'' proof of periodicity, i.e. one that follows from the case $X=\ag$. Indeed, as part of \cref{sect:multiwindow} we show that one can compute semistable loci (and thus windows) for root stack-type constructions universally (i.e. pulled back from $\A^k/\Gm^k$). Meanwhile, one can avoid windows by using categorical representation theory and simple results about the derived category of $\ag$, explained in \cref{sect:categorical}. The latter approach involves viewing $\Perf(X)$ as a module over the monoidal category $(\Perf(\ag), \otimes)$, and is explained in more detail in \cref{subsect:univintro}.


In \cref{subsect:3dmsintro} we explain how the root stack construction appears in 3d mirror symmetry. In brief: the $2$-category $\Perf(\ag)-\Mod$ is also $\Perft(\ag)$, which naturally occurs on the B-side of 3d mirror symmetry. But in this context, it is more natural to analyze the $2$-category $\Coht(\ag)$, which (up to a choice of embedding) contains $\Perft(\ag)$ as a sub-2-category. We extend the root stack operation (which can be viewed as an endofunctor of $\Perft(\ag)$) to an endofunctor of $\Coht(\ag)$; this functor can be thought of as a pullback on the B-side of 3d mirror symmetry. The equivalent operation on the A-side can be viewed as a pushforward of perverse schobers. 3d mirror symmetry also gives an A-side explanation for periodicity that follows from geometry of Lagrangian skeleta.

\subsection{Variation of GIT}
\label{subsect:vgitintro}
The paper \cite{window2} showed that for a stack quotient $X/G$, \textit{balanced} wall crossings in GIT lead to autoequivalences $\phi_w \colon \Coh(X^{-, ss}/G) \xto{\sim} \Coh(X^{+, ss}/G)$, called window-shift autoequivalences. These arise from the window theory of \cite{window}, by viewing the same category $\mc{C}$ in terms of two different stratifications of $X/G$ (coming from two different semistable loci), yielding the equalities $$\langle \Coh(X^{-, ss}/G), \Coh(Z/L) \rangle=\langle \Coh(Z/L)', \Coh(X^{-, ss}/G) \rangle$$ and $$\langle \Coh(Z/L), \Coh(X^{+, ss}/G) \rangle=\langle \Coh(X^{+, ss}/G) , \Coh(Z/L)' \rangle.$$ The window-shift autoequivalences thus arise from a certain 4-periodic semiorthogonal decomposition on $\mc{C}$. It was also shown in \cite{window2} that 4-periodic decompositions of $\langle \mc{A}, \mc{B} \rangle$ correspond to spherical functors $\mc{A} \to \mc{B}$. 

One could imagine a context where, rather than taking an ambient category $\mc{C}$ containing copies of both $\Coh(X^{-, ss}/G)$ and $\Coh(X^{+, ss}/G)$, one takes an ambient category $\mc{C}\simeq \Coh(X^{+, ss}/G)$ which contains a copy of $\Coh(X^{-, ss}/G)$. This context would correspond to a non-balanced wall crossing. Indeed, the paper \cite{BD} showed that for a specific GIT problem explained below (which has a non-balanced wall crossing), this method demonstrates periodicity of a semiorthogonal decomposition of the ambient $\mc{C}$.

Specifically, they showed that the semiorthogonal decomposition of the derived category of a root stack is periodic, by viewing the root construction through the lens of window theory and variation of GIT. A scheme $X$ with an effective Cartier divisor $D$ (along with its canonical section) determines a map $X \to \ag$. The $n$\textsuperscript{th} \textit{root stack} $\sqrt[n]{X/D}$ is the pullback of this map along the map $e_n \colon \ag \to \ag$ which is induced by taking the $n$\textsuperscript{th} power of the coordinate. Then $\sqrt[n]{X/D}$ has a stratification with an open substack $X-D$ and a closed substack that is a $\mu_n$-gerbe over $D$.

Consider the (non-Cartesian) diagram 
\[\begin{tikzcd}
	{\sqrt{O_D(D)}} & {\sqrt[n]{X/D}} \\
	D & X
	\arrow["i", from=1-1, to=1-2]
	\arrow["q", from=1-1, to=2-1]
	\arrow["p"', from=1-2, to=2-2]
	\arrow["{i_D}", from=2-1, to=2-2]
\end{tikzcd}\]

where $\sqrt{O_D(D)}$ is the $\mu_n$-gerbe over $D$.
\begin{thm}\textup{\cite[Theorem 1.2]{BD}} \label{thm:theirSOD}
    There are SODs
    \begin{align*}
        \Coh(\sqrt[n]{X/D})=&\langle p^*\Coh(X), i_*q^*\Coh(D), i_*q^*\Coh(D)\langle 1 \rangle, \dots, i_*q^*\Coh(D)\langle n-2 \rangle \rangle\\
        =&\langle i_*q^*\Coh(D), i_*q^*\Coh(D)\langle 1 \rangle, \dots, i_*q^*\Coh(D)\langle n-2 \rangle, p^*\Coh(X)\langle n-1 \rangle \rangle
    \end{align*}
    with each $p^*$ and $i_*q^*\langle i \rangle$ fully faithful.

    If we let $\mc{D}$ be the subcategory $\langle i_*q^*\Coh(D), i_*q^*\Coh(D)\langle 1 \rangle, \dots, i_*q^*\Coh(D)\langle n-2 \rangle \rangle$, then the SOD $\Coh(\sqrt[n]{X/D})=\langle p^*\Coh(X), \mc{D} \rangle$ is $2n$-periodic.
\end{thm}

The proof of their result involves viewing $\sqrt[n]{X/D} \simeq \X^+/\Gm$ and $X \simeq \X^-/\Gm$ as the positive and negative semistable quotients of a scheme $\mc{X}=\{yz^n=s\} \subset \Tot(\mc{O}_X(D) \oplus \mc{O}_X)$ with a $\Gm$-action (using a construction of the root stack as a global quotient from \cite{abramovich2008gromov}). Here $y$ and $z$ are fiberwise coordinates, $s$ is the canonical section of $\mc{O}_X(D)$, and $\Gm$ acts fiberwise with weights $(-n, 1)$. Then window theory gives a category $\mc{C} \subset \Coh(\X)$ such that $$\mc{C}=\Coh(\sqrt[n]{X/D})$$ and \begin{align*}\mc{C}=&\langle p^*\Coh(X), i_*q^*\Coh(D), i_*q^*\Coh(D)\langle 1 \rangle, \dots, i_*q^*\Coh(D)\langle n-2 \rangle \rangle\\
        =&\langle i_*q^*\Coh(D), i_*q^*\Coh(D)\langle 1 \rangle, \dots, i_*q^*\Coh(D)\langle n-2 \rangle, p^*\Coh(X)\langle n-1 \rangle \rangle.\end{align*} The extra ingredient for periodicity is the fact that a twist $\langle n \rangle$ preserves the components in the decomposition.

In our paper we extend the techniques of \cite{BD} to find new examples of periodicity involving non-balanced wall crossings and the root stack construction. Specifically, given $m<n$ coprime with $am+bn=1$, we construct a scheme $X_{n,m;a,b}$ with a $\Gm$-action such that the positive and negative semistable quotients are $X_{n,m;a,b}^+/\Gm \simeq \sqrt[n]{X/D}$ and $X_{n,m;a,b}^-/\Gm \simeq \sqrt[m]{X/D}$, respectively. We then have a (non-Cartesian) diagram

\[\begin{tikzcd}
	{\sqrt{O_D(aD)}} & {X_{n,m;a,b}^+/\Gm} \\
	D & X.
	\arrow["i", from=1-1, to=1-2]
	\arrow["q", from=1-1, to=2-1]
	\arrow["p", from=1-2, to=2-2]
	\arrow["{i_D}", from=2-1, to=2-2]
\end{tikzcd}\]

Note that the case $(n,m;a,b)=(n,1;1,0)$ recovers the construction of \cite{BD}. Using a similar window-theoretic argument, we show the following result.

\begin{thm}\label{thm:mainresult}
    There are SODs
    \begin{align*}
        \Coh(\sqrt[n]{X/D})=&\langle \mc{C}, i_*q^*\Coh(D), i_*q^*\Coh(D)\langle 1 \rangle, \dots, i_*q^*\Coh(D)\langle n-m-1 \rangle \rangle\\
        =&\langle i_*q^*\Coh(D), i_*q^*\Coh(D)\langle 1 \rangle, \dots, i_*q^*\Coh(D)\langle n-m-1 \rangle, \mc{C}\langle n-m \rangle \rangle
    \end{align*}
    with each $i_*q^*\langle i \rangle$ fully faithful, and with $\mc{C}$ a category abstractly isomorphic to $\Coh(\sqrt[m]{X/D})$.

    If we let $\mc{D}$ be the subcategory $\langle i_*q^*\Coh(D), i_*q^*\Coh(D)\langle 1 \rangle, \dots, i_*q^*\Coh(D)\langle n-m-1 \rangle \rangle$, then the SOD $\Coh(\sqrt[n]{X/D})=\langle \mc{C}, \mc{D} \rangle$ is $2n$-periodic (as the components are stable under $\langle n \rangle$).
\end{thm}

Our construction gives a new way to think of the root stack $\sqrt[n]{X/D}$ as a global quotient stack that may be of independent interest. The scheme $\X^+$ is a $\mu_a$-torsor over our scheme $X^+_{n,m;a,b}$ with the $a$-to-$1$ cover
$$\X^+ \to X^+_{n,m;a,b}$$
yielding an equivalence of stacks $$\sqrt[n]{X/D} \simeq \X^+/\Gm \simeq X^+_{n,m;a,b}/\Gm .$$ So the components in our SOD match up with the SOD in \cite{BD} in a certain way. Specifically, our $i_*q^*\Coh(D)\langle i \rangle$ is the same subcategory as their $i_*q^*\Coh(D)\langle ia \rangle$.

In \cref{sect:multiwindow} we also explain how to extend our construction to iterated root stacks, as well as to a higher-dimensional parameter space of semistable quotients.

We also compute the gluing functors between the components of the SOD in \cite{BD} (they already computed the first one from $\Coh(X)$ to $\Coh(D)$ as $i_D^*$), showing that they are abstractly equivalent to the identity between adjacent copies of $\Coh(D)$, and zero otherwise. This computation suggests a few things about the abstract structure of this decomposition:
\begin{itemize}
    \item The decomposition is the $(n-1)$st level of the relative Waldhausen S-construction from \cite{DKSS} for the spherical functor $i_D^* \colon \Coh(X) \to \Coh(D)$. So it is known for abstract reasons to be $2n$-periodic (see \cite[Theorem 5.4.2]{nspherical}).
    \item Taking the right (or equivalently, left) orthogonal to any set of copies of $\Coh(D)$ in the SOD from \cite{BD} yields a category that is abstractly equivalent to $\Coh(\sqrt[m]{X/D})$. Our geometric construction, however, picks out a specific one. We do not know whether the other possible embeddings of $\Coh(\sqrt[m]{X/D})$ into $\Coh(\sqrt[n]{X/D})$ have a geometric interpretation.
\end{itemize}

\subsection{SODs and periodicity universally}
\label{subsect:univintro}
The SOD result of \cite{BD} was also stated (without periodicity) by \cite{bls} in the case that $X$ is an algebraic stack in characteristic zero and we take $\Perf$ instead of $\Coh$ (generalizing the smooth DM stack case of \cite{ishiiueda}). Still assuming characteristic zero, we prove all such results (as well as reprove our results from \cref{subsect:vgitintro}) in the case that $X$ is a perfect derived stack (for $\Perf$) or a locally finite type stack over $\ag$ (for $\Coh$) in \cref{sect:categorical}.

The main content is the following calculation, which we combine with base change and reconstruction results from the Appendix.
\begin{thm}[\cref{thm:a1gmsod}]
    We have an SOD $$\Coh(\A^1/\Gm)=\langle e_n^*\Coh(\A^1/\Gm), i_*h_n^*\Coh(\pt/\Gm), i_*h_n^*\Coh(\pt/\Gm)\langle 1 \rangle, \dots, i_*h_n^*\Coh(\pt/\Gm)\langle n-2\rangle \rangle.$$ This decomposition has the following properties:
    \begin{enumerate}
        \item Each component is stable under the action of $\Coh(\A^1/\Gm)$ (acting via $e_n^*$).
        \item The gluing functors are zero between nonadjacent components, and are otherwise $i^* \colon \Coh(\A^1/\Gm) \to \Coh(\pt/\Gm)$ and $\id \colon \Coh(\pt/\Gm) \to \Coh(\pt/\Gm)$, which are functors of $\Coh(\A^1/\Gm)$-modules.
    \end{enumerate}
\end{thm}

By viewing both $\Coh(X)$ and $\Perf(X)$ as modules over $\Coh(\ag)$, we show that the periodicity on $\Coh(\sqrt[n]{X/D})$ or $\Perf(\sqrt[n]{X/D})$ follows from periodicity of $\Coh(\ag)$ above by taking either the tensor product of the above SOD with $\Perf(X)$ or the Hom with $\Coh(X)$. The main input are the results of \cite{BZFN} and \cite{BZNP} (though we need to slightly generalize \cite{BZNP} to relative DM stacks since $\ag \to \ag$ is one; see \cref{subsect:bznpextend}). These results tell us that the resulting SOD is a decomposition of either $\Perf$ or $\Coh$ (depending on if we do $\otimes$ or $\Hom$) of the fiber product $\sqrt[n]{X/D}$.

To use these base change results, we need to be able to combine categorical representation theory with semi-orthogonal decompositions. Essentially, we need semi-orthogonal decompositions that are $\mc{A}$-stable for a monoidal stable category $\mc{A}$ to be preserved under $\otimes_{\mc{A}}$ and $\Hom_{\mc{A}}$. As explained in \cref{subsect:vgitintro}, we also calculate gluing functors, so we need results about reconstructing the category comprising an SOD from its gluing functors. The necessary technical details are given in Appendices A and B (mainly \cref{thm:preserve}).

Note that \cite[Section 1.4]{yuzhao} hypothesizes a way to do categorical representation theory for the SOD using the monoidal category $\Coh(\ag \times_{e_n, \ag, e_n} \ag)$ (or in the language of \cref{sect:categorical}, $\Coh(\wtag \times_{\ag} \wtag)$), rather than $\Coh(\ag)$. We do not explore this possbility in our paper.

\subsection{3d mirror symmetry}
\label{subsect:3dmsintro}
3d mirror symmetry studies a conjectural equivalence between the 3d A-model (associated to $X$) and the 3d B-model (associated to its mirror $X^!$). While it started out in mathematics as symplectic duality (between pairs of symplectic resolutions), its more fundamental formulation is as an equivalence of $2$-categories (i.e. the assignment to a point by the relevant 3d TFTs).

Defining the relevant $2$-categories is an area of active research. The B-type category was first studied in \cite{kr, krs}, and \cite{potent} recently proposed that a model for the B-side should be a microlocal version of $\Coht(X^!)$. The (non-microlocal version of the) 2-category $\Coht(X^!)$ was defined in \cite{stefanich}, though \cite{arinkintalk1, arinkintalk2, hypertoric} modeled it via proper categorical descent, using closed embeddings into $X^!$ as generating objects.

However there is relatively less known about the A-type category. In his 2014 ICM address \cite{teleman}, Teleman gave a proposal for the A-side $2$-category in the case of pure gauge theory (i.e. $X=T^*(\pt/G)$) using categorical representation theory; a de Rham variant was later proposed in \cite{potent}. The first case beyond pure gauge theory was studied by Gammage, Hilburn, and Mazel-Gee in \cite{3dms}, where they gave a candidate for the A-side category of $T^*\C$ as \textit{perverse schobers} on $(\mb{C}, 0)$ (denoted as $\Pervt(\C, 0)$). They demonstrated that this category is equivalent to modules over the convolution category $\Am=\Coh((\ag \sqcup \pg) \times_{\ag} (\ag \sqcup \pg))$, which for simplicity we call $\Coht(\ag)$ (suppressing the singular support Lagrangian).

One can think of $\Pervt(\mb{C}, 0)$ as the naive A-side $2$-category associated to $T^*\mb{C}$ and the Lagrangian given by conormals to the stratification $\mb{C}, 0$. Meanwhile, $\Coht(\ag)$ is the naive B-side category associated to $T^*(\ag)$ (with Lagrangian given by conormals to $\ag, \pg$). In this paper we construct the (B-side) pullback on $\Coht(\A^1/\Gm)$ along the $n$\textsuperscript{th} power map $e_n \colon \A^1/\Gm \to \A^1/\Gm$ (see \cref{subsect:Bsidepb}). This pullback generalizes the root stack construction (explained in \cref{subsect:vgitintro}). We show in particular that it recovers SODs of both $\Perf(\sqrt[n]{X/D})$ and $\Coh(\sqrt[n]{X/D})$ in \cref{thm:recover} using some theory of singular support of coherent sheaves and indcoherent convolution categories (extending results of \cite{spectralincarnation}), and more abstractly in \cref{prop:commutingrootstackdiagram} using the dualizability of the bimodule defining the root stack operation. See \cref{rem:betterrecover} for a discussion of these two approaches.

We also define an A-side pushforward of perverse schobers in \cref{subsect:Asidepf} on the object level, with a justification for our definition given in \cref{sect:schob}. When the perverse schober comes from a Landau--Ginzburg model with a given superpotential, this pushforward agrees with the Landau--Ginzburg model associated to the pushforward of that superpotential by the $n$\textsuperscript{th} power map on $\C$. Note that the upcoming work \cite{ACJ} provides a $2$-categorical pushforward functor between certain categories of schobers that agrees with ours on the object level.

Our main result is then the following.

\begin{thm}[\cref{thm:main}]
Under the 3d mirror symmetry equivalence (of \textup{\cite{3dms}})
$$\Pervt(\C, 0) \simeq \Coht(\ag),$$
our defined B-side pullback and A-side pushforward are equivalent functors on the object level.
\end{thm}

Using 3d mirror symmetry, we show that periodicity for root stack derived categories is manifest geometrically on the A-side by writing the associated perverse schober as a cosheaf over a Lagrangian skeleton and finding a decomposition of the skeleton in \cref{sect:schob}. Similar notions have appeared in \cite{bondaltalk} and \cite[Section 3]{christ}.

Our results also give a new proof of $2n$--periodicity for the $(n-1)$st level of the relative Waldhausen S-construction for a spherical functor (see \cref{subsect:Waldhausen}), as it can be recovered as the base change of a $2n$-periodic SOD of $\mc{A}$-modules.

\subsection{Conventions} \label{subsect:convent}
By category we will always mean $\infty$-category (and functors are automatically derived). In particular, we use $\Coh(X)$ and $\Perf(X)$ to denote the bounded derived ($\infty$-)categories of coherent sheaves and perfect complexes, respectively, on $X$. Note that the (pullback and pushforward) functors from window theory upgrade to canonical functors between $\infty$-categorical enhancements, and \cref{cor:htpycat} tells us that SODs of $\infty$-categories yield SODs of their homotopy categories. If desired, the reader can assume that from \cref{sect:geometryrootstack} to \cref{sect:multiwindow} we just work with the ordinary bounded derived category of coherent sheaves.

We work over a ring of characteristic zero. While the 3d mirror symmetry equivalence can be defined in stable categories over any field (or ring) (see \cite[Remark 0.11]{hypertoric}), and we can define endofunctors on the A-side and B-side in this context, we use results of \cite{BZNP}, which assumes characteristic zero. However we believe that the main theorem \cref{thm:main} should hold in arbitrary characteristic by manual calculation of the SOD.

Note that root stack results (e.g. \cref{thm:recover}) heavily rely on \cite{BZNP} and \cite{BZFN}, so we don't know whether they hold outside characteristic zero. Extending to this case would involve a larger generalization of \cite{BZNP} and \cite{BZFN} (in the former case, showing $\ag$ is perfect outside characteristic zero is enough; however in the latter case, characteristic zero is assumed at the outset of \cite{BZFN} so more care is needed). Note also that the characteristic zero assumption is integral to results involving singular support like \cref{thm:gendual}.

We write our semiorthogonal decompositions (SODs) as $\mc{C}=\langle \mc{D}_1, \dots, \mc{D}_n \rangle$ if there are no maps from $\mc{D}_j$ to $\mc{D}_i$ for $j>i$. Generally we will liberally use definitions and basic results stated in \cite{DKSS}. In particular, coCartesian gluing functors (which we just call gluing functors) are the functors $i_j^Li_i \colon \mc{D}_i \to \mc{D}_j$ (when they exist), which go from left to right.

\subsection{Acknowledgments}
We are grateful to David Nadler (the advisor of S.G. and Y.O.) and Merlin Christ for providing feedback. We would like to thank Will Donovan, Daigo Ito, and John Nolan for helpful discussion regarding window theory and GIT; Merlin Christ, Will Fisher, and Lyne Moser for helpful discussion regarding higher category theory and lax limits, as well as Rune Haugseng for answering a question about lax limits; Ansuman Bardalai for helpful discussion regarding Rozansky--Witten theory and singular support; and David Nadler, Vivek Shende, and Peng Zhou for teaching us about Fukaya--Seidel categories and wrapping. S.G. is supported by an NSF Graduate Research Fellowship. R.L. is supported by a department fellowship funded by The McBeth Family Fellowship and The Manferdelli Family Fund in Mathematics. Y.O. is supported by a Simons Dissertation Fellowship. Google Gemini 3.1 Pro was used to help create TikZ code for the diagrams in \cref{sect:schob}. We also used OpenAI's ChatGPT 5.6-Sol to help with the proof of \cref{prop:properdm} (extending \cite[Theorem 3.0.2]{BZNP} to DM stacks).

\section{Geometry of root stacks}
\label{sect:geometryrootstack}
\subsection{Generalizing the Bodzenta--Donovan construction}\label{subsect:rootstack}
Let $m, n$ be coprime, with $m < n$. In this section we construct a stack $\X$ which is a global quotient of a scheme by $\Gm$, such that taking certain open substacks (corresponding to taking positive and negative GIT quotients) yields $\sqrt[m]{X/D}$ and $\sqrt[n]{X/D}$, respectively. We will use this construction in \cref{sect:window} to demonstrate a $2n$-periodic embedding of $\Coh(\sqrt[m]{X/D})$ into $\Coh(\sqrt[n]{X/D})$. This construction is not unique, but depends on a choice of integers $a, b$ such that $am+bn=1$ (it turns out that each choice of embedding has the same image).
Let $X$ be a locally Noetherian scheme over a field of characteristic $0$, and let $D$ be an effective divisor of $X$, with $s$ the canonical section of $\OO_X(D)$. This data defines a map $X \to \A^1/\Gm$, using the fact that $\A^1/\Gm(X)=\{s \colon \mc{O}_X \to \mc{L}\}$.

\begin{defn}
    The $n^{\text{th}}$ root stack $\sqrt[n]{X/D}$, for a positive integer $n$, is given by the following fiber product:
\[\begin{tikzcd}
	{\sqrt[n]{X/D}} & {\A^1/\Gm} \\
	X & {\A^1/\Gm}.
	\arrow[from=1-1, to=1-2]
	\arrow[from=1-1, to=2-1]
	\arrow["\lrcorner"{anchor=center, pos=0.125}, draw=none, from=1-1, to=2-2]
	\arrow["{e_n}"', from=1-2, to=2-2]
	\arrow[from=2-1, to=2-2]
\end{tikzcd}\]
Here $e_n$ is induced by the $n$th power map on $\A^1$, which we note is not $\Gm$-equivariant. On $T$-points, it sends \begin{align*}
    \A^1/\Gm(T) & \to \A^1/\Gm(T) \\
    (s\colon \mc{O}_T \to \mc{L}) & \mapsto (s^{\otimes n}\colon \mc{O}_T \to \mc{L}^{\otimes n}).
\end{align*}

\end{defn}

We now give a family of constructions of the root stack as a global quotient stack. Let $a, b$ be integers such that $am+bn=1$. We define $T_{n,m; a,b}$ to be the total space of the rank $2$ vector bundle $\OO_X(aD)\oplus\OO_X(bD)$, with fiber coordinates $(y, z)$; we also have a fiberwise $\Gm$-action with weights $(-n, m)$. Let $X_{n,m; a, b}$ be the hypersurface given by $\{y^mz^n=s\}$, with $s$ the canonical section of $\OO_X(D)$. Note that our $a, b, m, n,$ and $\Gm$-action are chosen so that our setup is well-defined (as $y^mz^n$ is a section of $\OO_X(aD)^{\otimes m} \otimes \OO_X(bD)^{\otimes n}=\OO_X(D)$) and that $X_{n,m;a,b}$ is $\Gm$-stable.

If the fixed locus for the $\Gm$-action in $X_{n,m;a,b}$ is $Z$, then the unstable loci for the positive and negative $\Gm$-actions are respectively
\begin{equation}
    S^+ = \{x \in X_{n,m;a,b} \vert \lim_{\lam \to 0} \lam \cdot x \in Z\} \qquad \text{and} \qquad S^- = \{x \in X_{n,m;a,b} \vert \lim_{\lam \to 0} \lam^{-1} \cdot x \in Z\}
\end{equation}\label{eqn:unstableloci}

with semistable loci $X_{n,m;a,b}^{\pm}=X_{n,m;a,b}-S^{\pm}$, respectively. Note that these loci are taken scheme-theoretically. See \cref{subsect:GITconv} for an explanation of our conventions regarding GIT.

In our case, the fixed point locus is just the locus where $y$ and $z$, and thus $s$, are $0$, yielding $$Z \simeq D.$$ (More precisely, the ideal generated by the nonzero weight variables $y$ and $z$ is the defining ideal of the fixed point locus $Z$. The section $s$ happens to be in this ideal.) The unstable locus $S^+$ is the locus $\{y=0\}$ (as we quotient by the negative weight variable $y$), so $$S^+ \simeq \Tot(\OO_D(bD));$$ in particular, $S^+$ lives over $D$. Similarly, $S^-$ is the locus $\{z=0\}$, so $$S^- \simeq \Tot(\OO_D(aD))$$ and also lives over $D$. Note that the intersection of $S^-$ and $S^+$ is $Z$. Over $X-D$, we have $s \neq 0$, and the condition $\gcd(m ,n)=1$ implies that $\Gm$ acts freely; thus $(X_{n,m;a,b}/\Gm)|_{X-D} \simeq X-D$.

If we quotient by $\Gm$, then we have isomorphisms $$(S^+-Z)/\Gm \simeq \sqrt[m]{\OO_D(bD)}$$ and $$(S^--Z)/\Gm \simeq \sqrt[n]{\OO_D(aD)};$$
note that these stacks are a $\mu_m$-gerbe and $\mu_n$-gerbe over $D$, respectively.

We define $\X_{n,m;a,b}$ to be the global quotient stack $[X_{n,m;a,b}/\Gm]$, with associated GIT quotients $\X_{n,m;a,b}^{\pm}=[X_{n,m;a,b}^{\pm}/\Gm]$. Note that $\X_{n,1;1,0}$ is the GIT quotient construction of a root stack that appeared in \cite{BD}; thus $\X_{n,1;1,0}^+ \simeq \sqrt[n]{X/D}$ and $\X_{n,1;1,0}^- \simeq X$ by \cite[Proposition 2.9]{BD}.

We now show that taking GIT quotients (i.e. taking the $\Gm$ stack quotient of the positive and negative semistable loci) of our construction yields the $n$th and $m$th root stacks.
\begin{prop}\label{prop:GITrootstack}
    We have $\X^+_{n,m;a,b} \simeq \sqrt[n]{X/D}$. By symmetry, we also have $\X^-_{n,m;a,b} \simeq \sqrt[m]{X/D}$.
\end{prop}

Note that $X^+_{n,m;a,b} \not\simeq X^+_{n,1;1,0}$ (for $m > 1$). In general, if one has a map of schemes $X_1 \to X_2$, both of which individually have $\Gm$-actions, such that $X_1/\Gm \xrightarrow{\sim}  X_2/\Gm$, that does not imply that $X_1 \xrightarrow{\sim} X_2$; we would need $\Gm$-equivariance for that. For example, the map $\Gm \xrightarrow{x \mapsto x^2} \Gm$ is a $\Z/2\Z$-torsor (and thus not an isomorphism), but if we quotient each by weight-$1$ $\Gm$-actions (so the map is not equivariant), then $\Gm/\Gm \xrightarrow{\sim} \Gm/\Gm$ is an isomorphism of the point.

This example is a specific kind of failure though: if $X_1$ is a $\mu_a$-torsor over $X_2$, and the $\mu_a$-action on $X_1$ lifts to a $\Gm$ action, then we have an isomorphism $X_1/\Gm \xrightarrow{\sim} X_2/\Gm$, essentially via the sequence of algebraic groups $1 \to \mu_a \to \Gm \to \Gm \to 1$.

We claim that this is our situation regarding semistable loci.
\begin{thm}
    The schemes $S^--Z \hookrightarrow X_{n,1;1,0}^+$ form a map of $\mu_a$-torsors over $S^--Z \hookrightarrow X^+_{n,m;a,b}$, with compatible $\Gm$-actions as above, with all maps and actions fiberwise over $D \hookrightarrow X$.
\end{thm}
\begin{proof}
    We just need to determine the map on fiberwise coordinates. It sends $(y, z) \mapsto (y^a, y^bz)$.
\end{proof}
\begin{proof}[Proof of \cref{prop:GITrootstack}]
    The fiberwise $\Gm$-action on $X^+_{n,m;a,b}$ is by weights $(-n, m)$, and on $X^+_{n,1;1,0}$ is is by weights $(-n, 1)$. The $\mu_a$-torsor map is also a $\mu_a$-torsor of $\Gm$-actions as it sends the action $(-n, 1)$ to $(-an, -bn+1)=(-an, am)$ which is $a$ times the action $(-n, m)$. Thus we get an equivalence of stacks after quotienting by $\Gm$.
\end{proof}

The reason for our distinction between scheme and stack is due to its effect on the identity of the sheaf $\OO\langle 1 \rangle$, which will play a large role in \cref{sect:window}. This sheaf should be the result of twisting the structure sheaf by the generator $1$ of the character group of $\Gm$, and for a stack $X/\Gm$, it is the pullback along $X/\Gm \to B\Gm$ of $\OO_{B\Gm}\langle 1\rangle$. For $X/\Gm \simeq Y/\Gm$ coming from $X/\mu_a \simeq Y$, we have $\OO_{X/\Gm}\langle a\rangle \simeq \OO_{Y/\Gm}\langle 1\rangle$, yielding the following result in our case.
\begin{prop}\label{prop:matchtwist}
    We have $\OO_{X_{n,1;1,0}^+/\Gm}\langle a \rangle \simeq \OO_{X_{n,m;a,b}^+/\Gm} \langle 1 \rangle$.
\end{prop}

\begin{proof}
In our setup, taking the quotient $(S^--Z)/\Gm$ yields the $\mu_n$-gerbe $\sqrt[n]{\OO_D(aD)}$ as $S^-$ is just the total space of $\OO_D(aD)$. However, we know that $S^{-,\circ}_{n,1;1,0}$ is a $\mu_a$-torsor over $S^{-,\circ}_{n,m;a,b}$, yielding a proof that the $\mu_n$-gerbes $\sqrt[n]{\OO_D(D)}$ and $\sqrt[n]{\OO_D(aD)}$ are isomorphic as stacks. (As $\mu_n$ gerbes over $Z$, they are not isomorphic, as they correspond to the classes $\mathcal{O}_D(D)$ and $\mathcal{O}_D(aD)$ respectively. For more details, see for example \cite[B.1]{abramovich2008gromov}.)

More concretely, the following diagram commutes:
    \[\begin{tikzcd}
	{S^{-}_{n,1;1,0}} & {S^-_{n,m;,a,b}} \\
	{\mathcal{O}_D(D)} & {\mathcal{O}_D(aD)} \\
	{\mathcal{O}_{D}(D)- Z} & {\mathcal{O}_{D}(aD)-Z} \\
	{(\mathcal{O}_{D}(D)- Z)/_{n}\mathbb{G}_m} & {(\mathcal{O}_{aD}(D)- Z)/_{n}\mathbb{G}_m} \\
	{\sqrt[n]{\mathcal{O}_D(D)}} & {\sqrt[n]{\mathcal{O}_D(aD))}}
	\arrow[from=1-1, to=1-2]
	\arrow["\cong", no head, from=1-1, to=2-1]
	\arrow["\cong", no head, from=1-2, to=2-2]
	\arrow[from=2-1, to=2-2]
	\arrow[hook, from=3-1, to=2-1]
	\arrow[from=3-1, to=3-2]
	\arrow[from=3-1, to=4-1]
	\arrow[hook, from=3-2, to=2-2]
	\arrow[from=3-2, to=4-2]
	\arrow[from=4-1, to=4-2]
	\arrow["\cong", no head, from=4-1, to=5-1]
	\arrow["\cong", no head, from=4-2, to=5-2]
	\arrow["\cong", from=5-1, to=5-2]
\end{tikzcd}\]

where the isomorphism in the last row is an isomorphism of stacks (not $\mu_n$-gerbes!). Denote it by $\sigma$. Concretely, $\sigma$ is the map $\sqrt[n]{\mathcal{O}_D(D)} \to f_*(\sqrt[n]{\mathcal{O}_D(D)}) \cong \sqrt[n]{\mathcal{O}_D(aD)}$ where $f_*$ denotes the pushout along the map $f: \mathbb{G}_m \to \mathbb{G}_m, x \mapsto x^a$. This map has an inverse which is precisely given by identifying $\sqrt[n]{\mathcal{O}_D(D)}$ as the pushout of $\sqrt[n]{\mathcal{O}_D(aD)}$ along $\mathbb{G}_m \to \mathbb{G}_m, x \mapsto x^m$. By \cite[Proposition 2.1.2.5]{lieblich2007moduli} we see that $\sigma^*\OO_{X_{n,m;a,b}^+/\Gm} \langle 1 \rangle \simeq \OO_{X_{n,1;1,0}^+/\Gm}\langle a \rangle$. 
\end{proof}

\subsection{The root stack construction universally}
We can prove universally that taking $\Gm$-quotients yields our desired root stacks by interpreting the above construction more abstractly as a fiber product over maps of toric stacks.  (though this method makes less clear the relationship between $\OO\langle 1\rangle$s). We will generalize this construction in \cref{sect:multiwindow}. We claim that $\X$ is given by the following fiber product:

\[\begin{tikzcd}
	{\mathcal{X}} & {\mathbb{A}^2/\Gm^2} \\
	X & {\mathbb{A}^1/\mathbb{G}_m.}
	\arrow[from=1-1, to=1-2]
	\arrow[from=1-1, to=2-1]
	\arrow["{e_{m,n}}", from=1-2, to=2-2]
	\arrow[from=2-1, to=2-2]
    \arrow["\lrcorner"{anchor=center, pos=0.125, rotate=0}, draw=none, from=1-1, to=2-2]
\end{tikzcd}\]

More precisely, we can realize the scheme $X_{n,m;a,b}$ as the fiber product

\[\begin{tikzcd}
	X_{n,m;a,b} & {\mathbb{A}^2/_{a,b}\mathbb{G}_m} \\
	X & {\mathbb{A}^1/\mathbb{G}_m.}
	\arrow[from=1-1, to=1-2]
	\arrow[from=1-1, to=2-1]
	\arrow["{e_{m,n}}", from=1-2, to=2-2]
	\arrow[from=2-1, to=2-2]
    \arrow["\lrcorner"{anchor=center, pos=0.125, rotate=0}, draw=none, from=1-1, to=2-2]
\end{tikzcd}\]

Then as $\mathbb{A}^2/\mathbb{G}_m^2 \simeq (\mathbb{A}^2/_{a,b}\Gm)/_{-n, m}\Gm$, we know that the fiber product of the top diagram is the quotient of the fiber product of the bottom diagram by a $\Gm$-action.

\begin{prop}
    The scheme $X_{n,m;a,b}$ is the fiber product of the bottom diagram, and yields $\X$ as the fiber product of the top diagram after taking $\Gm$-quotients. The positive and negative semistable loci for the $\Gm$-action on $X_{n,m;a,b}$ are the preimages under the map $X_{n,m;a,b} \to \A^2/_{a,b}\Gm$ of $(\Gm \times \A^1)/_{a,b}\Gm$ and $(\A^1 \times \Gm)/_{a,b}\Gm$, respectively.

    In particular, with reference to our construction of $X_{n,m;a,b}$ above, the map to $\A^2/\Gm$ fiberwise sends $(y,z)$ (coordinates on $Tot(\OO_X(aD)\oplus\OO_X(bD))$) to $(y,z)$ (coordinates on $\A^2$).
\end{prop}
\begin{proof}
For a proof of the fiber product and the fiberwise coordinate statements, see the proof of \cref{thm:multiwindow}. Then the $\Gm$-action is via $(-n, m)$ fiberwise. The positive semistable locus consists of points that do not limit to the origin under the $\Gm$-action limiting to $0$; such points are everything with a nonzero first coordinate. The case of the negative semistable locus is similar.
\end{proof}

\begin{thm}
    The GIT quotients $\X^+$ and $\X^-$ are isomorphic to $\sqrt[n]{X/D}$ and $\sqrt[m]{X/D}$, respectively.
\end{thm}
\begin{proof}
    We do the + case; the - case is similar. The positive GIT quotient is the $\Gm$-quotient of the positive semistable locus, which by the previous proposition is the preimage of $[(\Gm \times \A^1)/_{a,b}\Gm]$. So the stack $\X^+$ is the preimage of $[(\Gm \times \A^1)/_{a,b}\Gm^2]$ in the map $\X \to \A^2/\Gm^2$, which is $\sqrt[n]{X/D}$, due to the map $[(\Gm \times \A^1)/_{a,b}\Gm^2]=\A^1/\Gm \to \A^1/\Gm$ being the map $e_n$, as shown in the following diagram.

\[\begin{tikzcd}
	X & \X & {\sqrt[n]{X/D}} \\
	{\mathbb{A}^1/\mathbb{G}_m} & {\mathbb{A}^2/\mathbb{G}_m^2} & {(\mathbb{G}_m \times \mathbb{A}^1)/\mathbb{G}_m^2=\mathbb{A}^1/\mathbb{G}_m}
	\arrow[from=1-1, to=2-1]
	\arrow[from=1-2, to=1-1]
	\arrow["\lrcorner"{anchor=center, pos=0.125, rotate=-90}, draw=none, from=1-2, to=2-1]
	\arrow[from=1-2, to=2-2]
	\arrow[from=1-3, to=1-2]
	\arrow["\lrcorner"{anchor=center, pos=0.125, rotate=-90}, draw=none, from=1-3, to=2-2]
	\arrow[from=1-3, to=2-3]
	\arrow[from=2-2, to=2-1]
	\arrow["{e_n}", curve={height=-18pt}, from=2-3, to=2-1]
	\arrow[from=2-3, to=2-2]
\end{tikzcd}\]
\end{proof}

\begin{rem}
    One way to see the fact that $\OO \langle 1 \rangle$ in our construction corresponds to $\OO \langle a \rangle$ in the usual case is the following. If we think of the $\Gm$ character as restricted from a $\Gm^2$-character, then $\Gm$ embeds as $(-n, m)$, so a $\Gm^2$-character $(c_1, c_2)$ looks like $-nc_1+mc_2$ from the perspective of the $\Gm$-action. So the $\OO\langle 1 \rangle$ from our construction corresponds to $\OO \langle c_1, c_2 \rangle$ on $\A^2/\Gm^2$ such that $-nc_1+mc_2=1$. On the positive semistable quotient $(\Gm \times \A^1)/\Gm^2$, the $c_1$ is irrelevant and all we see is $\OO \langle c_2\rangle$. This case is just the usual root stack construction via fiber product with $e_n$, so $\OO \langle 1 \rangle$ in our $e_{m, n}$ case should correspond to $\OO \langle c_2 \rangle$ in the usual case. Then by $-nc_1+mc_2=1$, we have $c_2= m^{-1}$ mod $n$.

    See \cref{rem:nothingnew} for an application of this reasoning.
\end{rem}

\section{Periodicity from window theory}\label{sect:window}
\subsection{The HKKN stratification in more generality}\label{subsect:GITconv}
Given a $G$-linearized line bundle $\mc{L}$ on a scheme $X$ (i.e. a line bundle on $X/G$), traditionally a point $x$ is semistable if there is a section $f$ of $\mc{L}^n$ for some $n$ such that both of the following are true:
\begin{itemize}
    \item $f(x) \neq 0$
    \item $\{f \neq 0\}$ is an affine subset of $X$.
\end{itemize}

Often $X$ is assumed to be quasiprojective or $\mc{L}$ is assumed to be ample, and the affine condition is more or less automatic. The above definition of semistability is a global notion.

However, there is also the Hilbert--Mumford numerical criterion, which is a local notion. We use this criterion for the definitions of the semistable and unstable locus. In the case that $\Gm$ acts on $X$, and we take the line bundles $\OO\langle 1 \rangle$ and $\OO\langle -1 \rangle$, the unstable loci are respectively the ones in \cref{eqn:unstableloci}.

\begin{exmp}
    Suppose we take $X=\Gm \times Y$ for $Y$ proper, and $\Gm$ acting with weight $1$ on $\Gm$ and trivially on $Y$. Since there is no fixed point locus, by the Hilbert--Mumford numerical criterion, the whole of $X$ is semistable (for either the positive or negative $\Gm$-weight). However, an actual section of $\OO\langle n \rangle$ is just a section of $\Gm$ pulled back to $X$, and so the locus where it is nonzero is an affine of $\Gm$, times $Y$, which is not affine. In this case we would want the semistable locus to be all of $X$.
\end{exmp}

In other words, given a reductive group action of a group on a scheme $X$ and a character of the group, there is a way to get an open-closed stratification of $X$. This definition is inspired by GIT but not technically the same traditionally; for simplicity we reuse the GIT terminology for our definitions.

One can also extend the definition of the HKKN stratification of the unstable locus, which generalizes the Hilbert--Mumford numerical criterion, by not assuming quasiprojectiveness of $X$ (or fixing an ample line bundle). Traditionally the HKKN stratification is defined for a projective-over-affine variety $X$ with respect to a $G$-linearized ample $\OO(1)$ on this $X$. The procedure of defining the stratification involves ordering possible strata according to numbers arising from the Hilbert--Mumford criterion, given a norm on the cocharacter lattice of $G$. However, any stratification made up of pieces as in \cref{def:HKKNstrata} (not necessarily via the Hilbert--Mumford ordering) yields a valid stratification for window theory as long as the strata satisfy certain properties explained in \cref{prop:windowsize}.

\begin{rem}
    Even the original case of the canonical HKKN for $X$ projective-over-affine must involve scheme-theoretic notions, as shown by the following example. Consider the affine scheme $\Spec k[x, y, z]/(x^2-yz^2)$, acted on by $\Gm$ with weights $(0, -2, 1)$ (note that this comes from the square root stack of $\A^1$ with divisor ideal sheaf $(x^2)$). Then the scheme-theoretic positive unstable locus (which is the same as in usual GIT) is cut out by $(y)$. However the resulting scheme is $\Spec k[x, y, z]/(x^2, y)$ which is not a variety; the variety would be $\Spec k[x, y, z]/(x, y)$. The correct closed subscheme to use for the HKKN stratification is $\Spec k[x, y, z]/(x^2, y)$ which is regularly embedded in $\Spec k[x, y, z]/(x^2-yz^2)$, as $\Spec k[x, y, z]/(x, y)$ does not satisfy property (S3) (which it should by \cite[Remark 2.3]{window}). (The conormal sheaf for $\Spec k[x, y, z]/(x, y)$ has an extra $0$ $\Gm$-weight due to the presence of $x$ which has $0$ $\Gm$-weight in the ideal.)
\end{rem}

Now we explain the conventions/definition for strata that are necessary for window theory. We will work in the case of a torus action as the definition is simplified. The only requirement is that we have a scheme $X/G$ with an HKKN stratification, defined as the following (see \cite[Definition 2.2]{window}).

\begin{defn}\label{def:HKKNstrata}
Given a scheme $X$ with an action of a torus $G$, a closed HKKN stratum is the following data:
\begin{itemize}
    \item a 1-parameter subgroup $\lam \colon \Gm \to G$
    \item a connected component $Z$ of the scheme-theoretic fixed point locus $X^\lam$
    \item the closed subscheme $S$ of $X$ defined as the scheme-theoretic locus $\{x \in X | \lim_{t \to 0} \lam(t) \cdot x \in Z\}$.
\end{itemize}

These fit into a diagram 
\[\begin{tikzcd}
	Z & S & X.
	\arrow["\sigma", shift left, from=1-1, to=1-2]
	\arrow["\pi", shift left, from=1-2, to=1-1]
	\arrow["j", from=1-2, to=1-3]
\end{tikzcd}\]
An HKKN stratification is a stratification of $X$ into locally closed subvarieties $X^{ss} \cup \bigcup_{i=1}^n S_i$ such that each $S_i$ is the $S$ in a closed HKKN stratum (with the relevant data) of $X-\bigcup_{j>i} S_j$. (So $S_1$ is a closed HKKN stratum in $X^{ss} \cup S_1$, etc.)
\end{defn}

\begin{warn}\label{warn:technicalHKKN}
    To use window theory on an HKKN stratification as explained in \cref{sect:window}, one needs the following technical hypothesis as explained in \cite[Definition 2.2]{window}. If $X$ is not smooth along $Z$, there is a $G$-equivariant closed immersion $X \subset X'$ and a KN stratum $S' \subset X'$ such that $S$ is a union of connected components of $S' \cap X$ and $X'$ is smooth in a neighborhood of $Z$.

    We are not sure what conditions on $(X, D)$ satisfy this hypothesis for our construction (or the construction in \cite{BD}). However we believe that in our specific situation, the proofs of window-theoretic results still work and we only need to assume $X$ locally Noetherian. Note that our results are generalized in \cref{sect:categorical} (using categorical representation theory), where the proofs work in more generality (except that we assume locally finite type to apply \cref{cor:dmhom}; it is likely that locally Noetherian is enough though).
\end{warn}

Note that for our main construction there is only one stratum of the unstable locus.

\subsection{Window theory}

We now describe windows on $\X$. Window theory was developed in \cite{window}; we will summarize the parts that we use here, largely adapted from \cite{window2} and \cite{BD}. Specifically, we will focus just on the case of $\Gm$ acting on a scheme $X$. In this case we have unstable loci $S^{\pm}$, semistable loci $X^{\pm}$, and a fixed point locus $Z$. Note that these are two different HKKN stratifications of $X$, each with one stratum besides the open locus.

\begin{prop}\textup{\cite[Proposition 3.1]{BD}\cite[Section 2]{window}}\label{prop:windowsize}
Assume the further conditions on the HKKN strata:

(A): the projections $\pi_{\pm} \colon S^{\pm} \to Z$ are locally trivial bundles of affine spaces.
(R): the inclusions $S^{\pm} \to X$ are regular embeddings.

Then by (R), the relative cotangent complex $L_{S^{\pm}|X}$ is locally free, and is isomorphic to the shifted conormal complex $\mc{N}^{\vee}_{S^{\pm}}X[1]$. It turns out that $\det \sigma_{\pm}^*\mc{N}^{\vee}_{S^{\pm}}X$ has a single positive and negative weight with respect to $\Gm$ in the + and - cases, respectively. Let the absolute values of these be $\eta_{\pm}$. Then there is an equivalence under pullback from $X/\Gm$ to $X^{\pm}/\Gm$:
$$\Coh(X^{\pm}/\Gm) \simeq \mc{C}_{[w, w+\eta_{\pm})} = \{E \in \Coh(X/\Gm) | \mc{H}^{\bullet}(\sigma^*E)\text{ has $\Gm$-weights in }[w, w+\eta_{\pm})\}.$$
Here we abuse notation to let $\sigma$ be the embedding $Z \inclto X$, as well as $Z/\Gm \inclto X/\Gm$.
\end{prop}
\begin{rem}
Window theory works for a general HKKN stratification as long as (A) and (R) are satisfied for each stratum.
\end{rem}

The exact definition of $\mc{C}_{[w, w+\eta_{\pm})}$ is not relevant; rather we use the amplification property. First, consider the diagram 
\[\begin{tikzcd}
	Z & {Z/\Gm} & {S^{\pm}/\Gm} & {X/\Gm}.
	\arrow["\tau", from=1-2, to=1-1]
	\arrow["{\sigma_{\pm}}", shift left, from=1-2, to=1-3]
	\arrow["{\pi_{\pm}}", shift left, from=1-3, to=1-2]
	\arrow["j_{\pm}", from=1-3, to=1-4]
\end{tikzcd}\]
\begin{lem}\textup{\cite[Remark 2.13]{window}}
    The functor $\Phi^{\pm}_{w}(-) = j_{\pm, *}\pi_{\pm}^*\tau^* \langle w \rangle$ is fully faithful from $\Coh(Z)$ to $\Coh(X/\Gm)$. Let the image be $\mc{A}_{w}^{\pm}$.
\end{lem}

Recall that for a stack that is a global quotient of a scheme by $\Gm$, we denote by $\OO\langle 1 \rangle$ the pullback of the sheaf $\OO\langle 1 \rangle$ on $B\Gm$ (and by $\langle 1 \rangle$ the tensor product with this sheaf).

Now WLOG suppose $\eta_+>\eta_-$.

\begin{prop}\textup{\cite[Proposition 3.3]{BD}\cite[Amplification 2.11]{window}\cite[Equation 3]{window2}}
    We have semiorthogonal decompositions
    \begin{align*}
        \mc{C}_{[w, w+\eta_+)}=&\langle \mc{C}_{[w, w+\eta_-)}, \mc{A}^-_w, \mc{A}^-_{w+1}, \dots, \mc{A}^-_{w+\eta_+-\eta_--1}\rangle \\
        =&\langle \mc{A}^-_w, \mc{C}_{[w+1, w+1+\eta_-)}, \mc{A}^-_{w+1}, \dots, \mc{A}^-_{w+\eta_+-\eta_--1}\rangle\\
        =&\cdots\\
        =&\langle \mc{A}^-_w, \mc{A}^-_{w+1}, \dots, \mc{A}^-_{w+\eta_+-\eta_--1}, \mc{C}_{[w+\eta_+-\eta_-, w+\eta_+)}\rangle.
    \end{align*}
\end{prop}

Now we specialize the theory to our case; all the notation is the same except that the scheme $X$ that we do window theory on is now $X_{n,m;a,b}$.

\begin{prop}
    In our case, the assumptions (A) and (R) are satisfied.
\end{prop}
\begin{proof}
    This follows from our \cref{prop:strataAR}.
\end{proof}
\begin{prop}
    We have $\eta_+=n$ and $\eta_-=m$.
\end{prop}
\begin{proof}
    Locally $I_{S^+}/I_{S^+}^2$ is spanned by $y$ and $I_{S^-}/I_{S^-}^2$ is spanned by $z$. The $\Gm$-weights on these are $-n$ and $m$.
\end{proof}
Consider the diagram
\[\begin{tikzcd}
	{Z/\mathbb{G}_m} & {S^-/\mathbb{G}_m} & X_{n,m;a,b}/\Gm \\
	{Z/\mathbb{G}_m} & {S^{-,\circ}/\mathbb{G}_m} & {X_{n,m;a,b}^+/\Gm}
	\arrow[no head, from=1-1, to=2-1]
	\arrow[shift left, no head, from=1-1, to=2-1]
	\arrow["\pi_-", from=1-2, to=1-1]
	\arrow["j_-", hook, from=1-2, to=1-3]
	\arrow["k", hook, from=2-2, to=1-2]
	\arrow["\pi^\circ_-", from=2-2, to=2-1]
	\arrow["j^\circ_-", hook, from=2-2, to=2-3]
	\arrow["i_+", hook, from=2-3, to=1-3]
    \arrow["\lrcorner"{anchor=center, pos=0.125, rotate=90}, draw=none, from=2-2, to=1-3]
\end{tikzcd}\]

and the diagram (implied in \cref{subsect:rootstack})
\[\begin{tikzcd}
	{\sqrt[n]{\OO_D(aD)}} & {X_{n,m;a,b}^+/\Gm} \\
	D & X.
	\arrow["i", from=1-1, to=1-2]
	\arrow["q", from=1-1, to=2-1]
	\arrow["p", from=1-2, to=2-2]
	\arrow["{i_D}", from=2-1, to=2-2]
\end{tikzcd}\]

Note that this diagram is not Cartesian, but is pulled back from the diagram

\[\begin{tikzcd}
	{\pt/\Gm} & {\A^1/\Gm=(\A^1 \times \Gm)/\Gm^2} \\
	{\pt/\Gm} & {\A^1/\Gm}
	\arrow[from=1-1, to=1-2]
	\arrow["h_n", from=1-1, to=2-1]
	\arrow["{e_n}", from=1-2, to=2-2]
	\arrow[from=2-1, to=2-2]
\end{tikzcd}\]

via the map $X \to \A^1/\Gm$.

\begin{prop}[Embeddings of $D$]
    We have $i_+^* \circ \Phi_\omega^-(-) \simeq i_*q^*(-) \otimes \OO_{X_{n,m;a,b}/\Gm}\langle \omega \rangle$.
\end{prop}
\begin{proof}
Our proof is essentially the same as the ``Embeddings of $\Coh(D)$'' part of the proof of \cite[Theorem 3.7]{BD}. For clarity, we explain their proof in our setting. Let $S^{-, \circ}$ be the open subscheme $S^--Z$ of $S^-$. Note that this is the intersection of $S^-$ and $X^+$ inside $X_{n,m;a,b}$. Let $\tau \colon Z/\Gm \to Z$. Then by flat base change, \begin{align*}
    i_+^* \circ \Phi^-_\omega(-) := &i_+^*j_{-,*}\pi_-^*(\tau^*(-) \otimes \OO_{Z/\Gm} \langle \omega \rangle) \\
    \simeq & j_{-,*}^\circ k^* \pi_-^*(\tau^*(-) \otimes \OO_{Z/\Gm} \langle \omega \rangle) \\
    \simeq & j_{-,*}^\circ \pi_-^{\circ *}(\tau^*(-) \otimes \OO_{Z/\Gm} \langle \omega \rangle).
\end{align*}

The projection formula yields
\begin{align*}
    i_+^* \circ \Phi^-_\omega(-) \simeq & j_{-,*}^\circ \pi_-^{\circ *}\tau^*(-) \otimes \OO_{X_{n,m;a,b}^+/\Gm}\langle \omega \rangle \\
    \simeq & j_{-,*}^\circ \sigma_-^{\circ, *}(-) \otimes \OO_{X_{n,m;a,b}^+/\Gm}\langle \omega \rangle
\end{align*}
where $\sigma_-^\circ=\tau \pi_-^\circ \colon S^{-,\circ}/\Gm \to Z$.

Now we want to compare $j_{-,*}^\circ \sigma_-^{\circ, *}$ to $i_*q^*$. We also use $i'$ and $q'$ to refer to the maps called $i$ and $q$ in \cite[Diagram 3.C; Theorem 3.7]{BD}. Indeed, we have a commutative diagram with vertical isomorphisms:

\[\begin{tikzcd}
	D & {\sqrt[n]{\OO_D(D)}} & {X_{n,1;1,0}/\Gm \simeq\sqrt[n]{X/D}} \\
	D & {\sqrt[n]{\OO_D(aD)}} & {X_{n,m;a,b}^+/\Gm} \\
	Z & {S^{-,\circ}/\Gm} & {X_{n,m;a,b}^+/\Gm}.
	\arrow[no head, from=1-1, to=2-1]
	\arrow[shift left, no head, from=1-1, to=2-1]
	\arrow["q'", from=1-2, to=1-1]
	\arrow["i'", hook, from=1-2, to=1-3]
	\arrow["\sim", from=1-2, to=2-2]
	\arrow["\sim"', from=1-3, to=2-3]
	\arrow["q", from=2-2, to=2-1]
	\arrow["i", hook, from=2-2, to=2-3]
	\arrow[no head, from=2-3, to=3-3]
	\arrow[shift left, no head, from=2-3, to=3-3]
	\arrow["\sim"', from=3-1, to=2-1]
	\arrow["\sim"', from=3-2, to=2-2]
	\arrow["\sigma_-^\circ", from=3-2, to=3-1]
	\arrow["{j_-^\circ}"', hook, from=3-2, to=3-3]
\end{tikzcd}\]

\end{proof}

The following proposition is immediate, yielding a proof of the SODs (without periodicity) appearing in \cref{thm:mainresult}.
\begin{prop}[Embeddings of $m$th root stack]
\begin{enumerate}
\item We have an isomorphism $\mc{C}_{[0, m)} \simeq \Coh(\sqrt[m]{X/D})$ induced by pullback along $i_-\colon \X^-_{n,m;a,b} \to \X_{n,m;a,b}$.
\item As subcategories of $\Coh(X_{n,m;a,b}/\Gm)$, we have an equality $\mc{C}_{[k, k+m)}=\mc{C}_{[0, m)} \otimes \OO_{X_{n,m;a,b}/\Gm}\langle w \rangle$.
\item Pullback along $i_+ \colon \X^+_{n,m;a,b} \to \X_{n,m;a,b}$ is fully faithful on any $\mc{C}_{[k, k+m)}$ and yields an equality $i_+^*\mc{C}_{[k, k+m)}=i_+^*\mc{C}_{[0, m)} \otimes \OO_{X^+_{n,m;a,b}/\Gm}\langle w \rangle$.
\end{enumerate}
\end{prop}
\begin{proof} Note that this result is akin to the ``Embeddings of $\Coh(X)$'' part of \cite[Theorem 3.7]{BD}. However our case is simpler as we don't need to match a window with $p^*\Coh(X)$.

Part (1) is automatic from \cref{prop:windowsize} as $\Coh(\sqrt[m]{X/D})$ is equivalent to the - semistable quotient. Parts (2) and (3) are automatic from \cref{prop:windowsize}.
\end{proof}

For periodicity, we need to show that tensoring with $\OO_{X^+_{n,m;a,b}/\Gm}\langle n \rangle$ preserves the embeddings of $\Coh(D)$. (It will follow that it preserves the embedding of $\Coh(\sqrt[m]{X/D})$; see the end of the proof of \cref{thm:mainresult}.)

\begin{prop}\label{prop:Dintertwine}
    The autoequivalence $- \otimes \OO_{X_{n,m;a,b}^+/\Gm}\langle n \rangle$ preserves $i_*q^*\Coh(D) \otimes \OO_{X_{n,m;a,b}^+/\Gm}\langle k \rangle$ for any $k$, as it intertwines with $- \otimes \OO_D(-aD)$ via $i_*q^*$.
\end{prop}
\begin{proof}
Our proof is similar to \cite[Proposition 4.1]{BD}.
We have

    \begin{align*}
        i_*q^*(- \otimes \OO_D(-aD)) \simeq &i_*q^*(- \otimes i_D^*\OO_{X}(-aD)) \\
        \simeq&i_*(- \otimes q^*i_D^*\OO_{X}(-aD))q^*\\
        \simeq&i_*(- \otimes i^*p^*\OO_{X}(-aD))q^*\\
        \simeq&i_*(- \otimes i^*\OO_{X_{n,m;a,b}^+/\Gm}\langle n \rangle)q^*\\
        \simeq& (- \otimes \OO_{X_{n,m;a,b}^+/\Gm}\langle n \rangle)i_*q^*.
    \end{align*}

    Here we use the fact that $p'^*\OO_X(-D) \simeq \OO_{X_{n,1;1,0}^+/\Gm}\langle n \rangle$ from \cite[Proposition 4.1]{BD}. Then $p'^*\OO_X(-aD) \simeq \OO_{X_{n,1;1,0}^+/\Gm}\langle an \rangle \simeq \OO_{X_{n,m;a,b}^+/\Gm}\langle n \rangle$ by \cref{prop:matchtwist} (note that $p'$ and $p$ are the same map under $X_{n,1;1,0}^+/\Gm \simeq X_{n,m;a,b}^+/\Gm$).
    \end{proof}

We can now finish the proof of \cref{thm:mainresult}.

\begin{proof}[Proof of \cref{thm:mainresult}.]
    Via our amplifications, we have
    \begin{align*}
        \Coh(\X^+) &= \langle i_+^*\mathcal{C}_{[0, m)}, i_*q^*\Coh(D), i_*q^*\Coh(D) \otimes \OO_{\X^+}\langle 1 \rangle, \dots, i_*q^*\Coh(D) \otimes \OO_{\X^+}\langle n-m-1 \rangle \rangle \\
        &=\langle i_*q^*\Coh(D), i_+^*\mathcal{C}_{[1, m+1)}, i_*q^*\Coh(D) \otimes \OO_{\X^+}\langle 1 \rangle, \dots, i_*q^*\Coh(D) \otimes \OO_{\X^+}\langle n-m-1 \rangle \rangle\\
        &=\cdots\\
        &=\langle i_*q^*\Coh(D), i_*q^*\Coh(D) \otimes \OO_{\X^+}\langle 1 \rangle, \dots, i_*q^*\Coh(D) \otimes \OO_{\X^+}\langle n-m-1 \rangle, i_+^*\mathcal{C}_{[n-m, n)} \rangle \\
        &=\langle i_*q^*\Coh(D), i_*q^*\Coh(D) \otimes \OO_{\X^+}\langle 1 \rangle, \dots, i_*q^*\Coh(D) \otimes \OO_{\X^+}\langle n-m-1 \rangle, i_+^*\mathcal{C}_{[0, m)} \otimes \OO_{\X^+}\langle n-m \rangle \rangle.
        \end{align*}
Now the right orthogonal of $i_+^*\mathcal{C}_{[0, m)} \otimes \OO_{\X^+}\langle n-m \rangle$ can be obtained by simply tensoring our first SOD with $\OO_{\X^+}\langle n-m \rangle$. Using \cref{prop:Dintertwine}, we get
\begin{align*}\Coh(\X^+) = \langle i_+^*\mathcal{C}_{[0, m)}\otimes \OO_{\X^+}\langle n-m \rangle , i_*q^*\Coh(D)\otimes \OO_{\X^+}\langle n-m \rangle, i_*q^*\Coh(D) \otimes \OO_{\X^+}\langle n-m+1 \rangle\\ \dots, i_*q^*\Coh(D) \otimes \OO_{\X^+}\langle 2n-2m-1 \rangle \rangle .\end{align*}
Performing this mutation a total of $n$ times (so $2n$ total mutations) yields the result after twisting the original SOD $n$ times by $\OO_{\X^+}\langle n-m \rangle$. (Note that a twist by $\OO_{\X^+}\langle n \rangle$ preserves all the copies of $\Coh(D)$, as well as $\Coh(\X^+)$; thus it preserves $\mc{C}_{[0, m)}$ (the copy of $\Coh(\sqrt[m]{X/D})$) as well.)
\end{proof}

Let us now compare our SOD to the one in \cite{BD} and previous papers. We use their definition of $\OO_{\sqrt[n]{X/D}}\langle 1 \rangle$, i.e. by \cref{prop:matchtwist} we have $\OO_{\sqrt[n]{X/D}}\langle a \rangle \simeq \OO_{X^+_{n,m;a,b}/\Gm}\langle 1 \rangle$. Note also that $i_*q^*\Coh(D)=i'_*q'^*\Coh(D)$ as subcategories under the isomorphism $\sqrt[n]{X/D} \simeq X^+_{n,m;a,b}/\Gm$. Then we have the following theorem.

\begin{thm}
    We have a $2n$-periodic SOD
    $$\Coh(\sqrt[n]{X/D})=\langle \Coh(\sqrt[m]{X/D}), i'_*q'^*\Coh(D), i'_*q'^*\Coh(D) \otimes \OO_{\sqrt[n]{X/D}}\langle a \rangle, \dots, i'_*q'^*\Coh(D) \otimes \OO_{\sqrt[n]{X/D}}\langle (n-m-1)a \rangle \rangle.$$
\end{thm}

\begin{rem}
    In the original SOD of \cite{BD}, all pairs of embeddings of $\Coh(D)$ that have nonzero maps between them are ordered correctly: namely, $i_*q^*\Coh(D)$ is before $i_*q^*\Coh(D) \otimes \OO_{\sqrt[n]{X/D}}\langle 1 \rangle$. However, here the residues mod $n$ are somewhat scrambled; yet there are no issues. As $m$ is the inverse of $a$ mod $n$, we know that $ka+1 \equiv a(k+m) \equiv a(k+m-n)$, and this residue cannot appear before $ka$ since we take only $n-m$ consecutive multiples of $a$.

    By matching, we can even ``compute'' what the subcategory $\Coh(\sqrt[m]{X/D})$ is of our category. It is the category generated by the original copy of $\Coh(X)$ as well as all the $i_*q^*\Coh(D) \otimes \OO_{\sqrt[n]{X/D}}\langle \ell a \rangle$, for $n-m \le \ell \le n-1$, after mutating each of these copies all the way to the left.
\end{rem}

\subsection{Gluing functors and abstract embeddings} \label{subsect:abstractgluing1}
It turns out that the periodicity under mutation of the embedding $\Coh(\sqrt[m]{X/D}) \inclto \Coh(\sqrt[n]{X/D})$ actually happens in much more generality than in the case of the window theory construction above, for formal reasons involving gluing functors. It is shown in \cite[Proposition 4.4]{BD} that the gluing functor from the first copy of $\Coh(X)$ to $\Coh(D)$ is $i_D^*$, and that the gluing functors from other copies of $\Coh(X)$ to $\Coh(D)$ are zero.

We will show that the gluing functors between copies of $\Coh(D)$ are zero for nonadjacent copies and $[1]\langle 1 \rangle$ otherwise (note that the latter is abstractly equivalent to $\id$).

It is well known that a coherent sheaf on a $\mu_n$-gerbe over $D$ (as $\mu_n$ is commutative) splits into a direct sum of coherent sheaves for each character of $\mu_n$. This result then easily yields that the left and right adjoints to the pullback $q'^* \colon \Coh(D) \to \Coh(\sqrt[n]{\OO_D(D)})$ are the same.

\begin{prop}
    There is a natural isomorphism $[1] \simeq q'_*i'^*i'_*q'^*\langle -1 \rangle \colon \Coh(D) \to \Coh(D)$.
\end{prop}
\begin{proof}
Suppose we have a coherent sheaf $\mc{F}$ on $D$. We take a locally-free resolution, which locally on $D$ is free. We will show that locally $q'_*i'^*i'_*q'^*\mc{F} \simeq \mc{F}$ using this resolution.

So assume $D$ is such that all the bundles forming the resolution of $\mc{F}$ are free. If we resolve $\mc{F}$ and take $i'_*q'^*$ of this resolution, the result is no longer a free resolution but a complex of direct sums of $\mc{O}_D$. On $\sqrt[n]{X/D}$, using the complex $\OO_{\sqrt[n]{X/D}}\langle 1 \rangle \to \OO_{\sqrt[n]{X/D}} \to \OO_D$, we can replace each $\mc{O}_D$ with a two-term complex.

Now if we take $i'^*\langle -1 \rangle$, then we get a complex on $\sqrt[n]{\OO_D(D)}$ where each row is a two term complex consisting of (a direct sum of copies of) $\OO_{\sqrt[n]{\OO_D(D)}}\langle -1 \rangle$ and $\OO_{\sqrt[n]{\OO_D(D)}}$. Finally, by pushing forward to $D$, the $1$-graded pieces vanish, and we are left with the original resolution of $\mc{F}$ (cohomologically shifted by $1$).

All our operations are functorial and respect gluing, giving us a natural isomorphism.
\end{proof}

\begin{rem}
    Later we will give a proof using categorical representation theory (see \cref{thm:categoricalgluing}).
\end{rem}

The gluing functors of the SOD are thus only between adjacent components and are the pullback $\Coh(X) \to \Coh(D)$ and the the map $\langle 1 \rangle[1] \colon \Coh(D) \to \Coh(D)$; note that the latter is abstractly equivalent to the identity.

We can compute gluing functors under mutation, as explained in \cref{thm:mutation} (see also \cite[3.4.12]{infrared}) (here it is important that we work with $\infty$-categories). The key facts that we need are the following:
\begin{itemize}
    \item for a gluing functor $M_{i,i+1} \colon \mc{V}_i \to \mc{V}_{i+1}$, after mutating $\mc{V}_{i+1}$ over $\mc{V}_i$, the new gluing functor is $M^L_{i, i+1} \colon \mc{V}_{i+1} \to \mc{V}_i$;
    \item given an SOD $\langle \mc{V}_i, \mc{V}_{j}, \mc{V}_{j+1} \rangle$, if we mutate $\mc{V}_{j+1}$ over $\mc{V}_{j}$, the new gluing functor is $\fib(M_{i, j+1} \to M_{j, j+1} \circ M_{i, j}) \colon \mc{V}_i \to \mc{V}_{j+1}$; and
    \item given an SOD $\langle \mc{V}_i, \mc{V}_{i+1}, \mc{V}_{j} \rangle$, if we mutate $\mc{V}_{i+1}$ over $\mc{V}_i$, the new gluing functor is
    $\cofib(M_{i, j} \circ M_{i, i+1}^R \to M_{i+1, j}) \colon \mc{V}_{i+1} \to \mc{V}_j$.
\end{itemize}

In our case, all the relevant gluing functors are just the identity and $0$ (as are all adjoints of these functors).

\begin{thm}\label{thm:mutatearb1}
    Consider the SOD $$\Coh(\sqrt[n]{X/D})=\langle p^*\Coh(X), i_*q^*\Coh(D), i_*q^*\Coh(D)\langle 1\rangle, \dots, i_*q^*\Coh(D)\langle n-2 \rangle \rangle$$ from \textup{\cite{BD}} (so we have replaced the notation $i', q'$ with $i, q$ here). If we take any set of elements $0 \le a_1 < a_2 < \dots < a_{m-1} \le n-2$ and mutate each $i_*q^*\Coh(D)\langle a_i \rangle$ to the left, then the resulting category (generated by $p^*\Coh(X)$ and the left-mutated versions of the $i_*q^*\Coh(D)\langle a_i\rangle$) is equivalent to $\Coh(\sqrt[m]{X/D})$.
\end{thm}
\begin{proof}
    We will demonstrate this by induction. For simplicity, let us denote the $i$th copy of $\Coh(D)$ by $\Coh(D)_i$ (using $0$-indexing). Then we want to determine the gluing functors after we mutate each $\Coh(D)_{a_i}$ to the left. We claim that after doing this for $1 \le i \le k$, the gluing functors are $i_D^* \colon \Coh(X) \to \Coh(D)_{a_1}$, then $\id \colon \Coh(D)_{a_i} \to \Coh(D)_{a_{i+1}}$ for $i<k$, and finally $\id \colon \Coh(D)_{a_k} \to \Coh(D)_{a_{k+1}}$, with $\Coh(D)_{a_{k+1}}$ not connected to anything else between it and $\Coh(D)_{a_k}$. We will prove this by induction. Using the mutation rules above, it is easy to check that if we start with a diagram 
\[\begin{tikzcd}
	{\mc{A}} & {\mc{B}_1} & {\mc{B}_2} & {\mc{B}_3} 
	\arrow["F", from=1-1, to=1-2]
	\arrow["\id", from=1-2, to=1-3]
	\arrow["\id", from=1-3, to=1-4]
\end{tikzcd}\]

and mutate $\mc{B}_2$ to the left we get

\[\begin{tikzcd}
	{\mc{A}} & {\mc{B}_2} & {\mc{B}_3}
	\arrow["F", from=1-1, to=1-2]
    \arrow["\id", from=1-2, to=1-3]
\end{tikzcd}\]

where we ignore $\mc{B}_1$ as it has no nonzero gluing functors to its right. so every time we mutate a category to the left, we just keep a linear chain of gluing functors but shorten it by one. Thus the end result in our case is
$$\Coh(X) \xto{i_D^*} \Coh(D)_{a_1} \xto{\id} \Coh(D)_{a_2} \xto{\id} \cdots \xto{\id} \Coh(D)_{a_k}.$$

These are the same gluing functors making up $\Coh(\sqrt[m]{X/D})$; thus the category generated by these components is abstractly equivalent to $\Coh(\sqrt[m]{X/D})$ by \cref{thm:linearSOD}.
\end{proof}

It turns out that these embeddings are also $2n$-periodic. We will show this when we give a proof using categorical representation theory (see \cref{thm:categoricalgluing}).

\section{Generalizations of the root stack construction}\label{sect:multiwindow}
\subsection{Windows in higher dimensions}\label{subsect:multiwindow}
Suppose we have a map $\A^k/\Gm^{k} \xrightarrow{e_{n_1, n_2, \dots, n_k}} \A^1/\Gm$, generalizing the maps in \cref{subsect:rootstack} for $k=2$. We assume that $\gcd(n_1, \dots, n_k)=1$. This map can be defined as the composition of the product of maps $\A^1/\Gm \xrightarrow{e_{n_i}} \A^1/\Gm$ and the convolution product on $\A^1/\Gm$. Note that if we take the pullback of this map along $\A^1 \to \A^1/\Gm$, then the result is $\A^k/\Gm^{k-1}$, where $\Gm^{k-1}$ is the kernel of the map of tori $\Gm^k \xrightarrow{\begin{pmatrix} n_1 & n_2 & \cdots & n_k \end{pmatrix}} \Gm$.

The map $\A^k/\Gm^{k-1} \to \A^k/\Gm^k$ is just a quotient by $\Gm$. The set of unstable points of any $\Gm^{k-1}$-character $\chi$ on $\A^k$ witnessed by a 1-parameter subgroup $\lam \colon \Gm \to \Gm^{k-1}$ will be $V(x_{a_1}, x_{a_2}, \dots, x_{a_\ell})$ for some set of $1 \le a_i \le n$, i.e. a product of some coordinate axes. Thus the semistable set for $\chi$ is the preimage of some locus on $\A^k/\Gm^k$, as products of coordinate axes are stable under the $\Gm^k$-action. One can find the semistable loci for various $\Gm^{k-1}$-characters on $\A^k/\Gm^{k-1}$ by e.g. using the secondary fan. We claim that the GIT quotients in this case determine them for a general $(X, D)$.

\begin{thm}\label{thm:multiwindow}
    Let $X$ be a scheme with a map to $\A^1/\Gm$ given by an effective Cartier divisor $D$ on $X$, with $\X$ given by the following fiber product:

    \[\begin{tikzcd}
	{\mathcal{X}} & {\mathbb{A}^k/\Gm^k} \\
	X & {\mathbb{A}^1/\Gm.}
	\arrow[from=1-1, to=1-2]
	\arrow[from=1-1, to=2-1]
	\arrow["{e_{n_1, n_2, \dots, n_k}}", from=1-2, to=2-2]
	\arrow[from=2-1, to=2-2]
    \arrow["\lrcorner"{anchor=center, pos=0.125, rotate=0}, draw=none, from=1-1, to=2-2]
\end{tikzcd}\]

Then $\X$ is the quotient of a scheme $Y$ by $\Gm^{k-1}$, with the GIT quotients given by the corresponding GIT quotients if $X$ were replaced with $\A^1$ (and $Y$ with $\A^k$). More precisely, the GIT quotients in the $\A^1$ case are preimages of specific open substacks of $\A^k/\Gm^k$, and preimages of the same substacks give the GIT quotients of $\X$.

Furthermore, if we fix a norm on the cocharacter lattice of $\Gm^{k-1}$, then the HKKN strata coming from the Hilbert--Mumford procedure are the same.
\end{thm}

\begin{proof}
The kernel of the map $\Gm^k \to \Gm$ gives a short exact sequence of tori $\Gm^{k-1} \to \Gm^k \to \Gm$. Let us fix a section $a$ of $\Gm^k \to \Gm$, given by $a_1, a_2, \dots, a_k$, so $\sum a_in_i=1$. We then have $\A^k/\Gm^k=\A^k/(\Gm \times \Gm^{k-1})$, where the $\Gm$ acts via the section $a$, and the $\Gm^{k-1}$ is the inclusion of the kernel. Thus $\X$ is the quotient of $Y$ by $\Gm^{k-1}$, where $Y$ is given by the following fiber products:

\comment{
\[\begin{tikzcd}
	{Y} & {[\mathbb{A}^k/_{a_1, a_2, \dots, a_k}\mathbb{G}_m]} \\
	X & {[\mathbb{A}^1/\mathbb{G}_m]}
	\arrow[from=1-1, to=1-2]
	\arrow[from=1-1, to=2-1]
	\arrow["{e_{n_1, n_2, \dots, n_k}}", from=1-2, to=2-2]
	\arrow[from=2-1, to=2-2]
    \arrow["\lrcorner"{anchor=center, pos=0.125, rotate=0}, draw=none, from=1-1, to=2-2]
\end{tikzcd}\]
}

\[\begin{tikzcd}
	{Y} & {\mathbb{A}^k/_{a_1, a_2, \dots, a_k}\mathbb{G}_m} \\
    {\Tot(\EE)} & {\mathbb{A}^{k+1}/_{a_1, a_2, \dots, a_k,1}\mathbb{G}_m} \\
	X & {\mathbb{A}^1/\mathbb{G}_m}
	\arrow[from=1-1, to=1-2]
	\arrow[from=1-1, to=2-1]
	\arrow["{(y_1,\dots,y_k,y_1^{n_1}y_2^{n_2}\cdots y_k^{n_k})}", from=1-2, to=2-2]
	\arrow[from=2-1, to=2-2]
    \arrow[from=2-1, to=3-1]
	\arrow["{x_{k+1}}", from=2-2, to=3-2]
    \arrow[from=3-1, to=3-2]
    \arrow["\lrcorner"{anchor=center, pos=0.125, rotate=0}, draw=none, from=1-1, to=2-2]
    \arrow["\lrcorner"{anchor=center, pos=0.125, rotate=0}, draw=none, from=2-1, to=3-2]
\end{tikzcd}\]

The vertical arrow in the bottom square is given by total space of the vector bundle $\EE := \OO_X(a_1D)\oplus \cdots \oplus \OO_X(a_kD)$. The top square cuts out the locus where $\{y_1^{n_1}y_2^{n_2}\cdots y_k^{n_k}=s\}$. As such, $Y$ is a scheme.

\[\begin{tikzcd}
	Y & {\A^k/_{a_1, \dots, a_k}\Gm} \\
	\X & {\A^k/\Gm^k}
	\arrow[from=1-1, to=1-2]
	\arrow[from=1-1, to=2-1]
	\arrow["\lrcorner"{anchor=center, pos=0.125}, draw=none, from=1-1, to=2-2]
	\arrow[from=1-2, to=2-2]
	\arrow[from=2-1, to=2-2]
\end{tikzcd}\]

Fix a character $\chi$ of $\Gm^{k-1} \subset \Gm^k$. We claim that the scheme-theoretic HKKN stratification of $Y$ for $\chi$ (under the Hilbert--Mumford procedure) is determined by the case of $X=\A^1$, $Y=\A^k$, and that there is no dependence on $a$. Specifically, over $\A^k/\Gm^{k-1}$, we claim there is a stratification, such that pulling back this stratification yields the HKKN stratification on $Y$.

The key fact is that the map $Y \to \A^k/\Gm$ is just projecting to the coordinates of bundle $\EE$, with $\Gm^{k-1}$ acting fiberwise.

The HKKN stratification is determined by a procedure involving $1$-parameter subgroups $\lam \colon \Gm \to \Gm^{k-1}$. Given such a $\lam$, let us view it as a $\Gm^k$ cocharacter with integer weights $\lam(1), \dots, \lam(k)$. Then the fixed point set $Y^\lam$ is determined by the set of coordinates $1 \le i \le k$ such that $\lam(i)=0$. Namely, pulling back $(\prod_{\lam(i)=0} \A^1_i)/\Gm$ yields $Y^\lam$. The ``blade'' $Y_{\lam, Z}$ for a (set of) component(s) $Z$ of $Y^\lam$ is then the scheme-theoretic locus of $Y$ where the $\Gm$-action by $\lam$ limits to $Z$ at $0$. For $Z$ being the entire $X^\lam$, this locus is just the pullback of $(\prod_{\lam(i) \ge 0}\A^1_i)/\Gm$.

Finally, the stratification is constructed from an ordering of possible $(\lam, Z)$ via the value of the weight under $\lam$ of $\OO\langle \chi\rangle$ at $Z$, which is just $\langle \chi, \lam \rangle$. This ordering doesn't depend on the identity of $X$ or $Y$; only on the norm on the cocharacter lattice and the tuple $(n_1, \dots, n_k)$. Note that no step uses the section $a$.

So the HKKN stratification is specified by a pullback of a stratification on $\A^k/\Gm$ (or equivalently on $\A^k/\Gm^k$ since every stratum is a product of coordinate hyperplanes and coordinate hyperplane complements), and this pullback can be determined from the case $X=\A^1$.
\end{proof}



\begin{rem}
    The secondary fan for $\A^k/\Gm^{k-1}$ gives an analogous one for $Y/\Gm^{k-1}$, in the sense that it still determines how characters of $\Gm^{k-1}$ associate to semistable loci. In the notation of \cite[Section 14.3]{cls}, we have vectors $\beta_1, \dots, \beta_k \in \Z^{k-1}$ giving the rays of the secondary fan such that the $\beta_i$ are the columns of a matrix whose rows are a basis for the map $\Gm^{k-1} \subset \Gm^k$. Note that we have $n_1\beta_1+\cdots+n_k\beta_k=0$. Since the $\beta_i$ span $\mb{R}^{k-1}$ and a positive linear combination of them is zero, the fan they generate must be all of $\mb{R}^{k-1}$. So every character of $\Gm^{k-1}$ gives a nonempty semistable locus. In fact, each semistable locus is a product of $\A^1$ and $\Gm$, with the $\Gm$ corresponding to the $\beta_i$ that $\chi$ is in the cone of. 
\end{rem}

\begin{prop}\label{prop:strataAR}
    The HKKN strata above satisfy the conditions (A) and (R).
\end{prop}
\begin{proof}
    Each stratum $Z \subset S$ is the pullback of some coordinate hyperplanes on $\A^k/_{a_1, \dots, a_k} \Gm$. But if we set at least one $y_i$ to $0$, then the section $y_1^{n_1}\cdots y_k^{n_k}$ (which cuts out $D$) is also zero. So every $Z$ in $S$ is the total space of a sum of line bundles inside another sum of line bundles that are part of the direct sum decomposition of $\EE=\OO_X(a_1) \oplus \cdots \oplus \OO_X(a_k)$. It is clear that $S$ is a locally trivial affine bundle over $Z$, and $S$ is cut out by the coordinates $y_i$ where $\OO_X(a_i)$ is not part of the bundle making up $S$. This is a regular embedding.
\end{proof}

\begin{rem}
    It is not important to choose HKKN strata according to the Hilbert--Mumford procedure for window theory (as explained in \cref{subsect:GITconv}); however the proof of \cref{prop:strataAR} demonstrates that a stratum cannot be simply the point $0$.
\end{rem}

This result suggests a way to compute GIT quotients ``universally'' under certain conditions.

\begin{rem}\label{rem:nothingnew}

The multi-dimensional construction above gives a potentially new way via window theory to embed one root stack inside another root stack. However it appears that no new embeddings appear this way.

Consider  the root stack for $n_1$ arising from the diagram 
\[\begin{tikzcd}
	{\sqrt[n_1]{X/D}} & {(\A^1 \times \Gm^{k-1})/\Gm^k} \\
	\X & {\A^k/\Gm^k} \\
	X & {\A^1/\Gm}
	\arrow[from=1-1, to=1-2]
	\arrow[from=1-1, to=2-1]
	\arrow["\lrcorner"{anchor=center, pos=0.125}, draw=none, from=1-1, to=2-2]
	\arrow[from=1-2, to=2-2]
	\arrow[from=2-1, to=2-2]
	\arrow["p", from=2-1, to=3-1]
	\arrow["\lrcorner"{anchor=center, pos=0.125}, draw=none, from=2-1, to=3-2]
	\arrow["{e_{n_1, n_2, \dots, n_k}}", from=2-2, to=3-2]
	\arrow[from=3-1, to=3-2]
\end{tikzcd}\]

which comes from a GIT semistable locus that is pulled back from $\A^1 \times \Gm^k$.

This semistable locus corresponds to a window inside $\Coh(X)$. We have a choice of $w_i$ and $\lambda_i$, and we can say that the window is generated by $p^*\Coh(X)\langle c_1, \dots, c_k \rangle$ for vectors $c$ such that each $\lambda_i(c)$ lies within an interval $[w_i, w_i+\eta_i)$. In particular, we care about what the twist of the line bundle is when we pull back to $\sqrt[n_1]{X/D}$, which is determined by the value of $c_1$.

Now suppose we embed one window inside another. WLOG let the embedded window correspond to the semistable locus for $\Gm \times \A^1 \times \Gm^{k-2}$. We want to see which vectors $c$ are part of the second window.

One way to set our $\lam_i$s to get an embedding is to set $\lambda_1=(-n_2, n_1, 0, \dots, 0)$ for the first window and $\lambda'_1=(n_2, -n_1, 0, \dots, 0)$ for the second window, with all the other $\lambda_i$ equal. Then we want $-n_2c_1+n_1c_2$ to be in a range of length $n_1$ for the first window, and to be in a range of length $n_2$ in the second window (with the same constraints for the other $\lam_i$). The set of $c_1$ in the first window but not the second correspond to a consecutive range of values for $n_2c_1$, i.e. $c_1$ changes by $(n_2)^{-1}$ mod $n_1$. This value is exactly the value $a$ from Section 2. So the embedding is essentially the same as before (as the line bundle twist is the same).

\end{rem}

\begin{exmp}
Suppose we have $(n_1, n_2, n_3)=(7,3,10)$. We can set $\lam_1=(-3, 7, 0)$ (so $\lambda_1'=(3, -7,0)$) and $\lam_2=(-10, 0, 7)$.

First let us find $(c_1, c_2, c_3)$ such that $-3c_1+7c_2 \in [0, 7)$ and $-10c_1+7c_3 \in [0, 7)$. Note that two vectors $c$ that differ by a multiple of $(7, 3, 10)$ are equivalent since they give the same character the kernel $\Gm^2$ of the map $\Gm^3 \xto{7,3,10} \Gm$. The set of vectors is \{000 112 213 325 426 538 639\}.

Now let us find $(c_1, c_2, c_3)$ such that $-3c_1+7c_2 \in [0, 3)$ and $-10c_1+7c_3 \in [0, 7)$. The set of vectors is \{000 213 426\}.

The missing items are 112, 325, 538, 639. Only looking at $c_1$, these are in an arithmetic sequence $5, 3, 1, 6$ modulo $7$, with difference $5$.
\end{exmp}

\subsection{Iterated root stacks}
The analogous result to \cref{thm:multiwindow} holds in the case that we take the fiber product of a general map $\A^n/\Gm^n \to \A^m/\Gm^m$ with a map $X \to \A^m/\Gm^m$. This construction has a more straightforward interpretation in the case of an iterated root stack.
\begin{defn}
    Let $r$ be a multi index $(r_1, r_2, \dots, r_n)$, and suppose we have $n$ maps $X \to \A^1/\Gm$ given by a tuple of effective Cartier divisors $D=(D_1, \dots, D_n)$. Suppose also that every intersection is also ``effective'', i.e. for any intersection $D_{i_1} \cap D_{i_2} \cap \dots \cap D_{i_k}$, the intersection with another $D_j$ is effective (one can think of this as generalizing the notion of a snc divisor). Then the iterated root stack $\sqrt[r]{X/D}$ is given by the following fiber product:
    \[\begin{tikzcd}
	{\sqrt[r]{X/D}} & {\A^n/\Gm^n} \\
	X & {\A^n/\Gm^n}.
	\arrow[from=1-1, to=1-2]
	\arrow[from=1-1, to=2-1]
	\arrow["\lrcorner"{anchor=center, pos=0.125}, draw=none, from=1-1, to=2-2]
	\arrow["{e_r}"', from=1-2, to=2-2]
	\arrow[from=2-1, to=2-2]
\end{tikzcd}\]
\end{defn}

\begin{rem}
    The use of ``iterated'' is somewhat of a misnomer since the order of $r_i$ and $D_i$ does not matter. In fact, the iterated root stack is just the fiber product over $X$ of the ordinary root stacks. By ordering the ordinary root stack constructions, though, we get an iterated root stack.
\end{rem}

We claim that there is a periodic embedding of the $r$th root stack inside the $s$th root stack (fixing a tuple $D$) if $s_i \ge r_i$ for all $i$.

First we will give a proof by iterating the ordinary root stack case. We will write the $n=2$ case, as the general case is similar. First assume we just have $r_1=r_2=1$. Then we have an SOD

\begin{align*}\Coh(\sqrt[s_1,s_2]{X/(D_1,D_2)})=&\langle \Coh(\sqrt[s_1]{X/D_1}), \Coh(\sqrt[s_1]{D_2/(D_1 \cap D_2)}), \dots, \Coh(\sqrt[s_1]{D_2/(D_1 \cap D_2)})\rangle\\
=&\langle \langle \Coh(X), \Coh(D_1), \Coh(D_1), \dots, \Coh(D_1)\rangle,\\
&\langle \Coh(D_2), \Coh(D_1 \cap D_2), \dots, \Coh(D_1 \cap D_2)\rangle,\\
&\dots,\\
&\langle \Coh(D_2), \Coh(D_1 \cap D_2), \dots, \Coh(D_1 \cap D_2)\rangle \rangle.\end{align*}
as well as an SOD 
\begin{align*}\Coh(\sqrt[s_1,s_2]{X/(D_1,D_2)})=&\langle \Coh(\sqrt[s_2]{X/D_2}), \Coh(\sqrt[s_2]{D_1/(D_1 \cap D_2)}), \dots, \Coh(\sqrt[s_2]{D_1/(D_1 \cap D_2)})\rangle\\
=&\langle \langle \Coh(X), \Coh(D_2), \Coh(D_2), \dots, \Coh(D_2)\rangle,\\
&\langle \Coh(D_1), \Coh(D_1 \cap D_2), \dots, \Coh(D_1 \cap D_2)\rangle,\\
&\dots,\\
&\langle \Coh(D_1), \Coh(D_1 \cap D_2), \dots, \Coh(D_1 \cap D_2)\rangle \rangle.\end{align*}

We know that $\OO\langle 0, s_2 \rangle$ preserves the components in the first SOD (namely $\Coh(\sqrt[s_1]{X/D_1})$ and the copies of $\Coh(\sqrt[s_1]{D_2/(D_1 \cap D_2)}$), while $\OO\langle s_1, 0 \rangle$ preserves the subcomponents within each component. Similarly $\OO\langle s_1, 0 \rangle$ preserves the components in the second SOD, while $\OO\langle 0, s_2 \rangle$ preserves the subcomponents in each component. It turns out that really we have a rectangular grid of components making up an SOD, so the components above are embedded the same way in the two linear SODs (see \cite[Theorem 4.9]{bls}; though they assume snc divisor and take quasicoherent sheaves, the condition of checking the two embeddings above agree is the same). So all components are preserved by $\Gm^2$-twists by $\langle s_1, 0 \rangle$ and $\langle 0, s_2\rangle$.

Then by a similar argument to the previous section, to find $\Coh(\sqrt[r]{X/D})$ inside $\Coh(\sqrt[s]{X/D})$, we can use the fact that the left orthogonal is made up of a subset of the components we have already defined. We know under mutation that they change by a $\Gm^2$-twist, and so taking $\lcm(s_1, s_2)$ twists will yield the original SOD.

\begin{rem}
    We can think about this construction in the context of \cref{subsect:multiwindow}. Consider the global quotient stack $\X$ defined as the following fiber product:

\[\begin{tikzcd}
	{\X} & {\A^n/\Gm^n \times \A^n/\Gm^n} \\
	X & {\A^n/\Gm^n}.
	\arrow[from=1-1, to=1-2]
	\arrow[from=1-1, to=2-1]
	\arrow["\lrcorner"{anchor=center, pos=0.125}, draw=none, from=1-1, to=2-2]
	\arrow["{e_{r,s}}"', from=1-2, to=2-2]
	\arrow[from=2-1, to=2-2]
\end{tikzcd}\]

Then the secondary fan for $\X$ looks like a product of the $n$ coordinate axes, separating $\mb{R}^n$ into orthants. Crossing each coordinate hyperplane corresponds to a wall-crossing from the $r_i$th root stack to the $s_i$th root stack on the divisor $D_i$. So crossing all $n$ hyperplanes gives a variation of GIT embedding of the derived category of one iterated root stack inside the other.
\end{rem}

\section{Universal categorical representations}\label{sect:categorical}
Key theorems in \cite{BZFN,BZNP} state that whenever $U\xto{p_U} Y\xfrom{p_V}V$ are maps of stacks with $U,Y$ perfect,
\begin{enumerate}
    \item if moreover $V$ is perfect then we have an equivalence \[\Perf(U)\otimes_{\Perf(Y)}\Perf(V)\simeq \Perf(U\times_YV),\]
    and
    \item and if moreover $p_U:U\to Y$ is a proper relative DM stack, and $p_V$ is locally finite type, then we have an equivalence,
    \[\Hom_{\Perf(Y)}(\Perf(U),\Coh(V))\simeq \Coh(U\times_YV).\]
\end{enumerate}
The second statement as stated in \cite{BZNP} requires that $p_U$ is an algebraic space. This condition is weakened in \cref{cor:dmhom}.

Using techniques from categorical representation theory, we will prove root stack SOD results analogously to \cref{thm:theirSOD} for all perfect stacks (in the case of $\Perf$), and all locally finite type stacks (in the case of $\Coh$).
We will do this by defining a categorical root stack operation, taking as an input any category over $\ag$, which generalizes the geometric root stack in a suitable sense. We then prove that all categorical root stacks admit an SOD. Integral transform formulae provide geometric interpretations of the SOD components.

\subsection{Categories over stacks}
In this subsection, we will make precise what we mean by a category over $\A^1/\Gm$. Let $Y$ be any smooth stack. Recall $\Perf(Y)$ is naturally monoidal under $\otimes$.

\begin{defn}
    A category over $Y$ is a module category for $\Perf(Y)$. We write
    \[\Perft(Y):=\Perf(Y)-\Mod(\st)\]
    for the 2-category of categories over $Y$.
\end{defn}

Let $X$ be a stack living over $Y$.
We define the object $\ul{\Perf}(X)\in\Perft(Y)$
to be $\Perf(X)$ with the natural $\Perf(Y)$-module structure (via pullback). We likewise define the object $\ul{\Coh}(X)$ to be $\Coh(X)$ with its natural $\Perf(Y)$-module structure.

These upgrade to functors of 2-categories
\begin{align*}
\ul{\Perf},\ul{\Coh}&:\Stack^{\op}_{/Y}\to\Perft(Y)
\end{align*}
(here Stack is just a category of stacks such that these functors are well-defined (e.g. pullbacks preserve coherence); our application will only involve the diagrams in \cref{prop:compatiblesmall}).
We also denote by
\begin{align*}
\Gamma(Y,-)&:\Perft(Y)\to \st
\end{align*}
the ``global sections functor'' which forgets the module structure.

We now define two pullback functors for categories over stacks. We begin by recalling some classical constructions which we will categorify to our setting. Let $A\to\tilde{A}$ be a homomorphism of commutative rings, defining $f:\Spec(\tilde{A})\to\Spec(A)$. Then there are two natural pullback functors on quasi-coherent sheaves:
\begin{align*}
    f^*: A-\Mod &\to\tilde{A}-\Mod\\
    M&\mapsto \tilde{A}\otimes_{A}M
\end{align*}
and 
\begin{align*}
    f^!: A-\Mod &\to\tilde{A}-\Mod\\
    M&\mapsto \Hom_A(\tilde{A},M).
\end{align*}
These are the left and right adjoints to the functor $f_*$.

Suppose now that $f:\tilde{Y}\to Y$ is a map of stacks. Pullback of perfect complexes defines a monoidal functor,
$\Perf(Y)\to\Perf(\tilde{Y})$.
We define
\begin{align*}
    f^*: \Perft(Y) &\to\Perft(\tilde{Y})\\
    \mc{M}&\mapsto \Perf(\tilde{Y})\otimes_{\Perf(Y)}\mc{M}
\end{align*}
and 
\begin{align*}
    f^!: \Perft(Y) &\to\Perft(\tilde{Y})\\
    \mc{M}&\mapsto \Hom_{\Perf(Y)}(\Perf(\tilde{Y}),\mc{M}).
\end{align*}
These are the left and right adjoints to the functor $f_*$, which is defined by pulling back the module structure under the monoidal functor.

\subsection{The categorical root stack functor}
Let $e_n:\widetilde{\ag}\to\ag$, where $\wtag = \ag$ (as abstract stacks). The notation is to distinguish the source and target of $e_n$. We also have the $n$th power map $h_n \colon \wtpg \to \pg$.

\begin{defn}
    We call the functors
        $$\sqrt[n,*]{-}:=e_n^*:\Perft(\ag)\to\Perft(\wtag)$$
        and
        $$\sqrt[n,!]{-}:=e_n^!:\Perft(\ag)\to\Perft(\wtag)$$
    the $\ast$-categorical and $!$-categorical root stack functors, respectively. We also define the functor $$\sqrt[n]{-}:=h_n^* \colon \Perft(\pg) \to \Perft(\wtpg).$$
\end{defn}
Note that $h_n^* \colon \Perft(\pg) \to \Perft(\pg)$ can be written as $\Perf(\wtpg) \otimes_{\Perf(\pg)} (-)$, and that as a $\Perf(\pg)$-module, $\Perf(\wtpg)$ is a direct sum of $n$ copies of $\Perf(\pg)$.

\subsection{Compatibility with geometric root stack}
In order to relate our categorical root stack functors to the geometric root stack construction, we require some integral transform results.

\begin{prop}\label{prop:compatiblesmall}
    There are naturally commuting diagrams of 2-categories as follows.
    \[\begin{tikzcd}
    	{\Stack^{\operatorname{perfect}}_{/\ag}} & {\Stack^{\operatorname{perfect}}_{/\wtag}} \\
    	{\Perft(\ag)} & {\Perft(\wtag)}
    	\arrow["{\sqrt[n]{-}}", from=1-1, to=1-2]
    	\arrow["{\ul{\Perf}}", from=1-1, to=2-1]
    	\arrow["{\ul{\Perf}}", from=1-2, to=2-2]
    	\arrow["{\sqrt[n,*]{-}}"', from=2-1, to=2-2]
    \end{tikzcd}
    \;\;\;
    \begin{tikzcd}
    	{\Stack^{\textup{locally finite type}}_{/\ag}} & {\Stack^{\textup{locally finite type}}_{/\wtag}} \\
    	{\Perft(\ag)} & {\Perft(\wtag)}
    	\arrow["{\sqrt[n]{-}}", from=1-1, to=1-2]
    	\arrow["{\ul{\Coh}}", from=1-1, to=2-1]
    	\arrow["{\ul{\Coh}}", from=1-2, to=2-2]
    	\arrow["{\sqrt[n,!]{-}}"', from=2-1, to=2-2]
    \end{tikzcd}\]
\end{prop}

\begin{proof}
    The statement involving $\ul{\Perf}$ follows from \cite{BZFN}. The statement involving $\ul{\Coh}$ follows from \cite{BZNP} (modified in \cref{subsect:bznpextend}, as $\wtag \to \ag$ is a relative DM stack; see \cref{cor:dmhom}).
\end{proof}

\subsection{SODs for categorical root stacks}

We use the following (non-Cartesian) diagram to construct an SOD of $\Perf(\wtag)$.

\begin{equation}\label{equation:square}
\begin{tikzcd}
	{\wtpg} & {\wtag} \\
	{\pt/\Gm} & {\A^1/\Gm}.
	\arrow["i", from=1-1, to=1-2]
	\arrow["h_n", from=1-1, to=2-1]
	\arrow["{e_n}"', from=1-2, to=2-2]
	\arrow["i", from=2-1, to=2-2]
\end{tikzcd}
\end{equation}

First we need a preliminary calculation on maps between objects of $\Perf(\ag)$.

\begin{lem}\label{lem:a1gmhoms}
Let $\OO$ be the structure sheaf of $\ag$, and $\OO_0=\delta_0$ be the structure sheaf for $0/\Gm$. We have \begin{align*}\Hom(\OO, \OO\langle i \rangle)=&\begin{cases}
    k & \text{if } i \ge 0 \\
    0 & \text{otherwise}
\end{cases},\\\Hom(\OO, \OO_0 \langle i \rangle)=&\begin{cases}
    k & \text{if } i =0 \\
    0 & \text{otherwise} \end{cases},\\
    \Hom(\OO_0, \OO \langle i \rangle)=&\begin{cases}
    k[-1] & \text{if } i =1 \\
    0 & \text{otherwise} \end{cases},\\
    \text{and}\\
    \Hom(\OO_0, \OO_0 \langle i \rangle)=&\begin{cases}
    k & \text{if } i =0 \\
    k[-1] & \text{if } i =1 \\
    0 & \text{otherwise} \end{cases}.\end{align*}
\end{lem}
\begin{proof}
    The calculations follow from taking the resolution $0 \to \OO\langle 1 \rangle \to \OO \to \OO_0 \to 0$ in $\Perf(\ag)$.
\end{proof}

In the rest of \cref{sect:categorical}, any grading shift $\langle i \rangle$ on a functor is understood to be with respect to the $\Gm$-action on $\wtpg$ and $\wtag$ (as opposed to $\pg$ and $\ag$; note that the weights for the two actions differ by a factor of $n$).

\begin{lem}
    The maps $e_n^*\langle i \rangle \colon \Perf(\ag) \to \Perf(\wtag)$ and $i_*h_n^*\langle i \rangle \colon \Perf(\pg) \to \Perf(\wtag)$ are fully faithful.
\end{lem}
\begin{proof}
    We check on $\OO \langle j \rangle$ and $\delta_0 \langle j\rangle$, which compactly generate the respective categories. The maps between such objects are either $0$ or $k$, and are preserved under $e_n^*\langle i \rangle$ or $i_*h_n^*\langle i \rangle$ by \cref{lem:a1gmhoms}.
\end{proof}

\begin{thm}\label{thm:a1gmsod}
    Consider the following list of functors.
    \begin{align*}
        e_n^* =: I_1 &:\Perf(\ag)\to \Perf(\wtag)\\
        i_*h_n^* =: I_2 &:\Perf(\pg)\to \Perf(\wtag)\\
        i_*h_n^*\langle1\rangle =: I_3 &:\Perf(\pg)\to \Perf(\wtag)\\
        \cdots\\
        i_*h_n^*\langle n-2\rangle =: I_n &:\Perf(\pg)\to \Perf(\wtag).
    \end{align*}

    This defines an inclusion SOD of $\Perf(\wtag)$ with the following properties.
    \begin{enumerate}
        \item\label{enumerate:admissible}  The SOD is infinitely admissible. That is, each inclusion functor admits all repeated adjoints.
        \item\label{enumerate:gluing} The gluing functors $I_j^LI_i$ are zero between nonadjacent components, and are otherwise $i^* \colon \Perf(\A^1/\Gm) \to \Perf(\pt/\Gm)$ and $\id \colon \Perf(\pt/\Gm) \to \Perf(\pt/\Gm)$.
        \item\label{enumerate:periodic} If we write $\mc{D}$ for the subcategory generated by all components but the first, then the SOD $\Perf(\wtag)=\langle \Perf(\ag),\mc{D}\rangle$ is $2n$-periodic.
        \item\label{enumerate:stable} Each inclusion functor and their repeated adjoints are naturally $\Perf(\ag)$-module functors.
    \end{enumerate}
\end{thm}

\begin{proof} Note that $\Perf(\wtag)$ is generated by $\OO\langle i \rangle$ for all $i$, and $e_n^*\Perf(\A^1/\Gm)$ is generated by $\OO \langle in \rangle_{i \in \Z}$, while $i_*h_n^*\Perf(\pt/\Gm)\langle k \rangle$ is generated by $\OO_0\langle in+k \rangle_{i \in \Z}$. The SOD then follows from \cref{lem:a1gmhoms}.

To prove Statement~\ref{enumerate:admissible} it suffices to show that $e_n^*,h_n^*$ and $i_*$ admit all repeated adjoints. First, $h_n^*$ is left and right adjoint to $h_{n,*}$ so we have all repeated adjoints. A standard calculation shows that
\[i^*\adjto i_*\adjto i^*\langle-1\rangle[-1],\]
so $i_*$ has all repeated adjoints. The functor $e_n^*$ admits a right adjoint $e_{n,*}$, as this pushforward preserves coherent ($\iff$ perfect) complexes, by properness of $e_n$. Another standard calculation shows that
\[e_n^*\adjto e_{n,*}\adjto e_n^*\langle-n+1\rangle,\]
so $e_n^*$ admits all repeated adjoints.

The calculations of statement~\ref{enumerate:gluing} are straightforward.

The proof of statement~\ref{enumerate:periodic} is identical in structure to that of \cite[Theorem 4.3]{BD} and we reproduce it here. The above calculation shows that we also have an SOD
\begin{align*}
\Perf(\wtag)&=\langle i_*h_n^*\Perf(\pt/\Gm), i_*h_n^*\Perf(\pt/\Gm)\langle 1 \rangle, \dots, i_*h_n^*\Perf(\pt/\Gm)\langle n-2\rangle , e_n^*\Perf(\A^1/\Gm)\langle n-1 \rangle\rangle\\
&= \langle\mc{D},e_n^*\Perf(\ag)\langle n-1\rangle\rangle.
\end{align*}
So
\begin{align*}
    \Perf(\wtag)&=\langle e_n^*\Perf(\ag),\mc{D}\rangle\\
    &=\langle\mc{D},e_n^*\Perf(\ag)\langle n-1\rangle\rangle\\
    &=\langle e_n^*\Perf(\ag)\langle n-1\rangle,\mc{D}\langle n-1\rangle\rangle.
\end{align*}
This is the original SOD twisted by $\langle n-1 \rangle$, so $2n$ total mutations returns the original SOD.

For statement~\ref{enumerate:stable}, it suffices to produce $\Perf(\ag)$-linear structures on $i_*,h_{n,*},e_{n,*}$, and its adjoints $i^*,h_n^*,e_n^*$. The pullbacks are easy to do. The linearity structure on pushforward functors is exactly the base change formula. For example, given $\mc{F}\in\Perf(\ag)$,
\[e_n^*\mc{F}\otimes (i_*-)\simeq i_*((i^*e_n^*\mc{F})\otimes-).\]
\end{proof}

\begin{cor}\label{cor:abstractSODstar}
    Let $\Phi$ be an object of $\Perft(\A^1/\Gm)$. Define 
    \[\Psi^* :=\Perf(\pg)\otimes_{\Perf(\ag)} \Phi\] and consider the (lax-commuting) diagram of categories,
\begin{equation}\label{eqn:starsquare}
\begin{tikzcd}
	{\sqrt[n]{\Psi^*}} & {\sqrt[n,*]{\Phi}} \\
	{\Psi^*} & {\Phi}.
	\arrow["i_*", from=1-1, to=1-2]
	\arrow["h_n^*", from=2-1, to=1-1]
	\arrow["{e_n^*}"', from=2-2, to=1-2]
	\arrow["i_*", from=2-1, to=2-2]
\end{tikzcd}
\end{equation}
    Then the underlying category $\Gamma(\wtag,\sqrt[n,*]{\Phi})$ admits a $2n$-periodic SOD
    \[\Gamma(\wtag,\sqrt[n,*]{\Phi})=\langle e_n^*\Phi,i_*h_n^* \Psi^*,(i_*h_n^* \Psi^*)\langle1\rangle,\dots,(i_*h_n^* \Psi^*)\langle n-2\rangle\rangle.\]

    In other words, the components are $\Phi$ and $\Psi^*$ and the inclusion functors are obtained by applying $-\otimes_{\Perf(\ag)}\Phi$ to the $I_j$. In particular, the base-changed analogs of the first three statements in \cref{thm:a1gmsod} hold as well.
\end{cor}

\begin{proof}[Proof of \cref{cor:abstractSODstar}]
These follow from \cref{thm:a1gmsod} together with the definition of categorical root stack and the first part of \cref{thm:preserve}.
\end{proof}

We can also use \cref{thm:a1gmsod} to produce an SOD for $\Gamma(\wtag,\sqrt[n,!]{\Phi})$ (which is naturally a projection SOD with maps $(- \circ I_j)$). However this SOD has the components reversed compared to the one for $\sqrt[n, !]{\Phi}$. We use another SOD of $\Perf(\wtag)$ (naturally written as a projection SOD) to give our desired SOD of $\sqrt[n,!]{\Phi}$.

\begin{thm}\label{thm:a1gmsod2}
    Consider the following list of functors.
    \begin{align*}
        e_{n,*} =: P_1 &:\Perf(\wtag)\to \Perf(\ag)\\
        h_{n,*}i^* =: P_2 &:\Perf(\wtag)\to \Perf(\pg)\\
        h_{n,*}i^*\langle1\rangle =: P_3 &:\Perf(\wtag)\to \Perf(\pg)\\
        \cdots\\
        h_{n,*}i^*\langle n-2\rangle =: P_n &:\Perf(\wtag)\to \Perf(\pg).
    \end{align*}

    This defines an SOD of $\Perf(\wtag)$ via projection with the following properties.
    \begin{enumerate}
        \item\label{enumerate:admissible2}  The SOD is infinitely admissible. That is, each projection functor admits all repeated adjoints.
        \item\label{enumerate:gluing2} The gluing functors $P_iP_j^R$ are zero between nonadjacent components, and are otherwise $i_* \colon \Perf(\pg) \to \Perf(\ag)$ and $\id \colon \Perf(\pt/\Gm) \to \Perf(\pt/\Gm)$.
        \item\label{enumerate:periodic2} If we write $\mc{D}$ for the subcategory generated by all components but the first, then the SOD $\Perf(\wtag)=\langle \Perf(\ag),\mc{D}\rangle$ is $2n$-periodic.
        \item\label{enumerate:stable2} Each projection functor and their repeated adjoints are naturally $\Perf(\ag)$-module functors.
    \end{enumerate}
\end{thm}

\begin{proof}
    The proof is similar to that of \cref{thm:a1gmsod}.
\end{proof}

\begin{cor}\label{cor:abstractSODshriek}
    Let $\Phi$ be an object of $\Perft(\A^1/\Gm)$. Define 
    \[\Psi^! :=\Hom_{\Perf(\ag)}(\Perf(\pg),\Phi)\] and consider the (lax-commuting) diagram of categories,
\begin{equation}\label{eqn:shrieksquare}
\begin{tikzcd}
	{\sqrt[n]{\Psi^!}} & {\sqrt[n,!]{\Phi}} \\
	{\Psi^!} & {\Phi}.
	\arrow["-\circ i^*", from=1-1, to=1-2]
	\arrow["-\circ h_{n,*}", from=2-1, to=1-1]
	\arrow["-\circ P_1"', from=2-2, to=1-2]
	\arrow["- \circ i^*", from=2-1, to=2-2]
\end{tikzcd}
\end{equation}
    Then the underlying category $\Gamma(\wtag,\sqrt[n,!]{\Phi})$ admits a $2n$-periodic SOD
    \[\Gamma(\wtag,\sqrt[n,!]{\Phi})=\langle (-\circ P_1)\Phi,(-\circ P_2) \Psi^!,(-\circ P_3) \Psi^!,\dots,(-\circ P_n) \Psi^!\rangle.\]

    In other words, the components are $\Phi$ and $\Psi^!$ and the inclusion functors are $(- \circ P_j)$. In particular, the base-changed analogs of the first three statements in \cref{thm:a1gmsod2} hold as well.
\end{cor}
\begin{proof}
    The proof is similar to that of \cref{cor:abstractSODstar}. In this case we need to use the last part of \cref{thm:preserve}.
\end{proof}

\begin{rem}\label{rem:projectionSODcategoricalrootstack}
    We can also get SODs via projection of (omitting global sections notation for simplicity) $\sqrt[n, !]{\Phi}$ and $\sqrt[n, *]{\Phi}$ by applying \cref{thm:preserve} to \cref{thm:a1gmsod} and \cref{thm:a1gmsod2}, respectively.
\end{rem}

\subsection{SODs for root stacks}
We now deduce the sought-for SOD for geometric root stacks.

\begin{thm}\label{thm:categoricalSOD}
    Let $X$ be a perfect stack with a map $f:X\to \A^1/\Gm$ (corresponding to $D$). Then we have an SOD
    $$\Perf(\sqrt[n]{X/D})=\langle e_n^*\Perf(X), i_*h_n^*\Perf(D), \dots, i_*h_n^*\Perf(D)\langle n-2 \rangle \rangle,$$.
    
    If (instead) $X$ is a locally finite type $\ag$-stack, then we have the SOD $$\Coh(\sqrt[n]{X/D})=\langle e_n^*\Coh(X), i_*h_n^*\Coh(D), \dots, i_*h_n^*\Coh(D)\langle n-2 \rangle \rangle.$$

    In both cases, the gluing functors are zero between nonadjacent components, and are otherwise $i^*$ (for the first two components) and abstractly $\id$ (for the other adjacent pairs).
\end{thm}

\begin{proof}
    The first statement follows from \cref{cor:abstractSODstar} by taking $\Phi=\ul{\Perf}(X)$. Note that we use \cref{prop:compatiblesmall} to see that the relevant diagram (\ref{eqn:starsquare}) of categories is
    \[\begin{tikzcd}
	{\Perf(\sqrt[n]{\OO_D(D)})} & {\Perf(\sqrt[n]{X/D})} \\
	{\Perf(D)} & {\Perf(X).}
	\arrow["i_*", from=1-1, to=1-2]
	\arrow["h_n^*", from=2-1, to=1-1]
	\arrow["{e_n^*}"', from=2-2, to=1-2]
	\arrow["i_*", from=2-1, to=2-2]
    \end{tikzcd}\]

    The second statement follows from \cref{cor:abstractSODshriek} by taking $\Phi = \ul{\Coh}(X)$, via \cref{prop:compatiblesmall} in a similar manner. We compute the square (\ref{eqn:shrieksquare}) to be
    \[\begin{tikzcd}
    	{\Coh(\sqrt[n]{\OO_D(D)})} & {\Coh(\sqrt[n]{X/D})} \\
    	{\Coh(D)} & {\Coh(X)}.
    	\arrow["-\circ i^*", from=1-1, to=1-2]
    	\arrow["-\circ h_{n,*}", from=2-1, to=1-1]
    	\arrow["-\circ e_{n,*}"', from=2-2, to=1-2]
    	\arrow["- \circ i^*", from=2-1, to=2-2]
    \end{tikzcd}\]
    We explicitly calculate the functors in this case, using the integral transform identifications of \cite{BZNP}. We claim that
    \[-\circ e_{n,*}:\Coh(X)\to \Coh(\sqrt[n]{X/D})\]
    is identified with $e_n^*$. Indeed, if we unravel the above, we have
    \begin{align*}
        \Coh(X)&\overset{\mc{K}\mapsto\mc{K}\otimes f^*-}{\isomto} \Hom_{\Perf(\ag)}(\Perf(\ag,\Coh(X)))\\
        &\xto{-\circ e_{n,*}} \Hom_{\Perf(\ag)}(\Perf(\wtag,\Coh(X)))\\
        &\overset{\tilde{\mc{K}}\mapsto e_{n,*}(\tilde{\mc{K}}\otimes \tilde{f}^*-)}{\isomfrom} \Coh(\sqrt[n]{X/D}).
    \end{align*}
    We compute the composition of the first two arrows using base change and the projection formula:
    \begin{align*}
        \mc{K}&\mapsto \mc{K}\otimes (f^*e_{n,*})\\
        &\simeq \mc{K}\otimes (e_{n,*}\tilde{f}^*-)\\
        &\simeq e_{n,*}((e_n^*\mc{K})\otimes\tilde{f}^*-).
    \end{align*}
    Further composing with the inverse of the final arrow gives $\mc{K}\mapsto e_n^*\mc{K}$ as claimed.

    Similar calculations work for $-\circ h_{n,*} \simeq h_n^*:\Coh(D)\to \Coh(\sqrt[n]{\OO_D(D)})$, and to show that $-\circ i^* \simeq i_*$. The latter calculation shows that the nontrivial gluing functor $\Coh(X) \to \Coh(D)$ is $i^*$.
\end{proof}

\subsection{Embeddings of derived categories of root stacks}

Using the gluing functor calculation of \cref{thm:categoricalSOD}, we are in the same situation as in \cref{subsect:abstractgluing1}. So we have a categorical analog of the embeddings $\Coh(\sqrt[m]{X/D})\inclto\Coh(\sqrt[n]{X/D})$.

\begin{thm}\label{thm:categoricalgluing}
    Consider the SOD of $\Gamma(\ag,\sqrt[n,*]{\Phi})$ of \cref{cor:abstractSODstar}. If we take any set of elements $0 \le a_1 < a_2 < \dots < a_{m-1} \le n-2$ and mutate each $(i_*h_n^*\Psi^*)\langle a_i \rangle$ to the left, then the resulting category (generated by $e_n^*\Phi$ and the left-mutated versions of the $(i_*h_n^*\Psi^*)\langle a_i\rangle$) is equivalent to $\Gamma(\ag,\sqrt[m,*]{\Phi})$. So there is an embedding,
    \[\Gamma(\ag,\sqrt[m,*]{\Phi})\inclto \Gamma(\ag,\sqrt[n,*]{\Phi}).\]
    Moreover, the resulting SOD,
    \[\Gamma(\ag,\sqrt[n,*]{\Phi})=\langle\Gamma(\ag,\sqrt[m,*]{\Phi}),\mc{D}\rangle\]
    is $2n$-periodic.
\end{thm}

We give a proof that uses the structure of $\Coh(\ag)$; this proof will also make clear that each embedding is $2n$-periodic under mutation. Note also that we don't need any reconstruction result like \cref{thm:linearSOD}, unlike in the proof of \cref{thm:mutatearb1}.

\begin{lem}\label{lem:arbadm}
    For a set of residues $b_1, \dots, b_{n-m}$ mod $n$, the subcategory of $\Coh(\ag)$ generated by the set of $i_*h_n^*\Coh(\pg)\langle b_i \rangle$ is admissible. The left orthogonal to this subcategory is generated by $ \OO \langle j \rangle$ for $j \not\equiv b_i$ mod $n$, while the right orthogonal is generated by $ \OO \langle j \rangle$ for $j \not\equiv b_i+1$ mod $n$.

    Furthermore, the left (or right) orthogonal, as a module over $\Coh(\ag)$ acting via $e_n^*\langle i \rangle$ (for some $i$ such that $e_n^*\Coh(\ag)\langle i \rangle$ is a subcategory of the left/right orthogonal), is equivalent to $\Coh(\sqrt[m]{(\ag)/(\pg)})$.
\end{lem}
\begin{proof}
    Admissibility follows after checking the left and right orthogonals, since the resulting set of objects generates the category. The purported left and right orthogonals are indeed orthogonal using \cref{lem:a1gmhoms}, so they are each the entire orthogonal by generation.

    For the last part, WLOG suppose $i=0$ and suppose the residues of the possible $\OO \langle j \rangle$ are $a_0=0 <a_1 <  \dots < a_{m-1} < n$. We note that there is an abstract equivalence with $\Coh(\ag)$ that sends $\OO\langle kn+a_i\rangle$ to $\OO \langle km+i \rangle$ that intertwines the $e_n^*\Coh(\ag)$-action with the action of $e_m^*\Coh(\ag)$, giving the result.
\end{proof}

One can even match the copies of $\Coh(\pg)$ after mutation with \cref{thm:categoricalgluing}.
\begin{prop}
    Take $b_1, \dots, b_{n-m} \in [0, n-2]$, and let $a_1, \dots, a_{m-1}$ be the missing residues (in order). Then in the SOD
    $$\Coh(\A^1/\Gm)=\langle e_n^*\Coh(\A^1/\Gm), i_*h_n^*\Coh(\pt/\Gm), i_*h_n^*\Coh(\pt/\Gm)\langle 1 \rangle, \dots, i_*h_n^*\Coh(\pt/\Gm)\langle n-2\rangle \rangle,$$
    if we mutate each $i_*h_n^*\Coh(\pg)\langle a_i \rangle$ to the left (so that they remain in order but are to the left of the $i_*h_n^*\Coh(\pg)\langle b_j \rangle$), the mutated $i_*h_n^*\Coh(\pg)\langle a_i \rangle$ is generated by the following object and its twists by $\langle n \rangle$:
    \begin{itemize}
        \item $\cone(\OO\langle a_i+1\rangle \to \OO)$ for $i=1$; and
        \item $\cone(\OO\langle a_i+1\rangle \to \OO\langle a_{i-1}+1 \rangle)$ for $i>1$.
    \end{itemize}
\end{prop}
\begin{proof}
    This follows by induction using \cref{lem:a1gmhoms}.
\end{proof}

\begin{proof}[Proof of \cref{thm:categoricalgluing}]
    We use the method of \cref{thm:categoricalSOD} applied to \cref{lem:arbadm}.
\end{proof}

\subsection{Large categories}
We briefly discuss here what changes if we work with large categories (like $\QC$ and $\IndCoh$) instead of small categories (like $\Perf$ and $\Coh$). 

All statements involving $\otimes$ and perfect stacks generalize directly by taking ind-completions. However, there is a divergence in the $!$-version of the story. If we were to take the analogous $!$-root stack functor, we would get the same functor as the $*$-root stack functor, as $\QC(\wtag)$ is self-dual as a module over $\QC(\ag)$. However the main result \cref{thm:categoricalSOD} (as well as other results like \cref{thm:categoricalgluing}) still hold if we write $\IndCoh$ instead of $\Coh$ everywhere, using the fact that the operation $\Ind$ preserves SODs by \cref{prop:stvsst}. So $\IndCoh(\sqrt[n]{X/D})$ is not the target of a natural root stack functor applied to $\IndCoh(X)$; it is only the assignment of $\Ind \circ \sqrt[n, !]{-}$ applied to $\Coh(X)$.

The next section ``fixes'' this issue in that the two root stack functors in the small case will agree (as they do in the large case).

\section{The upgrade to schobers}
\label{sect:upgrade}
In the previous section we constructed functors $\sqrt[n,*]{-},\sqrt[n,!]{-}$ on $\Perft(\ag)$. In this section, we will construct a second pair of root stack functors:
\begin{align*}
\sqrt[n,*]{-},\sqrt[n,!]{-}:\Coht(\ag)\to\Coht(\wtag),
\end{align*}
defined on ``coherent sheaves of categories'' on $\ag$. These will be compatible with, and enhance, the root stack functors defined previously. The new root stack functors also will have the advantage that they (non-canonically) agree.

The most interesting feature of this upgrade is that $\Coht(\ag)$ appears on the B-side in the 3d mirror symmetry of \cite{3dms}:
\begin{thm}\textup{\cite{3dms}}\label{thm:3dms}
    We have an equivalence of 2-categories,
    \begin{align*}
    \Sph(\st) \simeq \Coht(\ag).
    \end{align*}
\end{thm}

Here, we denote by $\Sph(\st)$ the 2-category of spherical adjunctions in $\st$, describing the A-side of 3d mirror symmetry.

The main objective of this section is to describe the root stack functor on the A-side under this equivalence. Viewing our root stack functor as a ``B-side pullback'', he corresponding functor on spherical adjunctions will be interpreted later in \cref{sect:schob} as an ``A-side pushforward'' of perverse schobers.

At the end of the section, we will relate these results to the previous section. In particular, we show that the new root stack functors agree with our original root stack functor, i.e. commute with inclusions
\begin{align*}
\Perft(\ag)\inclto \Coht(\ag).
\end{align*}

\subsection{Coherent sheaves of categories over a stack}
Let $Y$ be a smooth stack, and suppose we have another smooth stack $X$ with a proper map $p:X\to Y$, which we think of as an object with which to probe $Y$. Consider the following diagram of stacks.
\[\begin{tikzcd}
	& {X \times_Y X \times_Y X} & \\
	{X \times_Y X} & {X \times_Y X} & {X \times_Y X}
	\arrow["{p_{12}}", from=1-2, to=2-1]
	\arrow["{p_{13}}"', from=1-2, to=2-2]
	\arrow["{p_{23}}"', from=1-2, to=2-3]
\end{tikzcd}\]
We equip the category $\Coh(X\times_YX)$ with the convolution monoidal product,
\[\mc{F}\ast\mc{G} := p_{13,*}(p_{12}^*\mc{F}\otimes p_{23}^*\mc{G}).\]
Indeed, by the assumption that $X$ is smooth and $p$ proper, it follows that this monoidal product preserves compact objects, so it is well-defined. We will write
\[\Coht_X(Y):=\Coh(X\times_YX)-\Mod(\st).\]

\subsection{Functors}\label{subsect:functors}
We define pullback and pushforward functors between categories of coherent sheaves of categories.

Suppose we have two pairs $(X_1,Y_1)$ and $(X_2,Y_2)$ as above. Let $f:Y_1\to Y_2$ be a morphism of stacks. Write
\begin{align*}
    \mc{A}_1 &:= \Coh(X_1\times_{Y_1}X_1),\\
    \mc{A}_2 &:= \Coh(X_2\times_{Y_2}X_2).
\end{align*}
We observe that
\[{}_1\mc{M}_2 := \Coh(X_1\times_{Y_2}X_2)\]
naturally has the structure of an $(\mc{A}_1,\mc{A}_2)$-bimodule by diagrams similar to the one defining the monoidal product. Define similarly
\[{}_2\mc{M}_1 := \Coh(X_2\times_{Y_2}X_1)\]
as an $(\mc{A}_2,\mc{A}_1)$-bimodule.

We define the functors
\begin{align*}
f^* &:= {}_1\mc{M}_2\otimes_{\mc{A}_2}-:\Coht_{X_2}(Y_2)\to \Coht_{X_1}(Y_1),\\
f^! &:= \Hom_{\mc{A}_2}({}_2\mc{M}_1,-):\Coht_{X_2}(Y_2)\to \Coht_{X_1}(Y_1),\\
f_* &:= \Hom_{\mc{A}_1}({}_1\mc{M}_2,-):\Coht_{X_1}(Y_1)\to \Coht_{X_2}(Y_2),\\
f_! &:= {}_2\mc{M}_1\otimes_{\mc{A}_1}-:\Coht_{X_1}(Y_1)\to \Coht_{X_2}(Y_2).\\
\end{align*}
These come in adjoint pairs $f^*\adjto f_*$, $f_!\adjto f^!$ according to the $2$-categorical tensor-hom adjunction (though we do not use this fact except in the discussion of \cref{rem:smallvslarge}).

\begin{rem}
    We warn here that these functors behave less well than what the notation may suggest, due to the existence of the probe. For example, if we had a third pair $(X_3,Y_3)$ together with a map $g:Y_2\to Y_3$, then it need not be true that
    \[(g\circ f)^*\isomto f^*\circ g^*,(g\circ f)^!\isomfrom f^!\circ g^!\]
    though there are natural transformations in the indicated direction.

    In another example, unlike in the setting of $\Perft$, there need not be an identification $f_!\simeq f_*$.
\end{rem}

\begin{exmp}\label{eg:globalsections}
    Given $p:X\to Y$, we define \[\Gamma(Y,-):\Coht_X(Y)\to \st\]
    to be $f_*$, where $(X_2,Y_2):=(\pt,\pt)$ and $f:Y\to \pt$.
\end{exmp}

\begin{exmp}\label{exmp:truncation}
    Given $p:X\to Y$, write $(X_1,Y_1) := (X,Y)$ and $(X_2,Y_2) := (Y,Y)$ given by $\id:Y\to Y$. Then,
    \[\Coht_{X_2}(Y_2)\simeq \Perft(Y).\]
    Let $\tau:Y_1\to Y_2$ also be given by $\id$. We then have four functors,
    \[\tau_?:\Coht_X(Y)\tofrom \Perft(Y):\tau^?,\]
    where $?$ is $\ast$ or $!$.
\end{exmp}

\subsection{B-side: root stack functor}
In this subsection, we will write $X:=\ag\sqcup\pg\to Y:=\ag$. As in previous sections, we also write $\tilde{X} \to \tilde{Y}$ for the same stacks and maps.

\begin{defn}
    We call the functors
    \begin{align*}
        \sqrt[n,*]{-}:=e_n^*:\Coht(\ag)&\to\Coht(\wtag),\\
        \sqrt[n,!]{-}:=e_n^!:\Coht(\ag)&\to\Coht(\wtag)
    \end{align*}
    the $\ast$-categorical root stack functor, and the the $!$-categorical root stack functor on coherent sheaves of categories respectively.
\end{defn}

For the remainder of the section, we will omit the subscript and write $\Coht(Y)=\Coht_X(Y)$ unless stated otherwise.

\begin{rem}
    This abuse of notation is not so bad by the remarks in \cite{3dms}.
\end{rem}

\begin{prop}\label{prop:selfdual}
    There is an identification $\sqrt[n,*]{-} \simeq \sqrt[n,!]{-}$.
\end{prop}

This is because the bimodule defining $e_n^?$ are self-dual. We delay the proof until later.

\subsection{A-side: pushforward on spherical adjunctions}
\label{subsect:Asidepf}
In this section, we will define an operation on spherical adjunctions. We will show in \cref{subsect:rootstackms} that this is identified with the root stack functor under \cref{thm:3dms}. The objective here is to define the operation independently of mirror symmetry.

\begin{defn}\label{def:asidepf}
    The $n$\textsuperscript{th} power pushforward functor on $\Sph(\st)$ sends a spherical functor
    \[F \colon \Phi \to \Psi\]
    to the spherical functor
    \[F' \colon \Phi' \to \Psi',\]
    where
    \begin{itemize}
        \item $\Phi'$ is the category of diagrams of the form $F(\varphi) \to \psi_1 \to \psi_2 \to \dots \to \psi_{n-1}$ (along with the data of $\varphi$),
        \item $\Psi':= \Psi^{\oplus n}$, and
        \item $F'$ sends
        \[[\varphi; F(\varphi) \to \psi_1 \to \psi_2 \to \dots \to \psi_{n-1}]\mapsto \fib(F(\varphi) \to \psi_1) \oplus \fib(\psi_1 \to \psi_2) \cdots \oplus \fib(\psi_{n-2} \to \psi_{n-1}) \oplus \psi_{n-1}.\]
    \end{itemize}
\end{defn}
    
We may also describe $\Phi'$ via the SOD,
\[\Phi'=\langle \Phi, \Psi, \dots, \Psi \rangle,\]
where the nontrivial gluing functors are only between adjacent components, and are $F \colon \Phi \to \Psi$ (the original spherical functor) as well as $\id\colon \Psi \to \Psi$ abstractly. The two descriptions are related by \cref{thm:linearSOD}.

We will give a conceptual explanation for this definition in \cref{sect:schob}.

\begin{prop}\textup{\cite[Lemma 3.8]{christ}}
    The functor $F':\Phi'\to \Psi'$ is indeed spherical.
\end{prop}

\begin{rem}\label{subsect:Waldhausen}
    The vanishing cycles category $\langle \Phi, \Psi, \dots, \Psi \rangle$ is equivalent to the relative Waldhausen S-construction of \cite{DKSS}; thus we know that it has a $2n$-periodic SOD by \cite[Theorem 5.4.2]{nspherical}. The SOD from \cref{prop:keycalculation1} in \cref{subsect:rootstackms} gives a new proof of periodicity, as it is periodic, and one can tensor this SOD with any $\mc{A}$-module (i.e. any spherical functor) to recover the relative Waldhausen S-construction.
\end{rem}

\subsection{Mirror symmetry}\label{subsect:rootstackms}
We will compute the root stack functor in terms of spherical adjunctions via the equivalence \cref{thm:3dms}. We will prove the following theorem.

\begin{thm}\label{thm:main}
    On the object level, the root stack functors $\sqrt[n,*]{-},\sqrt[n,!]{-}$ on $\Coht(\ag)$ are each identified with the $n$\ts{th} power pushforward functor on $\Sph(\st)$ under the mirror symmetry equivalence \cref{thm:3dms}.
\end{thm}

\begin{rem}
    While we define our A-side pushforward on the object level, it should really be a functor between 2-categories. Forthcoming work \cite{ACJ} constructs a 2-categorical pushforward functor from schobers on the $n$-spider to schobers on the $1$-spider which agrees with ours on the object level. Note that our B-side pullbacks are already defined 2-categorically since they are tensor/hom with a bimodule. We expect that the 2-categorical B-side pullback functors should match the 2-categorical pushforward from \cite{ACJ}.
\end{rem}

We recall here that the equivalence of \cref{thm:3dms} is given as follows. Let $$\mc{A}_\Phi:=\Coh(X\times_YY), \mc{A}_{\Psi}:=\Coh(X\times_Y\pg)$$ be left $\mc{A}$-modules. These admit an adjoint pair of $\mc{A}$-linear functors,
\begin{equation}\label{eqn:adjunction}
    i^*:\mc{A}_\Phi\tofrom \mc{A}_\Psi:i_*,
\end{equation}
induced by the map $i:\pg\to Y$.

Given a left $\mc{A}$-module $\mc{C}$, we write
\begin{align*}
    {}_{\Phi}\mc{C} &:= \Hom_{\mc{A}}(\mc{A}_\Phi,\mc{C})\\
    {}_{\Psi}\mc{C} &:= \Hom_{\mc{A}}(\mc{A}_\Psi,\mc{C}).
\end{align*}
Then, the spherical adjunction associated to $\mc{C}$ is
\[{}_{\Phi}\mc{C}\tofrom {}_{\Psi}\mc{C},\]
induced by the adjunction (\ref{eqn:adjunction}). Note that $\mc{A} \simeq \mc{A}_\Phi \oplus \mc{A}_\Psi$, so that $\mc{C} \simeq {}_{\Phi}\mc{C} \oplus {}_{\Psi}\mc{C}$.

We collect here two useful lemmas and along the way, prove \cref{prop:selfdual}. In what follows, we let ${}_\Phi\mc{A}:=\Coh(Y\times_YX)$, ${}_\Psi\mc{A}:=\Coh(\pg\times_YX)$ be right $\mc{A}$-modules.

We also let ${}_\Phi E\in \mc{A}$ be given by the pushforward of the structure sheaf of $Y$ along $Y\to Y\times_Y Y \to X\times_YX$, and ${}_\Psi E\in \mc{A}$ be given by the pushforward of the structure sheaf of $\pg$ along $\pg\to \pg\times_Y\pg\to X\times_YX$. We have that $1_{\mc{A}}\simeq{}_\Phi E\oplus {}_\Psi E$. These form complementary idempotent objects in $\mc{A}$ generating ${}_\Phi\mc{A}$ and ${}_\Psi\mc{A}$ as submodules.

\begin{lem}\label{lem:dual}
    The left $\mc{A}$-module $\mc{A}_\Phi$ is dualizable, with dual ${}_\Phi\mc{A}$. In particular, for any left $\mc{A}$-module $\mc{C}$, we have equivalences
    \[{}_\Phi\mc{C}=\Hom_{\mc{A}}(\mc{A}_\Phi,\mc{C})\simeq {}_\Phi\mc{A}\otimes_{\mc{A}}\mc{C}.\]
\end{lem}

\begin{lem}\label{lem:phiformula}
    Let $\mc{C}$ be any left $\mc{A}$-module. The submodule inclusion ${}_\Phi\mc{A}\to \mc{A}$ induces an inclusion,
    \[{}_{\Phi}\mc{C}\to \mc{C},\]
    and this identifies ${}_{\Phi}\mc{C}$ with the full subcategory generated by the image of
    \[{}_\Phi E\ast-:\mc{C}\to\mc{C}.\]
    In particular, whenever $U$ is a stack with a map to $Y$,
    \[{}_\Phi\Coh(X\times_Y U) \simeq \Coh(U).\]

    Similarly, ${}_{\Psi}\mc{C}$ is naturally identified with the full subcategory generated by the image of ${}_\Psi E\ast-$, and so
    \[{}_\Psi\Coh(X\times_Y U)\simeq \Coh(\pg\times_YU).\]
\end{lem}

The idea of the proof of \cref{thm:main} is to study the $(\tilde{\mc{A}},\mc{A})$-bimodule
\[\mc{M} := \Coh(\tilde{X}\times_YX)\]
and produce structures on this object which are linear under the right $\mc{A}$-action.

We recall the left $\tilde{\mc{A}}$-modules, $\tilde{\mc{A}}_\Phi,\tilde{\mc{A}}_\Psi$ which corepresent the vanishing cycles and nearby cycles functor respectively.

By \cref{lem:phiformula}, we have identifications:
\begin{itemize}
    \item ${}_\Phi\mc{M} := \Hom_{\tilde{\mc{A}}}(\tilde{\mc{A}}_\Phi,\mc{M}) \simeq \Coh(\tilde{Y}\times_YX)$, and
    \item ${}_\Psi\mc{M} := \Hom_{\tilde{\mc{A}}}(\tilde{\mc{A}}_\Psi,\mc{M}) \simeq \Coh(\wtpg\times_YX)$,
\end{itemize}
where $\wtag\to\ag$ is given by the $n$\ts{th} power map.

\begin{prop}\label{prop:keycalculation1}
    The category ${}_\Phi\mc{M}$ admits an SOD,
    \[{}_\Phi\mc{M}\simeq\langle {}_\Phi\mc{A}, {}_\Psi\mc{A}, \dots, {}_\Psi\mc{A} \rangle,\]
    where we have nontrivial gluing functors only between adjacent components, that are abstractly a pullback ${}_\Phi\mc{A} \to {}_\Psi\mc{A}$ coming from the map $\pg \to \ag$ and the identity functor on ${}_\Psi\mc{A}$.

    Moreover, this SOD is linear for the right $\mc{A}$-action.
\end{prop}

\begin{proof}
    We apply \cref{thm:categoricalSOD} to $p:X\to Y$ to get an SOD,
    \[{}_\Phi\mc{M}=\langle e_n^*\Coh(Y\times_YX), i_*h_n^*\Coh(\pg\times_YX), \dots, i_*h_n^*\Coh(\pg\times_YX)\langle n-2 \rangle \rangle.\]
    The SOD is written in a way where the linearity for the right $\mc{A}$-action is clear.

    See now that we have natural identifications of right $\mc{A}$-modules,
    \begin{align*}
        \Coh(\ag\times_YX) &\simeq {}_\Phi\mc{A},\\
        \Coh(\pg\times_YX) &\simeq {}_\Psi\mc{A}.
    \end{align*}
    The abstract gluing functors are calculated similarly to previously.

    Outside characteristic zero, we cannot direct apply \cref{thm:categoricalSOD}, but we believe that this result still holds by manual calculation on the generators (see \cref{subsect:convent}).
\end{proof}

\begin{prop}\label{prop:keycalculation2}
    The category ${}_\Psi\mc{M}$ admits a decomposition,
    \[{}_\Psi\mc{M}\simeq {}_\Psi\mc{A}^{\oplus n}.\]
    Moreover, this decomposition is naturally linear for the right $\mc{A}$-action.
\end{prop}

\begin{proof}
    We first study $\Perf(\wtpg)$ as a right $\Perf(\pg)$-module, via the $n$\ts{th} power map $h_n:\wtpg\to\pg$. It is easy to see that $h_n^*:\Perf(\pg)\to\Perf(\wtpg)$ is a fully faithful embedding and that
    \[\Perf(\wtpg) = h_n^*:\Perf(\pg) \oplus h_n^*\Perf(\pg)\langle1\rangle \oplus\cdots\oplus h_n^*\Perf(\pg)\langle n-1\rangle.\]

    The same decomposition is automatically linear over $\Perf(Y)$ via pullback along $\pg\to Y$. We therefore have
    \begin{align*}
        {}_\Psi\mc{M} &\simeq \Hom_{\Perf(Y)}\left(\Perf(\wtpg),\Coh(X)\right)\\
        &\simeq \Hom_{\Perf(Y)}\left(\oplus_{i=0}^{n-1}h_n^*\Perf(\pg)\langle i\rangle ,\Coh(X)\right)\\
        &\simeq \oplus_{i=0}^{n-1}h_n^*\Hom_{\Perf(Y)}\left(\Perf(\pg),\Coh(X)\right)\langle -i\rangle\\
        &\simeq \oplus_{i=0}^{n-1}h_n^*\Coh(\pg\times_YX)\langle -i\rangle.
    \end{align*}
    
    We again use the identification $\Coh(\pg\times_YX)\simeq {}_\Psi\mc{A}$ to conclude.
\end{proof}

\begin{prop}\label{prop:keycalculation3}
    The spherical functor ${}_\Phi\mc{M}\to {}_\Psi\mc{M}$ is given by the formula from before.
\end{prop}

\begin{proof}
    Via the equivalence of \cite{3dms}, this spherical functor is given as follows.
    \[\begin{tikzcd}
    	{{}_\Phi\mc{M}} & {{}_\Psi\mc{M}} \\
    	{\Coh(\tilde{Y}\times_YX)} & {\Coh(\wtpg\times_YX)}
    	\arrow[from=1-1, to=1-2]
    	\arrow["="', from=1-1, to=2-1]
    	\arrow["="', from=1-2, to=2-2]
    	\arrow["{i^*}", from=2-1, to=2-2]
    \end{tikzcd}\]
    The proof is now a matter of computing the functor $i^*$ with respect to the decompositions of \cref{prop:keycalculation1} and \cref{prop:keycalculation2}. We omit the details here.
\end{proof}

\begin{proof}[Proof of \cref{thm:main}]
    We first prove the $\ast$-version of the statement.
    
    Let $\mc{C}$ be our input left $\mc{A}$-module. The $\ast$-root stack functor outputs the left  $\tilde{\mc{A}}$-module $\mc{M}\otimes_{\mc{A}}\mc{C}$. By \cref{lem:dual}, we can compute the resulting spherical functor using tensor products instead of homs, giving us
    \[{}_\Phi\tilde{\mc{A}}\otimes_{\tilde{\mc{A}}}\mc{M}\otimes_{\mc{A}}\mc{C}\tofrom {}_\Psi\tilde{\mc{A}}\otimes_{\tilde{\mc{A}}}\mc{M}\otimes_{\mc{A}}\mc{C}.\]
    The calculations of \cref{prop:keycalculation1} and \cref{prop:keycalculation2} identify these with
    \[\langle {}_\Phi\mc{A}, {}_\Psi\mc{A}, \dots, {}_\Psi\mc{A} \rangle\otimes_{\mc{A}}\mc{C}\tofrom {}_\Psi\mc{A}^{\oplus n}\otimes_{\mc{A}}\mc{C}.\]
    The forward spherical functor is given by \cref{prop:keycalculation3}.
    Finally by \cref{thm:preserve}, this is identified with
    \[\langle {}_\Phi\mc{C}, {}_\Psi\mc{C}, \dots, {}_\Psi\mc{C} \rangle\tofrom {}_\Psi\mc{C}^{\oplus n},\]
    completing the proof for $\sqrt[n,*]{-}$.

    Let us now prove the $!$-version. The proof is largely analogous and so we will omit details here. Recall that the $!$-categorical root stack outputs the left $\tilde{\mc{A}}$-module $\Hom_{\mc{A}}(\mc{N},\mc{C})$, where $\mc{N}: = \Coh(X\times_Y\tilde{X})$. An analog of \cref{prop:keycalculation1} produces a projection SOD of $\mc{N}_{\Phi}$ via a dual version of \cref{thm:categoricalSOD} (by applying $!$-root stack on $\Perft(\ag)$ to \cref{thm:a1gmsod} instead of \cref{thm:a1gmsod2}). The analog of \cref{prop:keycalculation2} is trivial, and we then calculate the spherical functor $\mc{N}_{\Phi}\to \mc{N}_\Psi$ analogously to \cref{prop:keycalculation3}.

    Finally, applying $\Hom_{\mc{A}}(-,\mc{C})$ outputs an SOD on the vanishing cycle via inclusion and once again we can compare with the $n$\ts{th}-power pushforward to complete the proof.
\end{proof}

\begin{rem}
    It is interesting to understand to what extent the resulting identification \[\sqrt[n,*]{-}\simeq\sqrt[n,!]{-}\]
    is canonical. This is downstream of the extent to which the SODs of $\Perf(\wtag)$ via inclusion and projection from \cref{thm:a1gmsod} and \cref{thm:a1gmsod2}, respectively, are related to one another.
\end{rem}

\begin{rem}
    We recall that a stack over $\ag$ is equivalent to the information of a stack equipped with a (generalized Cartier) divisor. In \cite{seidelfukaya2}, a notion of non-commutative divisor is introduced. We expect this notion is generalized by categories over $\ag$.
\end{rem}

Note that we have proved \cref{prop:selfdual} indirectly, but we can also prove it directly by showing that $\mc{N}$ is non-canonically equivalent to the dual module to $\mc{M}$ over $\mc{A}$.

\begin{proof}[Proof of \cref{prop:selfdual}]
    By \cref{prop:keycalculation1} and \cref{prop:keycalculation2}, we have an equivalence of right $\mc{A}$-modules,
    \[\mc{M}\simeq {}_\Phi\mc{M}\oplus {}_\Psi\mc{M}\simeq\langle {}_\Phi\mc{A}, {}_\Psi\mc{A}, \dots, {}_\Psi\mc{A} \rangle\oplus{}_\Psi\mc{A}^{\oplus n}.\]
    
    By \cref{cor:dualexists}, since the components of the SOD of $\mc{M}$ are dualizable by \cref{lem:dual}, we know there is a dual. We claim that this dual is noncanonically identified with $\mc{N}$.

    By \cref{cor:dualexists}, the dual module to $\mc{M}$ has an SOD that looks like $\langle \mc{A}_{\Psi}, \dots, \mc{A}_{\Psi}, \mc{A}_{\Phi} \rangle \oplus \mc{A}_{\Psi}^{\oplus n}$, with a linear chain of gluing functors in the first component (if an SOD by inclusion is a linear chain, then after applying $\Hom_{\mc{A}}(-, \mc{A})$, the corresponding SOD by projection is also a linear chain). Now if we mutate to reverse the orientation of the linear chain, we get an SOD $\langle \mc{A}_{\Phi}, \mc{A}_{\Psi}, \dots, \mc{A}_{\Psi} \rangle$, which we can non-canonically identify with an SOD $\langle {}_\Phi\mc{A}, {}_\Psi\mc{A}, \dots, {}_\Psi\mc{A} \rangle$ with some gluing functors. Recall that in the initial SOD, all gluing functors other than the one between $\mc{A}_{\Phi}$ and the adjacent $\mc{A}_{\Psi}$ are identity or zero. Then we can use similar reasoning to \cref{thm:mutatearb1} to conclude that the mutated SOD has a linear chain where all the nonzero gluing functors are identity except for the first one.
    
    Finally, \cref{lem:identifyadjoint} shows that for the pullback case we care about, the adjoint functor of pushforward is (non-canonically) identified with the dual functor, showing that the functors $\mc{A}_\Phi \to \mc{A}_\Psi$ in our two SODs non-canonically agree, which is enough to give an abstract isomorphism of $\mc{N}$ with the dual of $\mc{M}$ using the fact that all other gluing functors are identity or zero.
\end{proof}

\begin{lem}\label{lem:identifyadjoint}
Recall that by the projection formula, the adjunction \cref{eqn:adjunction} is an $\mc{A}$-module adjunction. As ${}_\Phi\mc{A}, {}_\Psi\mc{A}$ are direct summands of $\mc{A}$, and $\Hom_{\mc{A}}(\mc{A}, \mc{A})$, the maps of the adjunction can each be identified with an element of $\mc{A}$. We claim that each map can be identified with the structure sheaf $\OO_{\pg}$ of $\Coh(\pg)$ (in the two different entries for the two maps).
\end{lem}
\begin{proof}
    Each map is identified by where the relevant idempotent object goes.
    
    The pullback is identified by where the structure sheaf of $\Coh(\ag)$ goes, i.e. to the structure sheaf fo $\pg$.
    
    The pushforward is identified by where the idempotent ${}_\Psi E \in \Coh(\pg \times_{\ag} \pg)$ goes under pushforward to $\Coh(\pg)$, which is $\OO_{\pg}$.
\end{proof}

One can view the passage to convolution categories as causing the *- and !-pullbacks to agree by making the relevant module dualizable. While $\Coh(\wtag)$ is not dualizable over $\Coh(\ag)$, it is true that $\mc{M}$ is dualizable over $\mc{A}$. In fact, we can prove this result in more generality.

\begin{thm}\label{thm:gendual}
    Suppose $Z, Y, X$ are smooth stacks with a proper maps $Z \to Y, X \to Y$. If $N^*_X(Y)=N^*_Z(Y)$ as subsets of $T^*Y$, then $\Coh(Z \times_Y X)$ is dualizable over $\Coh(X \times_Y X)$.
\end{thm}
\begin{proof}
    By \cite[Proposition 4.18]{hypertoric} (as suggested in \cite{arinkintalk1}), the bimodule $\Coh(Z \times_Y X)$ gives a Morita equivalence between $\Coh(X \times_Y X)$ and $\Coh(Z \times_Y Z)$. Under the Morita equivalence, $\Coh(X \times_Y Z) \in \Coh(X \times_Y X)-\Mod$ goes to $\Coh(Z \times_Y X) \otimes_{\Coh(X \times_Y X)} \Coh(X \times_Y Z)$ in $\Coh(Z \times_Y Z)-\Mod$.

    Then by \cref{cor:sameconormal}, we have $\Coh(Z \times_Y X) \otimes_{\Coh(X \times_Y X)} \Coh(X \times_Y Z) \simeq \Coh(Z \times_Y Z)$, which yields the result.
\end{proof}

This result suggests the following procedure. If one wants to analyze a pullback of categories over schemes via the map $Y_1 \to Y_2$, by modeling categories over $Y_2$ as $\Coh(X_2 \times_{Y_2} X_2)-\Mod$ and using the bimodule $\Coh(X_1 \times_{Y_2} X_2)$, then the two possible pullbacks agree (suggesting that it is a more genuine pullback) if $N^*_{X_1}Y_2=N^*_{X_2}Y_2$. Note that the bimodule gives an equivalence with $\Coh(X_1 \times_{Y_2} X_1)-\Mod$ in this case, not $\Coh(X_1 \times_{Y_1} X_1)$, so the root stack functor is not a Morita equivalence. 

For smooth stacks, we can essentially recover \cref{thm:categoricalSOD}. See \cref{rem:betterrecover} for more discussion.

\begin{thm}\label{thm:recover}
    Let $U$ be smooth, with a map to $\ag$. The root stack functor has the assignments
    $$\Perf((\ag \sqcup \pg) \times_{\ag} U) \mapsto \Perf((\age \sqcup \pge) \times_{\ag} U)$$
    and 
    $$\Coh((\ag \sqcup \pg) \times_{\ag} U) \mapsto \Coh((\age \sqcup \pge) \times_{\ag} U).$$
\end{thm}
\begin{proof}
    The $\Perf$ case follows from \cref{cor:QCtensor}. For the $\Coh$ case, we need to use the preceding criterion. (Note that in both cases the small category result follows by taking compact objects in the large category version.) We just need to show that for $a \in \Omega_{\ag}|y$ for some $y$, the existence of $\bar{x} \in \age$ with $\bar{x} \mapsto y$ and $df^*_{\bar{x}}a=0$ implies that there is also $x \in \ag$ such that $df^*_xy=0$.

    Indeed, the map $\age \to \ag$ is an isomorphism above $\gm/\gm$, so the only possibility is for $y$ to be a geometric point mapping to $\pg$. But in this case any $a$ works since we can take $x$ also mapping to $\pg$ (note that we only care about the classical points of $\Sing(\ag)$).
\end{proof}

\subsection{Root stack on \texorpdfstring{$\Coht$}{2Coh} as a generalization of the usual root stack}
We describe how the root stack functors in this section are related to those from the previous section.

We recall from \cref{exmp:truncation} the functors
\[\tau^?:\Perft(\ag)\to \Coht(\ag).\]
We claim that there are natural commutative diagrams
\[\begin{tikzcd}
	{\Perft(\ag)} & {\Coht(\ag)} \\
	{\Perft(\wtag)} & {\Coht(\wtag)}
	\arrow["{\tau^?}", from=1-1, to=1-2]
	\arrow["{\sqrt[n,?]{-}}"', from=1-1, to=2-1]
	\arrow["{\sqrt[n,?]{-}}"', from=1-2, to=2-2]
	\arrow["{\tau^?}", from=2-1, to=2-2]
\end{tikzcd}\]

\begin{prop}\label{prop:commutingrootstackdiagram}
    The above diagram naturally commutes for $?$ either $*$ or $!$.
\end{prop}

\begin{proof}
    We do the $*$ version first. Then we have
    \begin{align*}\Coh(\tilde{X}\times_YX) \otimes_{\mc{A}} \mc{A}_\Phi \otimes_{\Perf(Y)} (-) &\simeq \Coh(\tilde{X}) \otimes_{\Perf(Y)} (-) \\&\simeq \tilde{\mc{A}}_\Phi \otimes_{\Perf(\tilde{Y})} \Perf(\tilde{Y}) \otimes_{\Perf(Y)} (-),\end{align*}
    where the first identification follows from \cref{lem:phiformula}. The $!$ version is similar.
\end{proof}

\begin{rem}\label{rem:smallvslarge}
The authors of \cite[Remark 2.7]{3dms} note that the left and right adjoints above correspond in some sense to the \textit{Kleisli} and \textit{Eilenberg--Moore} adjunctions for (spherical) monads. However a priori it is unclear that the latter give spherical adjunctions (which require both a twist and a cotwist to be invertible). To get the left and right adjoint above may require additionally inverting the cotwist. 

Now we point out a difference between the small and large cases. Recall that $\otimes$ respects $\Ind$ while $\Hom$ does not. So in the case that we start with $\Perf(X)$ or $\QC(X)$ for a perfect stack $X$, the left adjoint yields $\Perf(X) \to \Perf(D)$ or $\QC(X) \to \QC(D)$ (using \cite{BZFN}). However, the right adjoint is different. For simplicity suppose that $X$ is smooth, so $\Perf(X)=\Coh(X)$. Then the right adjoint applied to this small category yields $\Coh(X) \to \Coh(D)$ (by \cite{BZNP}), whereas on the large category $\QC(X)$ it yields $\QC(X) \to \QC(D)$ (by the dualizability of $\QC(\ag)$).
\end{rem}

\begin{rem}\label{rem:betterrecover}
    Recall that by \cref{thm:recover}, we can directly show that the root stack construction on $\Coht(\ag)$ ``recovers'' the one on $\Perft(\ag)$ for the specific case of smooth stacks, using results about singular support of coherent sheaves. Namely, for a smooth $U$, performing the $\otimes$ root stack construction on $\Perf(U)$ yields $\Perf(\sqrt[n]{U/D})$, and performing the $\Hom$ root stack construction on $\Coh(U)$ yields $\Coh(\sqrt[n]{X/D})$. However, \cref{prop:commutingrootstackdiagram} shows that the root stack construction in \cref{sect:upgrade} recovers the results of \cref{sect:categorical} in much broader generality. Namely, for a perfect $U$, performing the $*$-root stack construction on $\Perf(U)$ yields $\Perf(\sqrt[n]{U/D})$, and for a locally finite type $U$, performing the $!$-root stack construction on $\Coh(U)$ yields $\Coh(\sqrt[n]{X/D})$, by \cref{thm:categoricalSOD}.
\end{rem}

\subsection{Large categories}
If we work with large categories instead of small categories, then the story is completely analogous. By upgrading to convolution categories (instead of using $\Perft(\ag)$ as in \cref{sect:categorical}), we already have equivalent *- and !-pullback. Note though that in the large category case we have the general result \cref{thm:hom} and wouldn't need to prove \cref{prop:selfdual}.

\subsection{B-side pushforward via the map \texorpdfstring{$e_n$}{en}}\label{subsect:Bsidepb}
We now investigate the pushforward map on the B-side. (It is less clear what this corresponds to conceptually on the A-side.)

Though there are two pushforward functors $e_{n,*}$ and $e_{n,!}$ from $\Coht(\wtag)$ to $\Coht(\ag)$, by \cref{prop:selfdual} they are equivalent. So for simplicity, we work with the shriek pushforward, which is tensoring with $\mc{N}=\Coh(Y \times_X \tilde{Y})$.

By \cref{lem:phiformula}, we have identifications:
\begin{itemize}
\item ${}_\Phi\mc{N} := \Hom_{\mc{A}}(\mc{A}_\Phi,\mc{N}) \simeq \Coh(Y \times_Y \tilde{X})$, and
    \item ${}_\Psi\mc{N} := \Hom_{\mc{A}}(\mc{A}_\Psi,\mc{N}) \simeq \Coh(\pg\times_Y \tilde{X})$.
\end{itemize}

We also define ${}_\Phi\tilde{\Am}:=\Coh(\tilde{Y} \times_{\tilde{Y}} \tilde{X})$ and ${}_\Psi\tilde{\Am}:=\Coh(\wtpg \times_{\tilde{Y}} \tilde{X})$.

Note that as a right $\tilde{\Am}$-module, ${}_\Phi\mc{N}$ is equivalent to ${}_\Phi\tilde{\Am}$. We now calculate an SOD for ${}_\Psi\mc{N}$.

\begin{prop}
    The category ${}_\Psi\mc{N}$ admits an SOD,
    $${}_\Psi\mc{N} \simeq \langle {}_\Psi\tilde{\Am}, \dots, {}_\Psi\tilde{\Am}\rangle,$$

    where we have nontrivial gluing functors only between adjacent components, that are abstractly the identity functor on ${}_\Psi\tilde{\Am}$. In other words, we have $${}_\Psi\mc{N} \simeq \Fun(A_n, {}_\Psi\tilde{\Am})$$ where $A_n$ is the linear quiver on $n$ nodes. Moreover, this SOD is linear for the right $\tilde{\Am}$-action.
\end{prop}
\begin{proof}
The key is the commutative diagram

\[\begin{tikzcd}
	V & W & {\wtag \sqcup \wtpg} \\
	\wtpg & {\Spec(k[t]/t^n)/\Gm} & \wtag \\
	& \pg & \ag
	\arrow[from=1-1, to=1-2]
	\arrow[from=1-1, to=2-1]
	\arrow["\lrcorner"{anchor=center, pos=0.125}, draw=none, from=1-1, to=2-2]
	\arrow[from=1-2, to=1-3]
	\arrow[from=1-2, to=2-2]
	\arrow["\lrcorner"{anchor=center, pos=0.125}, draw=none, from=1-2, to=2-3]
	\arrow[from=1-3, to=2-3]
	\arrow[from=2-1, to=2-2]
	\arrow[from=2-2, to=2-3]
	\arrow[from=2-2, to=3-2]
	\arrow["\lrcorner"{anchor=center, pos=0.125}, draw=none, from=2-2, to=3-3]
	\arrow[from=2-3, to=3-3]
	\arrow[from=3-2, to=3-3]
\end{tikzcd}\]
where $V$ and $W$ are defined as the fiber products. Then $\Coh(\Spec (k[t]/t^n)/\Gm) \simeq \Fun(A_n, \Perf(\pg))$; pulling back yields ${}_\Psi\mc{N} \simeq \Coh(W) \simeq \Fun(A_n, \Perf(V))\simeq \Fun(A_n, {}_\Psi\tilde{\Am})$. Since $\Perf(V)$ is stable under the right action of $\tilde{\Am}$, we get that the resulting SOD of ${}_\Psi\mc{N}$ is stable as well.
\end{proof}

We can determine the spherical functor we get in this case.

\begin{prop}
    Given a spherical functor 
$$F \colon \Phi \to \Psi,$$
the B-side pushforward sends it to the spherical functor
$$F' \colon \Phi \to \Fun(A_n, \Psi),$$

where $s \in \Phi$ is sent to $F(s) \to F(s) \to \cdots \to F(s)$ (with all maps the identity).
\end{prop}
\begin{proof}

The reasoning is similar to \cref{prop:keycalculation3}. The key step is that the pullback for the map $$\Spec(k[t]/t^n)/\Gm\to \wtpg$$ sends $\OO_{\wtpg}$ to $\OO_{\Spec(k[t]/t^n)/\Gm}$, and the latter is an iterated cone from the sequence $\OO_{\pg} \to \OO_{\pg} \to \cdots \to \OO_{\pg}$.
\end{proof}

\subsection{Functoriality for \texorpdfstring{$\pt=\gm/\gm \to \ag$}{pt=Gm/Gm->A1/Gm}}
Consider the map $\pt \to \ag$ that maps to the open set $\Gm/\Gm \subset \ag$. Under the pullback from \cref{subsect:functors}, we show that it sends $\mc{C}_\Phi \to \mc{C}_\Psi$ to $\ker(\mc{C}_\Phi \to \mc{C}_\Psi)$ in the case of large categories. We give two proofs.

\begin{proof}[Proof 1]
    First note that $\QC(\pt) \to \QC(\ag)$ is pushforward along an (affine) open embedding, so it is fully faithful with an adjoint. Thus $\QC(\pt) \otimes_{\QC(\ag)} \mc{C}_\Phi$ is a full subcategory of $\mc{C}_\Phi$ (see \cref{thm:preserve}).
    
    First suppose $\mc{F}$ is in $\QC(\pt) \otimes_{\QC(\ag)} \mc{C}_\Phi$. Then it is clearly in the kernel.

    Now suppose $\mc{F}$ is in the kernel. Then consider the element $\OO_{\pg} \in \mc{A}$, such that convolution with this element defines the spherical functor. We have $\OO_{\pg} * \OO_{\pg} * \mc{F}=\OO_{\pg} * 0=0$, but one can calculate $\OO_{\pg} * \OO_{\pg}=i_*\OO_{\pg}$, so $i_*\OO_{\pg} * \mc{F}=0$, meaning it is in $\QC(\pt) \otimes_{\QC(\ag)} \mc{C}$. (Note that $i_*\OO_{\pg} \in \QC(\ag) \in \mc{A}$.) 
\end{proof}

We also have a proof that uses the dualizability of $\mc{A}$.

\begin{proof}[Proof 2]
    Using \cref{lem:duals} the dual $\QC(\pt)$ is $\QC(\pt)$, and we can compute $\QC(\pt) \otimes_{\Am} (\mc{C}_\Phi \to \mc{C}_\Psi)$ as 
    $\Hom_{\Am}(\QC(\pt), (\mc{C}_\Phi \to \mc{C}_\Psi))$. Using the equivalence of \cite{3dms}, we want maps of spherical functors from $\QC(\pt) \to 0$ to $\mc{C}_\Phi \to \mc{C}_\Psi$, which is just the kernel of $\mc{C}_\Phi \to \mc{C}_\Psi$.
\end{proof}

The possible singular support of an object in $\Coht(\ag)$ is the set $\ag$ along with a one-dimensional cotangent fiber at $\pg$. (Note that this notion of singular support is at a higher categorical level than the notion in \cref{subsect:singsupp}.) Then one can think of the pullback to $\Perf(\pt)=\Perf(\gm/\gm)$ as taking the zero-section singular support component. Taking the Fourier transform and then the kernel corresponds to taking the singular support component above $\pg$. Note that the kernel of the pushforward $\Perf(D) \to \Perf(X)$ is zero; this then corresponds to the fact that the spherical functor $\Perf(X) \to \Perf(D)$ has no singular support outside the zero section and comes from the $\Perf(\ag)$-module $\Perf(X)$.

If $\mc{C}_\Phi \to \mc{C}_\Psi$ categorifies a perverse sheaf on $(\C, 0)$, then one can think of its kernel as categorifying the (co)stalk at zero, which also gives global (co)sections. So pullback to $\Perf(\gm/\gm)$ on the B-side corresponds to pushforward to a point on the A-side (see \cref{sect:schob} for more on this perspective).

\begin{rem}
Note that B-side pullback for $\pt \to \ag$ corresponds to restriction to a semistable locus, which corresponds by \cite{hypertoric} to stop removal on the A-side, suggesting that stop removal should be thought of as a pushforward in this context.
\end{rem}

\section{Functoriality of schobers}\label{sect:schob}
Under the equivalence from \cite{3dms}, we should think of the root stack construction on the B-side (which corresponds to a \textit{pullback} via a map $\A^1/\Gm \to \A^1/\Gm$) as dual to a \textit{pushforward} of perverse schobers. We can take this as the definition of the pushforward of a perverse schober on $(\mb{C}, 0)$ via the map $z \mapsto z^n$.

It would be nice, however, to have a intrinsic definition of a pushforward of a perverse schober, or at least one defined without reference to 3d mirror symmetry. In this section we give an ad hoc definition for the pushforward that is motivated purely on the A-side. We argue that this definition yields an SOD, as well as $2n$-periodicity, in a way that is manifest geometrically in the case of a Landau--Ginzburg model.

A Landau--Ginzburg model $(X, W)$ with the only critical value of $W$ at $0$ gives rise to a Lefschetz schober on the disk stratified at $0$, as explained for example in \cite[Section 3.6]{infrared}. If we fix a path from $0$ to $+\infty$, the vanishing cycles category $\Phi$ is the Fukaya-Seidel category of $(X, W)$ (i.e. the wrapped Fukaya category of $X$ with a stop at $+\infty$), and the nearby cycles category $\Psi$ is given by the Fukaya category of a smooth fiber. The map $\Phi \to \Psi$ is restriction to a point along the path. See \cref{fig:fibration} for a picture.

\begin{figure}
    \centering
\begin{tikzpicture}[
    >=stealth,
    fibre box/.style={draw=black!70, thin, fill=white, rounded corners=10pt}
]

    \draw[fibre box] (-3, 1.5) rectangle (3, 6);
    \node[anchor=west, xshift=10pt] at (3, 3.75) {\huge $X$};

    \begin{scope}[shift={(0, 3.75)}]
        \def\coneradius{0.8}
        \def\coneheight{1.5}

        \shade[shading=axis, bottom color=fibregrey!40, top color=fibregrey!10] 
            (0, -\coneheight) circle [x radius=\coneradius, y radius=\coneradius/3];

        \shade[shading=axis, bottom color=fibregrey!30, top color=white, fill opacity=0.85] 
            (0,0) -- (-\coneradius, -\coneheight) arc (180:360:{\coneradius} and {\coneradius/3}) -- cycle;
        
        \draw[black!70, thin, dashed] (\coneradius, -\coneheight) arc (0:180:{\coneradius} and {\coneradius/3});
        \draw[black!70, thin] (-\coneradius, -\coneheight) arc (180:360:{\coneradius} and {\coneradius/3});
        
        \draw[black!70, thin] (-\coneradius, -\coneheight) -- (0,0) -- (\coneradius, -\coneheight);

        \shade[shading=axis, top color=fibregrey!30, bottom color=white] 
            (0,0) -- (-\coneradius, \coneheight) arc (180:360:{\coneradius} and {\coneradius/3}) -- cycle;
        
        \shadedraw[shading=axis, bottom color=white, top color=fibregrey!30, draw=black!70, thin] 
            (0, \coneheight) circle [x radius=\coneradius, y radius=\coneradius/3];
            
        \draw[black!70, thin] (-\coneradius, \coneheight) -- (0,0) -- (\coneradius, \coneheight);

        \fill[black] (0,0) circle (2pt);
    \end{scope}

    \begin{scope}[shift={(1.5, 3.75)}]
        \def\cylradius{0.5}
        \def\cylheight{1.2}

        \shade[shading=axis, bottom color=fibrepurple!30, top color=fibrepurple!10] 
            (0, -\cylheight) circle [x radius=\cylradius, y radius=\cylradius/3];

        \shade[shading=axis, bottom color=fibrepurple!10, top color=white, fill opacity=0.85] 
            (-\cylradius, -\cylheight) rectangle (\cylradius, \cylheight);
        
        \draw[fibrepurple!70, thin, dashed] (\cylradius, -\cylheight) arc (0:180:{\cylradius} and {\cylradius/3});
        \draw[fibrepurple!70, thin] (-\cylradius, -\cylheight) arc (180:360:{\cylradius} and {\cylradius/3});
        
        \draw[fibrepurple!70, thin] (-\cylradius, -\cylheight) -- (-\cylradius, \cylheight);
        \draw[fibrepurple!70, thin] (\cylradius, -\cylheight) -- (\cylradius, \cylheight);

        \shadedraw[shading=axis, bottom color=fibrepurple!20, top color=white, draw=fibrepurple!70, thin] 
            (0, \cylheight) circle [x radius=\cylradius, y radius=\cylradius/3];
    \end{scope}

    \draw[->, ultra thick, black] (0, 1.3) -- (0, -0.8);

    \draw[fibre box] (-3, -1.5) rectangle (3, -5);
    \node[anchor=west, xshift=10pt] at (3, -3.25) {\huge $\mathbb{C}$};

    \begin{scope}[shift={(0, -3.25)}]
        \draw[thin, black!70] (-2.5, 0) -- (2.5, 0); 

        \fill[black] (0, 0) circle (2pt);
        \node[anchor=north, yshift=-4pt] at (0, 0) {\huge 0};

        \draw[fibrepurple, ultra thick] (1.4, 0) -- (1.6, 0);
    \end{scope}

\end{tikzpicture}
\caption{A map $W \colon X \to \C$ giving a Landau--Ginzburg model, where the only critical point of $W$ is at $0$. The associated Fukaya--Seidel category, which is the vanishing cycles category $\Phi$, is the wrapped Fukaya category of $X$ with a stop at $+\infty$, whereas the nearby cycles category $\Psi$ is the Fukaya category of the preimage of a regular point (shown as the purple fiber; we have drawn the purple point as slightly horizontal to emphasize that it is not a critical point). The map $\Phi \to \Psi$ applied to a Lagrangian $L \subset X$ can be computed by first wrapping $L$ to the ``skeleton'' that is the preimage of the ray from $0$ to $+\infty$, and then restricting to the purple fiber.}
\label{fig:fibration}
\end{figure}
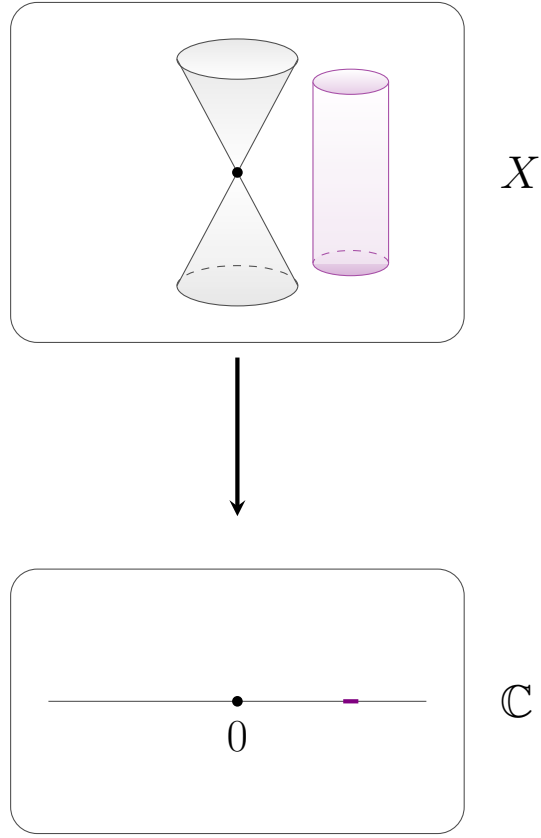

Now we want to determine the schober associated to $(X, (z \mapsto z^n) \circ W)$. The smooth fiber is now $n$ copies of the previous one, so the new nearby cycles category is $\Psi^{\oplus n}$. To find the vanishing cycles category, the path from $0$ to $+\infty$ for the $(z \mapsto z^n) \circ W$ fibration becomes an $n$-pointed star ``skeleton'' for the $W$ fibration, and we want to determine the wrapped Fukaya category associated to this skeleton. Alternatively, we have $n$ stops corresponding to $((z \mapsto z^n) \circ W)^{-1}(+\infty)=W^{-1}(e^{2\pi i/n}\infty)_{0 \le i \le n-1}$ and want to calculate the wrapped Fukaya category for $X$ with these $n$ stops. See \cref{fig:nstar} for a picture.

\begin{figure}
    \centering

\begin{tikzpicture}[>=stealth]
    
    \begin{scope}[shift={(0,0)}]
        \draw[thick, black!80] (0,0) circle (2.5);
        
        \fill[black] (0,0) circle (2.5pt);
        
        \foreach \angle in {25, 85, 145, 325} {
            \draw[thick, black!80] (0,0) -- (\angle:2.5);
            \draw[fibrepurple, ultra thick] (\angle:1) -- (\angle:1.2);
        }
        
        \foreach \angle in {205, 265} {
            \draw[thick, black!80, dashed] (\angle:0.5) -- (\angle:2.3);
        }
        
        \node at (0, -3.2) {\huge $\mathbb{C}$};
    \end{scope}
    
    \draw[->, ultra thick, black] (3.2, 0) -- (5.2, 0) node[midway, above=8pt] {\Large $z \mapsto z^n$};
    
    \begin{scope}[shift={(8.4,0)}]
        \draw[thick, black!80] (0,0) circle (2.5);
        
        \fill[black] (0,0) circle (2.5pt);
        
        \draw[thick, black!80] (0,0) -- (0:2.5);
        
        \draw[fibrepurple, ultra thick] (0:1) -- (0:1.2);
        
        \node at (0, -3.2) {\huge $\mathbb{C}$};
    \end{scope}
\end{tikzpicture}
    \caption{The vanishing cycles category associated to $(z \mapsto z^n) \circ W$ can be computed directly using the map $W \colon X \to \C$ using the left diagram. The vanishing cycles category $\Phi$ is the wrapped Fukaya category of $X$ with stops at the preimages of the $n$ ray endpoints (i.e. $n$ roots of $+\infty$). Alternatively, $\Phi$ can be thought of as the wrapped Fukaya category for the $n$-pointed star skeleton with a critical point at $0$. The nearby cycles category is the Fukaya category of the preimage of the $n$th roots of the purple point from \cref{fig:fibration}, which is $n$ copies of the original purple fiber.}
    \label{fig:nstar}
\end{figure}
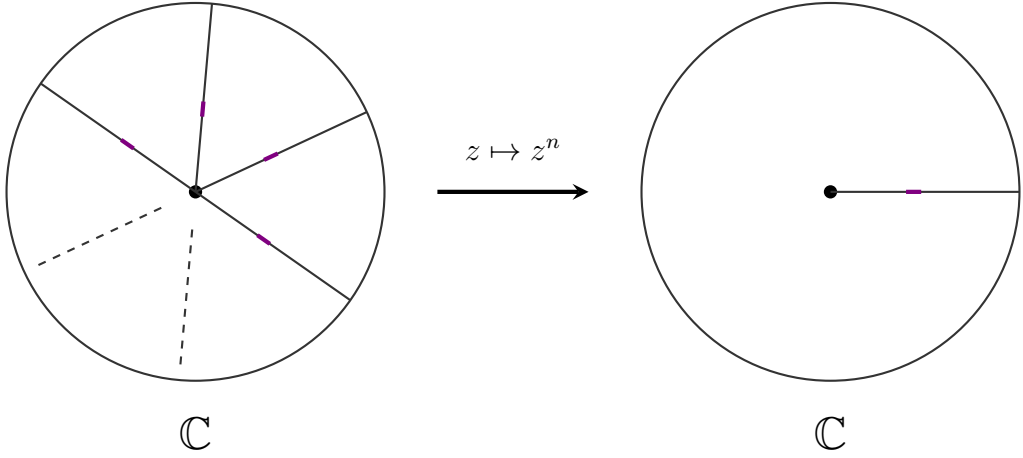

As a warmup, let us recall the wrapped Fukaya category associated to the Liouville domain on $D^2$ with $n$ stops. This category has an exceptional collection of $n-1$ generators and is $2n$-periodic under mutation. It is equivalent to $\Fun(A_{n-1}, \Vect)$ (this is a folklore expectation shown rigorously in \cite{gps2}). Note that \cite[Section 5.1]{nspherical} demonstrates $2n$-periodicity for $\Fun(A_{n-1}, \mc{D})$ for a general stable category $\mc{D}$.

\begin{align*}
\mathrm{Fuk}\left( \diskmark \right) &= \mathrm{Fuk}\left( \crossdash \right) \\
&= \mathrm{Fuk}\left( \crossskel \text{ skeleton} \right) \\
&= \langle \mathrm{Fuk}(\cornerBL), \mathrm{Fuk}(\cornerBR), \mathrm{Fuk}(\cornerTR) \rangle \\
&= \langle \mathrm{Fuk}(\cornerBR), \mathrm{Fuk}(\cornerTR), \mathrm{Fuk}(\cornerTL) \rangle \\
&= \langle \mathrm{Fuk}(\cornerTL), \mathrm{Fuk}(\cornerTR), \mathrm{Fuk}(\cornerBR) \rangle \\
&= \dots \\
&= \langle \mathrm{Fuk}(\cornerBL), \mathrm{Fuk}(\cornerBR), \mathrm{Fuk}(\cornerTR) \rangle
\end{align*}

Specifically, $\Fuk(\cornerBL)$ is the subcategory consisting of Lagrangians tranverse to the skeleton $\crossskel$ that, when wrapped, lie on $\cornerBL$ (or equivalently just the wrapped Fukaya category for the skeleton $\cornerBL$). This category is generated by the Lagrangian that is transverse to the ray pointing up (if wrapping is clockwise), and is equivalent to $\Vect$. That Lagrangian is an exceptional generator for $\Fuk(\crossskel)$.

Extrapolating from this calculation, if we have a Landau--Ginzburg model with no critical points, and we ask for the wrapped Fukaya category after adding $n$ stops, then the category would be equivalent to $\Fun(A_{n-1}, \Fuk(\text{fiber}))=\Fun(A_{n-1}, \Psi)$. The same skeleton decomposition shows that the category is $2n$-periodic under mutation (giving a geometric proof of the result from \cite[Section 5.1]{nspherical}).

Now in our case where there is a critical point at $0$, by perturbing the skeleton, we can write our Fukaya category as a pushout of $\Phi$ and $\Fun(A_n, \Psi)$. Elements of the pushout are equivalent to the data of a sequence $b_1 \to b_2 \to \cdots \to b_n$ for $b_i \in \Psi$, along with an isomorphism $G(b_n) \simeq a$ for some $a \in \Phi$. We can also take adjoints to get that our category is a pullback; in this case we have $F(a) \simeq b_1$ along with $b_1 \to b_2 \to \cdots \to b_n$. This description is exactly the $(n-1)$st level of the Waldhausen S-construction for $F \colon \Phi \to \Psi$ (see \cref{subsect:Waldhausen}).

We can decompose the skeleton into a path from $0$ to $+\infty$ (containing the critical point), and $n-1$ chords as shown below. The category above the two types of path are $\Phi$ and $\Psi$, respectively. By wrapping onto the skeleton, we see that there are only homs between adjacent components, showing that the total wrapped Fukaya category is made up of adjacent gluing functors $\Phi \to \Psi \to \cdots \to \Psi$. An rigorous derivation of this SOD appears in \cite{nadlerlg} for the case of the Landau--Ginzburg model $(\C^n, z_1\cdots z_n)$.

Now under mutation, the $\Phi$ path moves clockwise one sector. So after $2n$ mutations (i.e. $n$ changes of $\Phi$), we end up at the original SOD.

\begin{align*}
\mathrm{Fuk}\left( \diskmarkphi \right) &= \mathrm{Fuk}\left( \crossdashphi \right) \\
&= \mathrm{Fuk}\left( \crossskelphi \text{ skeleton} \right) \\
&= \langle \mathrm{Fuk}(\dotlineR), \mathrm{Fuk}(\cornerBL), \mathrm{Fuk}(\cornerBR), \mathrm{Fuk}(\cornerTR) \rangle \\
&= \langle \mathrm{Fuk}(\cornerBL), \mathrm{Fuk}(\cornerBR), \mathrm{Fuk}(\cornerTR), \mathrm{Fuk}(\dotlineU) \rangle \\
&= \langle \mathrm{Fuk}(\dotlineU), \mathrm{Fuk}(\cornerTL), \mathrm{Fuk}(\cornerBL), \mathrm{Fuk}(\cornerBR) \rangle \\
&= \dots \\
&= \langle \mathrm{Fuk}(\dotlineR), \mathrm{Fuk}(\cornerBL), \mathrm{Fuk}(\cornerBR), \mathrm{Fuk}(\cornerTR) \rangle
\end{align*}

Note that this description matches up with the ad-hoc categorification in \cite[Section 3]{christ}.

\appendix

\section{SODs \texorpdfstring{$\infty$}{infinity}-categorically}\label{sect:inftySOD}

In this section we give various characterizations of $n$-term semiorthogonal decompositions using the language of $(\infty, 2)$-categories. We let $\st$ be the category of small stable categories and exact functors, and $\sti$ be the subcategory where we take idempotent-complete objects. For large categories, we let $\St$ be the category of compactly-generated presentable stable categories, and $\Stc$ be the subcategory where we take quasi-proper functors. Then $\Ind$ gives an equivalence $\sti \simeq \Stc$, and the inclusion $\Stc \subset \St$ is colimit preserving. See \cref{subsect:convent} for our conventions regarding SOD notation and gluing functors.

All categories are assumed to be over a ring $k$. (For applications, we will assume $k$ is characteristic zero, though see \cref{subsect:convent}.)

\subsection{The internal characterization}

Let $\mc{C}$ be a stable $\infty$-category with strictly full subcategories $\mc{A}_1, \dots, \mc{A}_n$. For arbitrary stable subcategories $\mc{M}$, $\mc{N}$ of $\mc{C}$, \cite[Section 2.2.B]{DKSS} define a category $\langle\mc{M}, \mc{N}\rangle$ which is a subcategory of $\mc{C}$ if $\mc{N}$ is left orthogonal to $\mc{M}$, and say $\mc{M}, \mc{N}$ form a $2$-term SOD of $\mc{C}$ if the inclusion $\langle \mc{M}, \mc{N} \rangle \hookrightarrow \mc{C}$ is an equivalence. Section 2.6 of \cite{christ} extends their characterization to $n$-term SODs, and notes that such an SOD is also a repeated $2$-term SOD. Thus we also have the following easy generalization of \cite[Corollary 2.5.5]{DKSS}.
\begin{cor}\label{cor:htpycat}
    Strictly full stable subcategories $\mc{A}_1, \dots, \mc{A}_n$ of a stable category $\mc{C}$ form an SOD if and only if the same is true regarding their homotopy categories.
\end{cor}

We give an alternative construction of $\langle \Am_1, \dots, \Am_n \rangle$ based on the Waldhausen S-construction.
Let $\mc{A}_1, \dots, \mc{A}_n$ be subcategories of $\mc{C}$ (all categories are stable). We define $\langle \mc{A}_1, \dots, \mc{A}_n \rangle$ to be the category of diagrams of maps in $\mc{C}$ of the form
\[\begin{tikzcd}
	0 & {a_n} & {a_{n-1, n}} & \cdots & {a_{1,2,\dots,n}} \\
	& 0 & {a_{n-1}} & \cdots & {a_{1,2,\dots, n-1}} \\
	&& 0 & \ddots & \vdots \\
	&&& \ddots & {a_1} \\
	&&&& 0
	\arrow[from=1-1, to=1-2]
	\arrow[from=1-2, to=1-3]
	\arrow[from=1-2, to=2-2]
	\arrow[from=1-3, to=1-4]
	\arrow[from=1-3, to=2-3]
	\arrow[from=1-4, to=1-5]
	\arrow[from=1-4, to=2-4]
	\arrow[from=1-5, to=2-5]
	\arrow[from=2-2, to=2-3]
	\arrow[from=2-3, to=2-4]
	\arrow[from=2-3, to=3-3]
	\arrow[from=2-4, to=2-5]
	\arrow[from=2-4, to=3-4]
	\arrow[from=2-5, to=3-5]
	\arrow[from=3-3, to=3-4]
	\arrow[from=3-4, to=3-5]
	\arrow[from=3-4, to=4-4]
	\arrow[from=3-5, to=4-5]
	\arrow[from=4-4, to=4-5]
	\arrow[from=4-5, to=5-5]
\end{tikzcd}\] where $a_i \in \mc{A}_i$ for all $i$, and all squares are Cartesian.

By the standard theory of the Waldhausen S-construction, we can remove everything except the top row and get an equivalence (see \cite[Proposition 3.1.3]{DKSS}. But the data of the top row is equivalent to \cite[Definition 2.25]{christ}, yielding the following result.

\begin{prop}
    The category $\mc{C}$ has an SOD $\langle \mc{A}_1, \dots, \mc{A}_n \rangle$ iff the functor to $\mc{C}$ given by restriction to the top right vertex is an equivalence.
\end{prop}

\begin{rem}
    The usual Waldhausen S-construction is for $\Fun(A_n, \mc{C})$ for some category $\mc{C}$; so we can think of the top row as picking out the graded pieces of the construction. However our construction above has the graded pieces on the diagonal, so it somewhat different in flavor. It has a fully faithful map into the Waldhausen S-construction for $\Fun(A_n, \mc{C})$.
\end{rem}

\subsection{SODs via inclusion and projection} 
\begin{defn}
Given an SOD $\mc{C}=\langle \mc{D}_1, \mc{D}_2, \dots, \mc{D}_n \rangle$ with inclusions $i_j \colon \mc{D}_j \to \mc{C}$, we say that the sequence of functors $i_j$ exhibits an SOD of $\mc{C}$ by inclusion.
\end{defn}

The reason for this terminology is to introduce the dual notion. Suppose that we are given an SOD $\mc{C}=\langle \mc{D}_1, \mc{D}_2, \dots, \mc{D}_n \rangle$, and suppose that there exist functors $p_j:\mc{C}\to\mc{D}_j$ such that $p_j$ is zero on $\mc{D}_i$ whenever $i\neq j$, and $p_j\circ i_j\simeq \id_{\mc{D}_j}$.

\begin{defn}
    In the above setting, we say that the sequence of functors $p_j \colon \mc{C} \to \mc{D}_j$ exhibits the SOD $\mc{C}=\langle \mc{D}_1, \mc{D}_2, \dots, \mc{D}_n \rangle$ by projection.
\end{defn}

Note that for a fixed $i$, the intersection of $\ker(p_j)$ for $j \neq i$ is equal to $\mc{D}_i$. Thus one can recover the inclusions $i_j$ from the projections $p_j$.

This definition will be helpful for us when we do categorical representation theory for SODs (see \cref{thm:preserve} and \cref{thm:a1gmsod2}).

\begin{rem}
    If an SOD $\mc{C}=\langle \mc{D}_1, \mc{D}_2, \dots, \mc{D}_n \rangle$ is infinitely admissible, we can always exhibit the SOD by projection functors as follows. It may help to observe that $p_n := i_n^R$ (for $i_n \colon \mc{D}_n \to \mc{C})$ satisfies $p_n i_j = 0$ whenever $1\neq j$.

    We apply a mutation by $\delta\in\operatorname{Br}_n$, the half-twist element, to get a new SOD
    $\mc{C} = \langle\mc{D}'_n,\dots,\mc{D}'_1\rangle$, exhibited by inclusion functors $i_j':\mc{D}'_j\to \mc{C}$ (note that $i_n=i_n'$ as $\mc{D}_n=\mc{D}_n'$).

    Then, the functors $p_j := i_j^{'R}:\mc{C}\to \mc{D}'_j\simeq \mc{D}_j$ exhibit the original SOD by projections.

    Similarly, we can take the negative half-twist and get projections as left adjoints to inclusions. 
\end{rem}

We collate the above remarks into the following result, which allows us to be agnostic about the actual SOD to an extent and not choose specific inclusion functors. For simplicity we assume our SODs are infinitely admissible in the following result.
\begin{prop}\label{prop:equivalentprojectionSOD}
Suppose we have a stable category $\mc{C}$ with projection functors $p_i \colon \mc{C} \to \mc{D}_i$. Then the following are equivalent:
\begin{itemize}
    \item The sequence of functors $p_i$ exhibits $\mc{C}=\langle \mc{D}_1, \dots \mc{D}_n \rangle$ by projection (where the components of the SOD are then the intersections of the kernels of all but one $p_i$).
    \item The sequence of functors $p_i^L$ exhibits $\mc{C}=\langle \mc{D}_n, \dots \mc{D}_1 \rangle$ by inclusion.
    \item The sequence of functors $p_i^R$ exhibits $\mc{C}=\langle \mc{D}_n, \dots \mc{D}_1 \rangle$ by inclusion.
\end{itemize}
\end{prop}

\begin{rem}\label{rem:howtoshowprojectionSOD}
    Given a candidate series of projections $p_i \colon \mc{C} \to \mc{D}_i$, to show that they give an SOD by projection, it is enough to show both orthogonality and generation (after e.g. showing that their left, equivalently right, adjoints are inclusions).
    
    For orthogonality, one can use the second or third criteria of \cref{prop:equivalentprojectionSOD} by showing that gluing functors are zero. For $i<j$, the second criterion requires showing that $p_ip_j^L \simeq 0$ (or equivalently $p_j^{LL}p_i^L \simeq 0$), whereas the third criterion requires showing that $p_i^{RR}p_j^R \simeq 0$ (or equivalently $p_jp_i^R \simeq 0$). Of course $(p_ip_j^L)^R=p_jp_i^R$ so the two criteria are equivalent.

    For generation, one can use the first criterion and simply show that the intersection of $\ker(p_i)$ over all $i$ is just the zero object in $\mc{C}$.
\end{rem}

\begin{rem}
    We consider a simple analog of the notion of an SOD via projection. Recall that a decomposition of a vector space $V$ is a collection of subspaces $V_i$ which generate and such that $V_i\cap V_j = 0$ whenever $i\neq j$. This decomposition is exhibited by inclusions $V_i\inclto V$.

    We can also choose to exhibit the decomposition by the natural projections $V\onto V_i$.
\end{rem}

\subsection{The external characterization}
We want a result saying that we can reconstruct the ambient category $\mc{C}= \langle \mc{A}_1, \dots, \mc{A}_n \rangle$ up to equivalence from the data of a ``quasitriangular (co)monad'' over the components $\mc{A}_i$ in the sense of Section 3 of \cite{infrared}, e.g. by saying that $\mc{C}$ is the (op)lax limit (or (op)lax colimit) of this (co)monad. Such a result will allow us to prove abstract equivalences of root stack-type derived categories (and their SODs) when no explicit functor is available.

However for our applications we only need the case where all nonadjacent gluing functors are zero (so all higher coherence data is also zero). Note that for two term SODs $\langle \mc{A}, \mc{B} \rangle$ with a gluing functor $F$, the paper \cite{lax} shows that the lax limit, oplax limit, lax colimit, and oplax colimit of $F$ are all equivalent to $\langle \mc{A}, \mc{B} \rangle$. So we prove a reconstruction result in this case, which only requires taking (op)lax (co)limits over diagrams that are $(\infty, 1)$-categories (rather than $(\infty, 2)$-categories); note however that \cref{lem:length3SOD} additionally uses the $(\infty,2)$-categorical universal property of lax limit.

\begin{lem}\label{lem:proj}
    For $\mc{A}$ a part of an SOD of a category $\mc{C}$, we have $i_{\mc{A}}^Li_{\mc{A}} \simeq i_{\mc{A}}^Ri_{\mc{A}} \simeq \id$ when the adjoints to $i_{\mc{A}}$ exist.
\end{lem}
\begin{proof}
    This follows from the inclusion being fully faithful.
\end{proof}
\begin{lem}\label{lem:sodexact}
    Given $\mc{C}=\langle \mc{A}, \mc{B} \rangle$, with $\mc{A}$ left admissible and $\mc{B}$ right admissible, we have an exact sequence of functors $i_{\mc{B}}i_{\mc{B}}^R \to \id_{\mc{C}} \to i_{\mc{A}}i_{\mc{A}}^L$ from $\mc{C}$ to $\mc{C}$.
\end{lem}
\begin{proof}
    The exact sequence can be checked pointwise, and then it follows from the description of $\mc{C}$ via diagrams of the form \cite[(2.2.1)]{DKSS}, as well as the proof of \cite[Proposition 2.3.2]{DKSS} which identifies the adjoints to the inclusions $i_{\mc{B}}, i_{\mc{B}}$.
\end{proof}

\begin{prop}\begin{enumerate}
    \item 
    Suppose $\mc{C}$ has a length-$2$ SOD $\langle\mc{D}_1, \mc{D}_2 \rangle$, exhibited by inclusions $i_1,i_2$. Then the oplax cocone

\[\begin{tikzcd}
	{\mc{D}_1} && {\mc{D}_2} \\
	& {\mc{C}}
	\arrow["{i_2^Li_1}", from=1-1, to=1-3]
	\arrow[""{name=0, anchor=center, inner sep=0}, "{i_1}"', from=1-1, to=2-2]
	\arrow["{i_2}", from=1-3, to=2-2]
	\arrow[shorten <=15pt, shorten >=15pt, Rightarrow, from=0, to=1-3]
\end{tikzcd}\]

exhibits $\mc{C}$ as an oplax colimit.

\item The oplax cone

\[\begin{tikzcd}
	& {\mc{C}} & \\
	{\mc{D}_1} && {\mc{D}_2}
	\arrow["{i_1^L}"', from=1-2, to=2-1]
	\arrow[""{name=0, anchor=center, inner sep=0}, "{i_2^L}", from=1-2, to=2-3]
	\arrow["{i_2^Li_1}"', from=2-1, to=2-3]
	\arrow[shorten <=15pt, shorten >=15pt, Rightarrow, from=0, to=2-1]
\end{tikzcd}\]

exhibits $\mc{C}$ as an oplax limit.
\end{enumerate}
\end{prop}
\begin{proof}
    This is essentially shown in \cite[Section 2.4]{complexes} (above their Remark 2.4.3). The oplax colimit construction uses $i_1$ and $i_2$ (the latter with a shift by $1$), and the oplax limit construction uses $i_1^L$ (which has kernel $i_2\mc{D}_2$) and $i_2^L$ (which is $\fib(\eta)$ up to a shift by $1$).
\end{proof}
\begin{rem}
    Of course, the analogous statement holds for lax (co)limit using right adjoints and Cartesian gluing functors. 
\end{rem}

\begin{lem}\label{lem:length3SOD}
    Given $\mc{C}=\langle \mc{D}_1, \mc{D}_2, \mc{D}_3 \rangle$, the gluing functor $\mc{D}_1 \to \langle \mc{D}_2, \mc{D}_3 \rangle$ is the canonical functor coming from the lax cone

\[\begin{tikzcd}
	& {\mc{D}_1} & \\
	{\mc{D}_2} && {\mc{D}_3}
	\arrow["{i_2^Li_1}"', from=1-2, to=2-1]
	\arrow[""{name=0, anchor=center, inner sep=0}, "{i_3^Li_1}", from=1-2, to=2-3]
	\arrow["{i_3^Li_2}"', from=2-1, to=2-3]
	\arrow[shorten <=15pt, shorten >=15pt, Rightarrow, from=0, to=2-1]
\end{tikzcd}\]
    
    so $\mc{C}=\oplaxlim(\mc{D}_1 \to \langle \mc{D}_2, \mc{D}_3 \rangle)$.

    Similarly, the gluing functor $\langle \mc{D}_1, \mc{D}_2 \rangle \to \mc{D}_3$ is the canonical functor coming from the lax cocone between $\mc{D}_1, \mc{D}_2, \mc{D}_3$, and then $\mc{C}=\oplaxcolim(\langle \mc{D}_1, \mc{D}_2 \rangle \to \mc{D}_3)$.
\end{lem}
\begin{proof}
    We do the first case as the second is similar.
    Consider the gluing functor $\mc{D}_1 \to \langle \mc{D}_2, \mc{D}_3 \rangle$ in $\mc{C}$. We know that composing this functor with $i_2^L$ or $i_3^L$ yields the gluing functors $i_2^Li_1 \colon \mc{D}_1 \to \mc{D}_2$ and $i_3^Li_1 \colon \mc{D}_1 \to \mc{D}_3$.

    Now note that we have natural isomorphisms $i_3^Li_1\xRightarrow{\sim}i_3^Li_{23}i_{23}^Li_1$ (where $i_{23} \colon \langle \mc{D}_2, \mc{D}_3 \rangle \to \mc{C}$), and $i_3^Li_2i_2^Li_1 \xRightarrow{\sim} i_3^Li_2i_2^Li_{23}i_{23}^Li_1$. 
    
    Then the natural transformation $i_3^Li_1 \Rightarrow i_3^Li_2i_2^Li_1$ is equivalent to $i^L_3i_{23}i^L_{23}i_1 \Rightarrow i^L_3i_2i^L_2i_{23}i^L_{23}i_1$, i.e. there is a lax cone factoring

\[\begin{tikzcd}
	{\mc{D}_1} && \\
	& {\langle \mc{D}_2, \mc{D}_3 \rangle} \\
	& {\mc{D}_2} & {\mc{D}_3.}
	\arrow["{i^L_{23}i_1}", from=1-1, to=2-2]
	\arrow["{i^L_2i_{23}}"', from=2-2, to=3-2]
	\arrow[""{name=0, anchor=center, inner sep=0}, "{i^L_3i_{23}}", from=2-2, to=3-3]
	\arrow["{i^L_3i_2}"', from=3-2, to=3-3]
	\arrow[shorten <=6pt, Rightarrow, from=0, to=3-2]
\end{tikzcd}\]
    
    So by the universal property of oplax limit, the gluing functor $\mc{D}_1 \to \langle \mc{D}_2, \mc{D}_3 \rangle$ is the functor coming from the oplax triangle between $\mc{D}_1, \mc{D}_2, \mc{D}_3$ given by the gluing functors.
\end{proof}

\begin{prop} \label{thm:linearSOD}
    Suppose we have an SOD $\mc{C}=\langle \mc{D}_1, \dots, \mc{D}_n \rangle$ such that the pairwise gluing functor from $\mc{D}_i$ to $\mc{D}_j$ exists for $i<j$, and is zero unless $j=i+1$. Then $\langle \mc{D}_1, \dots, \mc{D}_n \rangle$ can be constructed as $\oplaxlim(\mc{D}_1 \to \langle \mc{D}_2, \dots, \mc{D}_n \rangle)$ for a canonical gluing functor $\mc{D}_1 \to \langle \mc{D}_2, \dots, \mc{D}_n \rangle$ (exhibiting a projection SOD). Similarly, it can be constructed as $\oplaxcolim( \langle \mc{D}_1,  \dots \mc{D}_{n-1} \rangle \to \mc{D}_n) $ for a canonical gluing functor $\langle \mc{D}_1,  \dots \mc{D}_{n-1} \rangle  \to \mc{D}_n$ (exhibiting an inclusion SOD).
    \end{prop}
    \begin{proof}
    We do the first case as the second is similar. We use induction on $n$. Then we know that there is a unique reconstruction $\langle \mc{D}_3, \dots,\mc{D}_n \rangle$. Since the gluing functor from $\mc{D}_1$ to each $\mc{D}_j$ for $j \ge 3$ is zero, there are no homs between $\mc{D}_1$ and such $\mc{D}_j$. Thus there are no homs either between $\mc{D}_1$ and $\langle \mc{D}_3, \dots,\mc{D}_n \rangle$, so the gluing functor $\mc{D}_1 \to \langle \mc{D}_3, \dots,\mc{D}_n \rangle$ is zero.

    Then we are reduced to the $n=3$ case and are done by \cref{lem:length3SOD}.
    \end{proof}

\begin{rem}
    We believe that a more general reconstruction result holds: namely, one can take an oplax limit over a general quasitriangular monad (not just a linear SOD). In other words, one takes an oplax limit over the $(\infty, 2)$-diagram $\Delta_n^{(\infty, 2)}$, which is like the usual simplex $\Delta_n$ except we have noninvertible $2$-morphisms. One could then prove canonical reconstruction by showing that this oplax limit is equivalent to an iterated oplax limit of just two categories at a time, as suggested in \cite[Remark 2.4.4]{complexes}, using the universal property of lax limits from \cite[Corollary 5.1.7]{laxlim}. Note also that the analog of \cref{lem:computeafterforgetting} holds in this case using the fact from \cite{ahm} that (op)lax (co)limits in $(\infty, 2)$-categories are weighted (co)limits. However we do not check the details of this method carefully as it is not necessary for our purposes.
\end{rem}

\subsection{SODs for small versus large categories}
\begin{prop}\label{prop:stvsst}
    Suppose we have an SOD $\mc{C}=\langle \mc{A}_1, \dots, \mc{A}_n \rangle$ in $\st$. Then we have an SOD $\Ind(\mc{C})=\langle \Ind(\mc{A}_1), \dots, \Ind(\mc{A}_n) \rangle$ in $\St$.
\end{prop}
\begin{proof}
    We will do the $n=2$ case; the general case follows from induction. First we claim that small colimits taken in $\Ind(\mc{A}_2)^\perp$ are the same as in $\Ind(\mc{C})$. To show this, suppose we have a small colimit of objects $c \in \Ind(\mc{A}_2)^\perp$. Then for any $a \in \mc{A}_2$, we have $\Hom(a, c) \simeq 0$, and $\Hom(a, -)$ preserves small colimits, so we have $\Hom(a, \colim c)=\colim \Hom(a, c) \simeq 0$.

    Using this fact, we get that $\mc{A}_1 \subset (\Ind(\mc{A}_2)^\perp)^c$, as homs out of $\mc{A}_1$ preserve small colimits in $\Ind(\mc{C})$, and thus in $\Ind(\mc{A}_2)^\perp$.

    Since $\Ind(\mc{C})$ is compactly generated by $\mc{C}$, it is compactly generated by the collection of elements in $\mc{A}_1$ and $\mc{A}_2$, implying that $\Ind(\mc{A}_2)^\perp$ is compactly generated by $\mc{A}_1$. So we get that $\Ind(\mc{A}_1) = \Ind(\mc{A}_2)^\perp$ as subcategories of $\mc{C}$, yielding the result.
\end{proof}

\subsection{Mutation and gluing functors}
We now compute how gluing functors change under mutation. Note that the formulas match \cite[3.4.12]{infrared}.

\begin{prop}\label{thm:mutation}
    Given a semiorthogonal decomposition $\langle \mc{A}, \mc{B}, \mc{C} \rangle$, if we mutate $\mc{C}$ over $\mc{B}$ to $\mc{C}'$, then under the equivalence $i_{\mc{C}'}^Li_{\mc{C}} \colon \mc{C} \xrightarrow{\sim} \mc{C}'$, the new gluing functor from $A$ to $\mc{C}$ is $\fib(i_{\mc{C}}^Li_{\mc{A}} \to i_{\mc{C}}^Li_{\mc{B}}i_{\mc{B}}^Li_{\mc{A}})$. 

    Given a semiorthogonal decomposition $\langle \mc{B}, \mc{C}, \mc{D} \rangle$, if we mutate $C$ over $\mc{B}$ to $\mc{C}'$, then under the equivalence $i_{\mc{C}'}^Li_{\mc{C}} \colon \mc{C} \xrightarrow{\sim} \mc{C}'$, the new gluing functor from $\mc{C}$ to $\mc{D}$ is $\cofib(i_{\mc{D}}^Li_{\mc{B}}i_{\mc{B}}^Ri_{\mc{C}} \to i_{\mc{D}}^Li_{\mc{C}})$.
\end{prop}
\begin{proof}
    We prove the first statement as the second is similar. On $\langle \mc{B}, \mc{C} \rangle$, using we have \cref{lem:sodexact} we have an exact sequence $i_{\mc{C}}i_{\mc{C}}^R \to \id \to i_{\mc{B}}i_{\mc{B}}^L$. So our cone $\fib(i_{\mc{C}}^Li_{\mc{A}} \to i_{\mc{C}}^Li_{\mc{B}}i_{\mc{B}}^Li_{\mc{A}})$ is then $i_{\mc{C}}^Li_{\mc{C}}i_{\mc{C}}^Ri_{\mc{A}}=i_{\mc{C}}^Ri_{\mc{A}}$.

    Now we claim that $i_{\mc{C}}^R=i_{\mc{C}'}^L$ (or more precisely $i_{\mc{C}'}^Li_{\mc{C}}i_{\mc{C}}^R=i_{\mc{C}'}^L$). The cone between these two functors is $i_{\mc{C}'}^Li_{\mc{B}}i_{\mc{B}}^L$. But $i_{\mc{C}'}^Li_{\mc{B}}=0$ by \cite[Proposition 2.3.2(a)]{DKSS}.
\end{proof}

\section{Categorical representation theory for SODs}
In this section we develop the basic theory of SODs where each component is now a module for a monoidal category.

Let $\mc{A}$ be a monoidal category in $\St$ or $\st$.

\begin{defn}
    An $\mc{A}$-module $\mc{C}$ has a $2$-term SOD of $\mc{A}$-modules $\mc{D}_1, \mc{D}_2$ if there is an SOD after forgetting the $\mc{A}$-module structure, and the left and right adjoints, respectively, of the inclusions $\mc{D}_1 \to \mc{C}$ and $\mc{D}_2 \to \mc{C}$ (of respectively left and right admissible subcategories) are $\mc{A}$-module functors. An $\mc{A}$-module $\mc{C}$ has an $n$-term SOD of $\mc{A}$-modules $\mc{D}_1, \dots, \mc{D}_n$ if there is an SOD after forgetting the $\mc{A}$-module structure, and there is some casting of the $n$-term SOD of a repeated $2$-term SOD of $\mc{A}$-modules as above.
\end{defn}

All the SODs we use will be infinitely admissible and have adjoints (that are module functors) to all possible inclusion and gluing functors, so it is possible that a more restrictive definition is better in general. For simplicity, we assume our SODs are infintely admissible in the rest of this section.

\begin{prop}\label{prop:adjpreserve}
    Suppose we are given $\mc{A}$-modules $\mc{D}_1$ and $\mc{D}_2$ with adjoint functors between them (in either the small or large cases). For a right $\mc{A}$-module $\mc{M}$ or a left $\mc{A}$-module $\mc{N}$, applying the operation $\Hom_{\mc{A}}(\mc{N}, -),$ $\Hom_{\mc{A}}(-, \mc{N}),$ or $\mc{M} \otimes_{\mc{A}} (-)$ to this adjoint pair yields an adjoint pair (and keeps smallness/largeness).
\end{prop}
\begin{rem}
    Recall that for functor categories between small categories we always take exact functors, and for large categories we take continuous (i.e. colimit-preserving) functors. Also the functor category between small categories is still small.
\end{rem}
\begin{proof}[Proof of \cref{prop:adjpreserve}]
    First let us do the $\Hom$ case. Given $F \colon \mc{D}_1 \to \mc{D}_2$, we define $\Hom_{\mc{A}}(\mc{N}, \mc{D}_1) \to \Hom_{\mc{A}}(\mc{N}, \mc{D}_2)$ pointwise. Checking the adjunction boils down to checking diagrams involving the unit and counit (see \cite[I.1.4.4]{dag}), which is straightforward. The other Hom case is similar.

    Now let us do the $\otimes$ case. Then we know that $\Ind$ commutes with $\otimes$. In the large case it is known that $\mc{M} \otimes_{\mc{A}} (-)$ is a $2$-functor, so it preserves adjunctions. Then the small case follows from taking compact objects (note that a functor preserves adjoints if its right adjoint is continuous, which is preserved under $\otimes_{\mc{A}}$).
\end{proof}
\begin{rem}
    This proof also shows that gluing functors of an SOD are preserved under $\otimes_{\mc{A}}$ and $\Hom_{\mc{A}}$ (using the result \cref{thm:preserve} that these operations indeed yield an SOD).
\end{rem}

\begin{prop}\label{thm:preserve}
    Suppose we are given an $\mc{A}$-module SOD $\mc{C}=\langle \mc{D}_1, \dots, \mc{D}_n \rangle$, with inclusion functors $i_j \colon \mc{D}_j \to \mc{C}$.

    \begin{enumerate}
        \item For a right $\mc{A}$-module $\mc{M}$, applying the operation $\mc{M} \otimes_{\mc{A}} (-)$ yields an SOD of $\mc{M} 
        \otimes_{\mc{A}} \mc{C}$ with inclusion functors induced by the tensor product.
        \item For a left $\mc{A}$-module $\mc{N}$, applying the operation $\Hom_{\mc{A}}(-, \mc{N})$ yields an SOD of $\Hom_{\mc{A}}(\mc{C}, \mc{N})$ by projection, with projection functors $(- \circ i_j) \colon \Hom_{\mc{A}}(\mc{C}, \mc{N}) \to \Hom_{\mc{A}}(\mc{D}_i, \mc{N})$.
    \end{enumerate}

    Suppose instead we are given an $\mc{A}$-module SOD $\mc{C}$ via projection functors $p_i \colon \mc{C} \to \mc{D}_i$.

    \begin{enumerate}
        \item For a right $\mc{A}$-module $\mc{M}$, applying the operation $\mc{M} \otimes_{\mc{A}} (-)$ yields an SOD of $\mc{M} 
        \otimes_{\mc{A}} \mc{C}$ via projection, with projection functors induced by the tensor product.
        \item For a left $\mc{A}$-module $\mc{N}$, applying the operation $\Hom_{\mc{A}}(-, \mc{N})$ yields an SOD of $\Hom_{\mc{A}}(\mc{C}, \mc{N})$ by inclusion, with inclusion functors $(- \circ p_j) \colon \Hom_{\mc{A}}(\mc{D}_i, \mc{N}) \to \Hom_{\mc{A}}(\mc{C}, \mc{N})$.
    \end{enumerate}
\end{prop}

\begin{proof}
    Note that by \cref{prop:adjpreserve}, the operations $\Hom_{\mc{A}}(-, \mc{N})$ and $\mc{M} \otimes_{\mc{A}} (-)$ preserve (co)localizations, i.e. adjoints with a one-sided inverse are preserved. So in particular, $\mc{M} \otimes_{\mc{A}} (-)$ preserves inclusions and projections (adjoints to inclusions) in the SODs above, and $\Hom_{\mc{A}}(-, \mc{N})$ switches inclusion and projection functors. So we just need to show orthogonality and generation in each case.

    \begin{enumerate}
        \item  \begin{enumerate}
            \item We use compact generation of tensor products (see \cite[I.1.7.4]{dag}). Then $\mc{M} \otimes_{\mc{A}} \mc{C}$ is compactly generated by $m \otimes c$ for $m, c$ compact. Each $c$ fits into an exact triangle with objects from $\mc{D}_1$ and $\mc{D}_2$. Thus $\mc{C}$ is compactly generated by the $\mc{M} \otimes_{\mc{A}} \mc{D}_i$ giving the result.
            \item We use \cref{rem:howtoshowprojectionSOD}. For $j<k$, we have $$(- \circ i_j)(- \circ i_k)^L = (- \circ i_k^Ri_j)=0,$$
            giving orthogonality.

            For generation, if we have a map $\mc{C} \to \mc{N}$ such that each $\mc{D}_i \to \mc{C} \to \mc{N}$ is zero, then clearly $\mc{C} \to \mc{N}$ is zero.
        \end{enumerate}
        \item \begin{enumerate}
            \item After taking adjoints the result is equivalent to part 1a.
            \item After taking adjoints the result is equivalent to part 1b. Note that for $j<k$, we have $$(- \circ p_k)^R(- \circ p_j) = (- \circ p_jp_k^L)=0,$$
            giving orthogonality in the correct direction.
        \end{enumerate}
    \end{enumerate}
\end{proof}
Of course, the same proof shows that $\Hom_{\mc{A}}(\mc{N}, -)$ preserves SODs by inclusion and projection.

We can now easily calculate the dual module to an $\mc{A}$-module given an SOD with dualizable components.

\begin{cor}\label{cor:dualexists}
An $\mc{A}$-module $\mc{M}$ that possesses an (infinitely admissible) SOD with $\mc{A}$-dualizable components is $\mc{A}$-dualizable.

In particular, suppose we have an $\mc{A}$-module SOD of $\mc{M}$ via inclusion functors $i_j \colon \mc{D}_j \to \mc{M}$. and that each $\mc{D}_i$ is dualizable as an $\mc{A}$-module (with dual then given by $\Hom_{\mc{A}}(\mc{D}_i, \mc{A})$). Then $\mc{M}$ has a dual $\Hom_{\mc{A}}(\mc{M}, \mc{A})$ with an SOD via projection functors $(- \circ i_j)$.
\end{cor}
\begin{proof}
    For an arbitrary $\mc{A}$-module $\mc{C}$, we can calculate \begin{align*}\Hom_{\mc{A}}(\mc{M}, \mc{C})\simeq& \Hom_{\mc{A}}(\langle \mc{D}_1, \dots, \mc{D}_n\rangle, \mc{C})\\
    \simeq& \langle \Hom_{\mc{A}}(\mc{D}_n, \mc{C}), \dots, \Hom_{\mc{A}}(\mc{D}_1, \mc{C})\rangle\\
    \simeq& \langle \Hom_{\mc{A}}(\mc{D}_n, \mc{A}) \otimes_{\mc{A}} \mc{C}, \dots, \Hom_{\mc{A}}(\mc{D}_1, \mc{A}) \otimes_{\mc{A}} \mc{C} \rangle\\
    \simeq& \langle \Hom_{\mc{A}}(\mc{D}_n, \mc{A}), \dots, \Hom_{\mc{A}}(\mc{D}_1, \mc{A})\rangle \otimes_{\mc{A}} \mc{C} \\
    \simeq& \Hom_{\mc{A}}(\mc{M}, \mc{A}) \otimes_{\mc{A}} \mc{C}.\end{align*}

    The relevant projection functors are obtained via \cref{thm:preserve}.
\end{proof}

In the $\Hom$ case, we can also give a more abstract proof of \cref{thm:preserve} using (op)lax (co)limits, which we explain one version of below.

\begin{prop}\label{thm:preserve2}
    Suppose we are given an $\mc{A}$-module SOD $\mc{C}=\langle \mc{D}_1, \dots, \mc{D}_n \rangle$, with inclusion functors $i_j \colon \mc{D}_j \to \mc{C}$.

    Suppose that we have gluing functors $F_{ij} \colon \mc{D}_i \to \mc{D}_j$ (recall that $F_{ij}=i_j^Li_i$), and take a left $\mc{A}$-module $\mc{N}$.
        \begin{enumerate}
            \item  The category $\Hom_{\mc{A}}(\mc{N}, \mc{C})$ has an SOD 
            $$\langle \Hom_{\mc{A}}(\mc{N}, \mc{D}_1), \dots, \Hom_{\mc{A}}(\mc{N}, \mc{D}_n) \rangle $$ with inclusion functors $(i_j \circ -) \colon \Hom_{\mc{A}}(\mc{N}, \mc{D}_j) \to \Hom_{\mc{A}}(\mc{N}, \mc{C})$ and gluing functors $(F_j \circ -)$.
            \item The category $\Hom_{\mc{A}}(\mc{C}, \mc{N})$ has an SOD $$\langle \Hom_{\mc{A}}(\mc{D}_n, \mc{N}), \dots, \Hom_{\mc{A}}(\mc{D}_1, \mc{N}) \rangle,$$ with inclusion functors $(- \circ i_j^L) \colon \Hom_{\mc{A}}(\mc{D}_j, \mc{N}) \to \Hom_{\mc{A}}(\mc{C}, \mc{N})$ and gluing functors $(- \circ F_{ij}) \colon \Hom_{\mc{A}}(\mc{D}_j, \mc{N}) \to \Hom_{\mc{A}}(\mc{D}_i, \mc{N})$. (Most naturally, it turns an inclusion SOD into a projection SOD (with projections $(- \circ i_j)$) with reversed order.)
        \end{enumerate}
\end{prop}

\begin{proof}
    By induction we only need to show this for $n=2$. Then note that by \cref{prop:adjpreserve}, the operations $\Hom_{\mc{A}}(\mc{N}, -)$ and $\Hom_{\mc{A}}(-, \mc{N})$ preserve (co)localizations, i.e. adjoints with a one-sided inverse are preserved. So the fact that $\mc{D}_1$ includes into $\mc{C}$ with a right adjoint and $\mc{D}_2$ includes into $\mc{C}$ with a left adjoint is preserved.

        \begin{enumerate}
            \item We view the SOD as a projection SOD and use the universal property of the oplax limit (using the fact that it is a weighted limit, see\cite[Definition 2.9]{ghn}): $$\Hom_{\mc{A}}(\mc{N}, \oplaxlim(\mc{D}_1, \mc{D}_2))\simeq \oplaxlim(\Hom_{\mc{A}}(\mc{N}, \mc{D}_1), \Hom_{\mc{A}}(\mc{N}, \mc{D}_2).$$

    The oplax limit cone with projection maps $p_1, p_2$ becomes an oplax limit cone with projection maps $(p_1 \circ -), (p_2 \circ -)$. Note that the inclusion maps are the right adjoints to these, which are $(i_1 \circ -)$ and $(i_2 \circ -)$.
    \item Let us check that the description makes sense. Indeed, we know that $F_{ij}=i_j^Li_i$ by definition, and we want to show that $(- \circ F_{ij})=(- \circ i_i^L)^L(- \circ i_j^L)$. The latter functor is $(- \circ i_i)(- \circ i_j^L)=(- \circ i_j^Li_i)=(- \circ F_{ij})$.
    
    Now we use the fact from \cite{lax} that we can write $\mc{C}=\langle \mc{D}_1, \mc{D}_2 \rangle$ as $\oplaxcolim(\mc{D}_1, \mc{D}_2)$, as well as the universal property of (op)lax colimits (using the fact that it is a weighted limit, see\cite[Definition 2.9]{ghn}).
    
    Then we have $$\Hom_{\mc{A}}(\oplaxcolim(\mc{D}_1, \mc{D}_2), \mc{N}) \simeq \oplaxlim(\Hom_{\mc{A}}(\mc{D}_2, \mc{N}), \Hom_{\mc{A}}(\mc{D}_1, \mc{N})).$$

    So this operation most naturally turns an (inclusion) SOD with maps $i_j$ into a projection SOD with maps $(- \circ i_j)$; the equivalent inclusion SOD has maps $(- \circ i_j)^R=(- \circ i_j^L)=(- \circ p_j)$.
        \end{enumerate}
\end{proof}

Note that the above proof uses the following lemma.

\begin{lem}\label{lem:computeafterforgetting}
    (Op)lax (co)limits in $\mc{A}-\Mod(\St)$ can be computed in $\St$ (i.e. they commute with the forgetful functor).
\end{lem}
\begin{proof}
Since the forgetful functor is both a left and right adjoint, it preserves weighted (co)limits, of which (op)lax (co)limits are an example (see \cite[Definition 2.9]{ghn}).
\end{proof}

\subsection{Extending \cite{BZNP} to DM stacks} \label{subsect:bznpextend}

In this section, for completeness, we indicate how to extend \cite[Theorem 3.0.2]{BZNP} from relative algebraic spaces to relative DM stacks. We somewhat follow the strategy given in the footnote to \cite[Remark 3.0.4]{BZNP}.

We need to show the following (if $\mc{Z} \to S$ is a Noetherian separated $S$-algebraic space (without assuming further finiteness conditions), this is Proposition 3.0.11 in \cite{BZNP}).
\begin{prop}\label{prop:properdm} Suppose $p \colon \mc{Z} \to S =\Spec A$ is a separated (derived) DM stack of finite type and with finite diagonal, over a Noetherian base $S=\Spec A$ (both living over a field of characteristic $0$). Then the following conditions on $\mc{F} \in \QC(\mc{Z})$ are equivalent:
\begin{enumerate}
    \item $\mc{F} \in \Coh(\mc{Z})$ with support proper over $S$\footnote{Assuming that $\mc{F}$ is coherent, we define proper support to mean that $\pi_*(\mc{P} \otimes \mc{F})$ has proper support for all $\mc{P} \in \Perf(\mc{Z})$. The usual definition for $\text{Supp}(\mc{F})$ is the closed substack defined by $\mc{I}:= \ker(\mc{O}_{\mc{Z}} \to \mc{H}om(\mc{F}, \mc{F}))$, which also defines a closed subscheme $Z' \xhookrightarrow{} Z$ since $\mc{I}$ descends to the coarse space. Set theoretically, $Z'$ contains points $p$ such that $\mc{F}|_{\{p\} \times_Z \mc{Z}} \neq 0$. We note that our definition of having proper support agrees with the statement that $Z'$ is proper over $S$. };

    \item $\Gamma(\mc{Z}, \mc{P} \otimes \mc{F}) \in \Coh S$ for all $\mc{P} \in \Perf(\mc{Z})$.
\end{enumerate}
\end{prop}
\begin{proof}
    We will use the map $\pi \colon \mc{Z} \to Z$ of $\mc{Z}$ to its coarse space to reduce to Proposition 3.0.11 of \cite{BZNP}. Let $q \colon Z\to S$ so $p=q \circ \pi$. First note that (1) implies (2) since objects in $\Perf \mc{Z}$ have finite Tor-amplitude and proper pushforward preserves coherent objects (note that $\pi_*$ is locally taking invariants, so it preserves coherence).

    We just need to show that (2) implies (1). Assuming (2) holds for a quasicoherent sheaf $\mc{F} \in \QC(\mc{Z})$, we know it holds for $\mc{P}\otimes \mc{F}$ for $\mc{P} \in \Perf(\mc{Z})$. Then by the projection formula, (2) holds for $\pi_*(\mc{P} \otimes \mc{F})$ with respect to $q \colon Z \to S$, meaning that $\pi_*(\mc{P} \otimes \mc{F}) \in \Coh(Z)$ with support proper over $S$ by \cite[Proposition 3.0.11]{BZNP}.

    For coherence, we use the existence of a compact generator\footnote{We thank GPT 5.6-Sol for suggesting this argument.} $\mc{R}$ of $\QC(\mc{Z})$ by \cite[Theorem A]{hallrydh} (generalizing \cite[Lemma 3.0.8]{BZNP} to DM stacks). This is an \'etale-local and classical statement (and $\mc{R}$ is still a generator on an \'etale cover), so (e.g. by \cite[Lemma 2.2.3]{av}; see also \cite{keelmori}) we can assume the map $\pi$ is $(\Spec A)/G \to \Spec A^G$ for a finite flat group scheme $G$ and ordinary ring $A$. Let $M$ be the $G$-equivariant $A$-module corresponding to $\mc{F}$. Then since $\mc{R}$ is a (weak) generator of $\QC(\Spec A/G)$, it is also a classical generator of the subcategory of perfect=compact objects as $\QC(\Spec A/G)$ is compactly generated (see \cite{Neeman}, especially Lemma 2.2). Thus we can build the perfect module $A \otimes k[G]$ out of $\mc{R}$ (via finite direct sums, shifts, direct summands, and cones), meaning that $((A \otimes k[G])^\vee \otimes M)^G \simeq M$ is finitely generated over $A^G$. Thus $M$ is finitely generated over $A$, so $\mc{F}$ is coherent.\end{proof}

    Note that we only use the finite diagonal and locally finite type assumptions to use the coarse moduli space description.

We then have the following result.

\begin{cor}\label{cor:dmhom}
Suppose that $S$ is a perfect stack in characteristic $0$; that $p_X \colon X \to S$ is a quasi-compact
and separated $S$-relative DM stack locally of finite type with finite diagonal; and that $Y$ is a locally finite type $S$-stack. The $*$-integral transform construction provides an equivalence
$$\Coh_{\text{prop}/Y}(X \times_S Y) \xrightarrow{\sim} \Hom^{\text{ex}}_{\Perf S}(\Perf X, \Coh Y).$$
\end{cor}
\begin{proof}
    The proof is identical to that of \cite[Theorem 3.0.2]{BZNP}, except we use \cref{prop:properdm} in place of \cite[Proposition 3.0.11]{BZNP}.
\end{proof}

\section{Indcoherent convolution categories and singular support}

In this section we use the shriek pullback (i.e. $!$-pullback) for categories of coherent sheaves and convolution categories. As our $\Am$ (with $*$-convolution) is equivalent to the analogous category with $!$-convolution (via a relative Serre duality), the results in this section still apply to $\Am$. 

\subsection{Singular support and tensor prodcuts}\label{subsect:singsupp}

For a quasismooth QCA stack $X$, a coherent sheaf has a singular support (defined in \cite{singsupp}) in $\Sing(X)=T^{*-1}X$. Given a map $p \colon X \to Y$, we have a correspondence

\[\begin{tikzcd}
	& {T^{*-1}Y \times_Y X} & \\
	{T^{*-1}X} && {T^{*-1}Y}
	\arrow["{dp^*}", from=1-2, to=2-1]
	\arrow["{\tilde{p}}"', from=1-2, to=2-3]
\end{tikzcd}\]

with maps of sets $p^!=dp^* \circ \tilde{p}^{-1} \colon T^{*-1}Y \to T^{*-1}X$ and $p_*=\tilde{p} \circ (dp^*)^{-1} \colon T^{*-1}X \to T^{*-1}Y$. Let $\supp$ denote the singular support of a coherent sheaf. Then we have the following.

\begin{prop}\textup{\cite[Proposition 7.1.3 and Lemma 8.3.2]{singsupp}}
    We have $\supp p^!(\mc{F}) \subset p^! (\supp \mc{F})$. For $p$ schematic we have $\supp p_*(\mc{F}) \subset p_*(\supp \mc{F})$.
\end{prop}

Let $p \colon X \to Y$ be a proper surjective schematic morphism of smooth QCA stacks over a field $k$ of characteristic zero.

Let $U$ and $V$ be smooth, with maps $f \colon U \to Y, g \colon V \to Y$.

\begin{defn}
    For $\Lambda_{12} \subset \Sing(U \times_{Y} X)$ and $\Lambda_{23} \subset \Sing(X \times_Y V)$, we have a convolution product $\Lambda_{12} * \Lambda_{23}$ defined as $\pi_{13,*}(\pi_{12}^!\Lambda_{12} \cap \pi_{23}^!\Lambda_{23})$.
\end{defn}

\begin{prop}\label{prop:lambdasuperset}
    For $U=X$ and $\Lambda_{12}=\Sing(X \times_Y X)$, we have $\Lambda_{12} * \Lambda_{23} \supset \Lambda_{23}$.
\end{prop}
\begin{proof}
    We check the assertion on geometric $k$-points, following the strategy of \cite[Section 3.3.1]{spectralincarnation}. Let $A=X \times_Y X \times X \times_Y X, B=X \times_Y X 
    \times_Y X, C=X \times_Y V$. Then $\Lambda_{12} \boxtimes \Lambda_{23}$ is a subset of $\Sing(A)$; we want to determine the image after we pull this back to $\Sing(B)$ and then pushforward to $\Sing(C)$.

    A geometric point of $A$ is a tuple $(x_1, x_2, x_3 \in X; v \in V; y_1, y_2 \in Y)$ such that $p(x_1)=p(x_2)=y_1$ and $p(x_3)=g(v)=y_2$. Then over this point, the fiber of $\Sing(A)$ is $\{a_1 \in \Omega_Y|_{y_1}, a_2 \in \Omega_Y|_{y_2} :  dp^*_{x_1}a_1=dp^*_{x_2}a_1=0, dp^*_{x_3}a_2=dg^*_{v}a_2=0\}.$ The subset $\Lambda_{12} \boxtimes \Lambda_{23}$ comes from an additional condition for $a_2$ imposed by $\Lambda_{23}$.

    We have a correspondence diagram 
\[\begin{tikzcd}
	& {T^{*-1}A \times_A B} && {T^{*-1}C\times_C B} & \\
	{T^{*-1}A} & {} & {T^{*-1}B} && {T^{*-1}C.}
	\arrow["{\tilde{\Delta}}", from=1-2, to=2-1]
	\arrow["{d\Delta^*}"', from=1-2, to=2-3]
	\arrow["{dp_{13}^*}", from=1-4, to=2-3]
	\arrow["{\tilde{p_{13}}}"', from=1-4, to=2-5]
\end{tikzcd}\]

We first find the pullback of $\Lambda_{12} \boxtimes \Lambda_{23}$ to $T^{*-1}A \times_A B$. Geometric points of the latter are the same as for $T^{*-1}$ except with the extra condition that $x_2=x_3$, and the image of the pullback $\tilde{\Delta}^{-1}$ from $T^{*-1}A$ to $T^{*-1}A \times_A B$ is just points satisfying this condition.

To find geometric points of $T^{*-1}B$, we note that the base $B$ is the same as $A$ but with the $x_2=x_3$ condition (and then $y_1=y_2$; let this be $y$). The fiber of $T^{*-1}$ over a given point is now $\{a_1, a_2 \in \Omega_Y|_y :  dp^*_{x_1}a_1=0; dp^*_{x_2}a_1=dp^*_{x_3}a_2; dg^*_{v}a_2=0\}.$ So the composition $\Delta^!(\Lambda_{12} \boxtimes \Lambda_{23})=d\Delta^*(\tilde{\Delta}^{-1}(\Lambda_{12} \boxtimes \Lambda_{23}))$ yields points with $dp^*_{x_2}a_1=dp^*_{x_3}a_2=0$ (along with the condition on $\Lambda_{23}$, which is a condition on $x_2=x_3, v, y, a_2$).

Now we tackle the right half of the diagram. A geometric point of $C$ is a tuple $(x \in X, v \in V, y \in Y)$ such that $p(x)=g(v)=y$. The geometric points of the fiber of $T^{*-1}$ over a given point form the set $\{a \in \Omega_Y|_y : dp^*_xa=dg^*_va=0.\}$. The only difference with $T^{*-1}C \times_C B$ is that we have an extra $x' \in X$ with $p(x')=y$.

The pullback $(dp_{13}^*)^{-1}$ from $T^{*-1}B$ to $T^{*-1}C \times_C B$ of $\Delta^!(\Lambda_{12} \boxtimes \Lambda_{23})$ yields points with $(x, x', v, y, a)$ such that $a=a_1=a_2$, along with the $\Lambda_{23}$ condition (which is now a condition on $(x', y, v, a)$). Finally the pushforward $\tilde{p_{13}}$ just yields points $(x, v, y, a)$. Note that the $\Lambda_{23}$ condition now applies only to $v, y, a$, which is less strict than adding a condition on $x$ (because the $\Lambda_{23}$ condition on $x$ could have been satisfied by e.g. just taking $x=x'$). 
\end{proof}

We have a monoidal product of $\IndCoh(X \times_Y X)$ by convolution (by shriek pullback and star pushforward). Then any $\IndCoh(X \times_Y V)$ is a module over $\IndCoh(X \times_Y X)$.

The end of the proof  of \cref{cor:lambdapreserve} gives the following corollary.
\begin{cor}\label{cor:lambdapreserve}
If $\Lambda_{23}$ is ``independent of $X$'', i.e. (in the terminology of the proof of \cref{prop:lambdasuperset}) $(x, v, y, a) \in \Lambda_{12}$ implies $(x', v, y, a) \in \Lambda_{12}$ for any $x' \in X$, then $\Lambda_{12} * \Lambda_{23} = \Lambda_{23}$. In this case, $\IndCoh_{\Lambda_{23}}(X \times_Y V)$ is an $\IndCoh(X \times_Y X)$-module. In particular, we can take $\Lambda_{23}$ to be everything (giving $\IndCoh(X \times_Y V)$) or the zero section (giving $\QC(X \times_Y V)$).
\end{cor}

For arbitrary (smooth) $U, V$, and $\Lambda_{12}, \Lambda_{23}$ ``independent of $X$'' (in the setting of \cref{cor:lambdapreserve}), we can ask what the tensor product $\IndCoh_{\Lambda_{12}}(U \times_Y X) \otimes_{\IndCoh(X \times_Y X)} \IndCoh_{\Lambda_{23}}(X \times_Y V)$ is. Then we have the following result, which (after \cref{cor:lambdapreserve}) is a basic application of the method used to prove \cite[Theorem 3.3.1]{spectralincarnation}; note that $U, V$ do not actually need to be proper over $Y$.

\begin{thm}\textup{\cite[Proposition 3.30]{BZCHN}\cite[Theorem 3.2.13]{CD}}
    The tensor product $$\IndCoh_{\Lambda_{12}}(U \times_Y X) \otimes_{\IndCoh(X \times_Y X)} \IndCoh_{\Lambda_{23}}(X \times_Y V)$$ is equivalent to $$\IndCoh_{\Lambda_{12}*\Lambda_{23}}(U \times_Y V).$$
\end{thm}

Note that in the language of the references, ``independent of $X$'' means ``$Z_{22}$-stable''.

Using the technique of \cref{prop:lambdasuperset} we can also compute some cases of $\Lambda_{12} * \Lambda_{23}$.
The same proof as in \cref{prop:lambdasuperset} tells us that $\Lambda_{12} * \Lambda_{23}$ in general is $(u \in U, v \in V, y \in Y, a \in \Omega_Y|_y)$ with $f(u)=g(v)=y, df^*_ua=dg^*_va=0$ along with the condition that there is some $x \in X$ such that $(u, x, y, a) \in \Lambda_{12}$ and $(x, v, y, a) \in \Lambda_{23}$. Note that if $\Lambda_{12}$ is the entire $\Sing(X \times_Y X)$ then as long as there is some $u$ with $df^*_ua=0$, the condition is just the $\Lambda_{23}$ condition (including the existence of $x$ for which $dp^*_xa=0$). In particular, using surjectivity of $X \to Y$, we have the following:

\begin{cor}\label{cor:QCtensor}
    The tensor product $$\IndCoh_{\Lambda_{12}}(U \times_Y X) \otimes_{\IndCoh(X \times_Y X)} \QC(X \times_Y V)$$ is equivalent to $\QC(U \times_Y V)$ for any $\Lambda_{12}$ containing the zero section.
\end{cor}

We also have another simple corollary using the above calculation.

\begin{cor}\label{cor:sameconormal}
    Suppose $X, Z$ are smooth, with proper surjective maps to $Y$. If $N_X^*Y=N_Z^*Y$ as subsets of $T^*Y$, then 
    $$\IndCoh(Z \times_Y X) \otimes_{\IndCoh(X \times_Y X)} \IndCoh(X \times_Y Z) \simeq \IndCoh(Z \times_Y Z).$$
\end{cor}
\begin{proof}
    We just need to show that $\Sing(Z \times_Y X) * \Sing(X \times_Y Z)$ is the entire $\Sing(Z \times_Y Z)$. Using the above criterion, we just need to show that given $(u,v,y,a)$ (i.e. the criterion for $\Sing(Z \times_Y Z)$), there always exists some $x$ (the criterion for $\Sing(Z \times_Y X) * \Sing(X \times_Y Z)$). But we know that $a \in N_Z^*Y$ if we can pull it back to $u$ and $v$ and get zero. So there must exist some $x$ for which this is also true since $a \in N_X^*Y$.
\end{proof}

We can use this result to show that $\Coh(X \times_Y Z)$ is dualizable over $\Coh(Z \times_Y Z)$ in this case, which is shown in \cref{thm:gendual}.

\subsection{Hom formula}

Let $\mc{A}=\IndCoh(X \times_Y X)$ as a monoidal category, for $X \to Y$ a proper schematic surjective morphism of QCA stacks. We recall from \cite{BZCHN} that $\mc{A}$ is rigid. So a left $\mc{A}$-module $\mc{M}$ is dualizable if and only if the underlying category $\mc{M}$ is dualizable. Then the $\mc{A}$-dual is identified with $\mc{M}^{\vee}_{RR}$.

It turns out that the right $\mc{A}$-module $\mc{M}^{\vee}_{RR}$ is equivalent to the right $\mc{A}$-module $\mc{M}^{\vee}$. Indeed, $\mc{M}^{\vee}_{RR}$ is just $\mc{M}^{\vee}$ with the $\mc{A}$-action twisted by the monoidal autoequivalence $(-)^{RR}$ on $\mc{A}$. But \cite[Section 4]{yuji} shows that this monoidal autoequivalence is conjugation by the invertible element $\mb{D}1_{\mc{A}}$ of $\mc{A}$. Calling this element $x$, the map \begin{align*}\mc{M}^{\vee} &\to \mc{M}^{\vee}_{RR}\\ m &\mapsto m*x\end{align*}
is an $\mc{A}$-module equivalence. Note that the above references use $!$-pullback convolution rather than $*$-pullback convolution which we use; however the results are the same (though the element $1_\mc{A}$ is different).

Note also that if we have an action of a monoidal category $\mc{B}$ on the left of $\mc{M}^{\vee}$, this module equivalence preserves the $\mc{B}$-action.

\begin{thm}\label{thm:hom}
    There is a natural identification
    \[\Hom_{\mc{A}}\left(\IndCoh_{\Lambda_{12}}(X\times_YU),\IndCoh_{\Lambda_{23}}(X\times_YV)\right)\simeq \IndCoh_{\Lambda_{12}*\Lambda_{23}}(U\times_YV)\]
    (assuming $\Lambda_{12}, \Lambda_{23}$ are such that these are indeed modules). This equivalence is also a module equivalence over any $\IndCoh_{\Lambda}(U \times_{Y'} U)$-action defined via a map $Y' \to Y$.
\end{thm}

We first compute the $\mc{A}$-module structure on the plain dual of $\IndCoh(X\times_YU)$.

\begin{lem}\label{lem:duals}
    Consider the natural self-duality,
    \[\IndCoh_{\Lambda}(X\times_YU)^{\vee}\isomfrom \IndCoh_{\Lambda}(X\times_YU),\]
    under which $\IndCoh_{\Lambda}(X\times_YU)$ obtains the structure of a right $\mc{A}$-module. Consider also $\IndCoh_{\Lambda}(U\times_YX)$ with its usual right $\mc{A}$-module structure (abusing notation for $\Lambda$).
    
    The swap equivalence,
    \[\sigma:\IndCoh_{\Lambda}(X\times_YU)\to \IndCoh_{\Lambda}(U\times_YX),\]
    intertwines these right $\mc{A}$-module structures.
\end{lem}

\begin{proof}[Proof of \cref{thm:hom}]
By the above, we have an identification of the $\mc{A}$-dual of $\IndCoh_{\Lambda_{12}}(X\times_YU)$ with its plain dual (and this preserves any $\IndCoh_{\Lambda}(U \times_{Y'} U)$-action). This is identified with $\IndCoh_{\Lambda_{12}}(U\times_YX)$ by \cref{lem:duals}. It follows that
\begin{align*}
\Hom_{\mc{A}}\left(\IndCoh_{\Lambda_{12}}(X\times_YU),\IndCoh_{\Lambda_{23}}(X\times_YV)\right) &\simeq \IndCoh_{\Lambda_{12}}(U\times_YX)\otimes_{\mc{A}}\IndCoh_{\Lambda_{23}}(X\times_YV)\\
&\simeq \IndCoh_{\Lambda_{12}*\Lambda_{23}}(U\times_YV).
\end{align*}
\end{proof}

\printbibliography
\end{document}